\pdfoutput=1
\documentclass[11pt]{article}
\usepackage{authblk}
\usepackage[toc,page]{appendix}

\usepackage{color}
\usepackage{helvet}         
\usepackage{courier}        
\usepackage{type1cm}        
\usepackage{framed}
\usepackage{tikz}
\usepackage{makeidx}         
\usepackage{graphicx}        
\usepackage{multicol}        
\usepackage[bottom]{footmisc}
\usepackage{amsmath}
\usepackage{amssymb}
\usepackage{bbold}
\usepackage{amsthm}
\usepackage{subcaption}
\usepackage{sidecap}
\usepackage{floatrow}
\usepackage{pdflscape}

\usepackage{comment}
\usepackage[font=small]{caption}
\usepackage[top=2cm, bottom=2cm, left=2cm, right=2cm]{geometry}
\newtheorem{theorem}{Theorem}
\newtheorem{corollary}[theorem]{Corollary}

\newtheorem{lemma}[theorem]{Lemma}

\newtheorem{proposition}[theorem]{Proposition}

\newtheorem{remark}[theorem]{Remark}

\numberwithin{theorem}{section}
\numberwithin{figure}{section}
\numberwithin{equation}{section}

\DeclareMathOperator{\dist}{dist}
\DeclareMathOperator{\SLE}{SLE}
\DeclareMathOperator{\hSLE}{hSLE}

\DeclareMathOperator{\diam}{diam}

\DeclareMathOperator{\hF}{{}_2F_1}

\DeclareMathOperator{\ust}{UST}
\DeclareMathOperator{\Leb}{Leb}

\begin{document}

\title{Hypergeometric SLE with $\kappa=8$: \\Convergence of UST and LERW in Topological Rectangles}

\author{Yong Han\thanks{Funded by National Science Foundation of China (12131016, 12201419)}}
\affil{Shenzhen University, China. Email: hanyong@szu.edu.cn}

\author{Mingchang Liu}
\affil{Tsinghua University, China. Email: liumc\_prob@163.com}

\author{Hao Wu\thanks{Funded by Beijing Natural Science Foundation (JQ20001).}}
\affil{Tsinghua University, China. Email: hao.wu.proba@gmail.com}

\date{\today}

%
%
\maketitle
\vspace{-1cm}
\begin{center}
\begin{minipage}{0.9\textwidth}
\abstract{We consider uniform spanning tree (UST) in topological rectangles with alternating boundary conditions. The Peano curves associated to the UST converge weakly to hypergeometric SLE$_8$, denoted by $\hSLE_8$. From the convergence result, we obtain the continuity and reversibility of $\hSLE_8$ as well as an interesting connection between $\SLE_8$ and $\hSLE_8$. The loop-erased random walk (LERW) branch in the UST converges weakly to $\SLE_2(-1, -1; -1, -1)$. We also obtain the limiting joint distribution of the two end points of the LERW branch. 
}
\end{minipage}
\end{center}

\tableofcontents

\newcommand{\eps}{\epsilon}
\newcommand{\ov}{\overline}
\newcommand{\U}{\mathbb{U}}
\newcommand{\T}{\mathbb{T}}
\newcommand{\HH}{\mathbb{H}}
\newcommand{\LA}{\mathcal{A}}
\newcommand{\LB}{\mathcal{B}}
\newcommand{\LC}{\mathcal{C}}
\newcommand{\LD}{\mathcal{D}}
\newcommand{\LF}{\mathcal{F}}
\newcommand{\LK}{\mathcal{K}}
\newcommand{\LE}{\mathcal{E}}
\newcommand{\LG}{\mathcal{G}}
\newcommand{\LL}{\mathcal{L}}
\newcommand{\LM}{\mathcal{M}}
\newcommand{\LQ}{\mathcal{Q}}
\newcommand{\LP}{\mathcal{P}}
\newcommand{\LR}{\mathcal{R}}
\newcommand{\LT}{\mathcal{T}}
\newcommand{\LS}{\mathcal{S}}
\newcommand{\LU}{\mathcal{U}}
\newcommand{\LV}{\mathcal{V}}
\newcommand{\LX}{\mathcal{X}}
\newcommand{\PartF}{\mathcal{Z}}
\newcommand{\LH}{\mathcal{H}}
\newcommand{\R}{\mathbb{R}}
\newcommand{\C}{\mathbb{C}}
\newcommand{\N}{\mathbb{N}}
\newcommand{\Z}{\mathbb{Z}}
\newcommand{\E}{\mathbb{E}}
\newcommand{\PP}{\mathbb{P}}
\newcommand{\QQ}{\mathbb{Q}}
\newcommand{\A}{\mathbb{A}}
\newcommand{\one}{\mathbb{1}}
\newcommand{\bn}{\mathbf{n}}
\newcommand{\MR}{MR}
\newcommand{\cond}{\,|\,}
\newcommand{\la}{\langle}
\newcommand{\ra}{\rangle}
\newcommand{\tree}{\Upsilon}
\newcommand{\prob}{\mathbb{P}}
\renewcommand{\Im}{\mathrm{Im}}
\renewcommand{\Re}{\mathrm{Re}}
\newcommand{\ii}{\mathfrak{i}}

\section{Introduction}
\label{sec::intro}
In \cite{SchrammScalinglimitsLERWUST}, O. Schramm introduced a random process---Schramm Loewner Evolution (SLE)---as a candidate for the scaling limit of interfaces in two-dimensional critical lattice models. The setup is as follows.  We consider a bounded simply connected domain $\Omega\subsetneq\C$ such that $\partial\Omega$ is locally connected. 
Let $\phi$ be any conformal map from the unit disk $\U$ onto $\Omega$. As $\partial\Omega$ is locally connected, the conformal map $\phi$ can be extended continuously to $\overline{\U}$ and $\partial\Omega$ is a curve (see~\cite[Theorem~2.1]{Pommerenke}).
A Dobrushin domain $(\Omega; x, y)$ is a bounded simply connected domain $\Omega$ with two boundary points $x, y$ such that $\partial\Omega$ is locally connected. 
We denote by $(xy)$ the boundary arc going from $x$ to $y$ in counterclockwise order. 
Suppose that $(\Omega_{\delta}; x_{\delta}, y_{\delta})$ is an approximation of $(\Omega; x, y)$ on $\delta\Z^2$. 
Consider a critical lattice model on $\Omega_{\delta}$ with 
 Dobrushin boundary conditions, for instance, Ising model, percolation, uniform spanning tree etc. In these examples, there is an interface in $\Omega_{\delta}$ connecting $x_{\delta}$ to $y_{\delta}$. It is conjectured that such interface has a conformally invariant scaling limit which can be identified by $\SLE_{\kappa}$ where the parameter $\kappa$ varies for different models. Since the introduction of $\SLE$, there are several models for which the conjecture is proved: the Peano curve in uniform spanning tree converges to $\SLE_8$ (with further assumption that $\partial\Omega$ is $C^1$ and simple, see Theorem~\ref{thm::ust_Dobrushin}) and the loop-erased random walk converges to $\SLE_2$ \cite{LawlerSchrammWernerLERWUST}, the interface in percolation converges to $\SLE_6$ \cite{SmirnovPercolationConformalInvariance}, the level line of discrete Gaussian free field converges to $\SLE_4$ \cite{SchrammSheffieldDiscreteGFF}, the interface in Ising model converges to $\SLE_3$ and the interface in FK-Ising model converegs to $\SLE_{16/3}$ \cite{CDCHKSConvergenceIsingSLE}. 
The proof requires two inputs: 1st. tightness of interfaces; 2nd. discrete martingale observable. With the tightness, there are always subsequential limits of interfaces; and then one uses the observable to show that all the subsequential limits are the same  and also identify the unique limit.

Dobrushin boundary conditions are the simplest boundary condition. It is then natural to consider critical lattice models with more complicated boundary conditions. One possibility is to consider the model in polygons with alternating boundary conditions. 
In general, a (topological) polygon $(\Omega; x_1, \ldots, x_p)$ is a bounded simply connected domain $\Omega$ with distinct boundary points $x_1, \ldots, x_p$ such that $\partial\Omega$ is locally connected 
and $x_1, \ldots, x_p$ lies on $\partial\Omega$ in counterclockwise order. 
In this article, we focus on topological rectangles $(\Omega; a, b, c, d)$, i.e. a polygon with four marked points on the boundary. We call it a quad. Suppose  that $(\Omega_{\delta}; a_{\delta}, b_{\delta}, c_{\delta}, d_{\delta})$ is an approximation of $(\Omega; a, b, c, d)$ on $\delta\Z^2$. Consider a critical lattice model on $\Omega_{\delta}$ with alternating boundary conditions, it turns out that the scaling limit of interfaces in this case becomes hypergeometric SLE, denoted by $\hSLE$, which is a variant of SLE process  whose drift term involves a hypergeometric function. 
Note that variant of SLE whose driving function involves hypergeometric functions are considered earlier in~\cite{ZhanReversibilityMore}, ~\cite{QianConformalRestrictionTrichordal} and~\cite{WuHyperSLE}, see Appendix~\ref{appendix_hSLE} for their connection. We stick to the definition in~\cite{WuHyperSLE} where the author relates $\hSLE$ to critical lattice model in quad.
For instance, the interface in critical Ising model in quad converges to $\hSLE_3$, see \cite{IzyurovObservableFree} and \cite{WuHyperSLE}; the interface in critical FK-Ising model in quad converges to $\hSLE_{16/3}$, see \cite{KemppainenSmirnovFKIsingHyperSLE} and \cite{BeffaraPeltolaWuUniqueness}. In this article, we focus on uniform spanning tree in quad with alternating boundary conditions. Not surprisingly, the associated Peano curve converges to $\hSLE_8$ process  (see precise setup in Theorem~\ref{thm::cvg_triple}). 
The third author of this article studies $\hSLE$ process in~\cite{WuHyperSLE} with $\kappa\in (0,8)$.
She only treats the process with $\kappa\in (0,8)$ due to technical difficulty: the analysis there does not apply to $\kappa= 8$. The first goal of this article is to address $\hSLE$ process with $\kappa=8$. 

\subsection{Hypergeometric SLE with $\kappa=8$}
Hypergeometric SLE is a two-parameter family of random curves in quad. The two parameters are $\kappa>0$ and $\nu\in\R$, and we denote it by $\hSLE_{\kappa}(\nu)$.  An $\hSLE_{\kappa}(\nu)$ in quad $(\Omega; a, b, c, d)$ is a process in $\Omega$ from $a$ to $d$ with marked points $(b,c)$. When describing scaling limit of interfaces in quad, the parameter $\kappa$ corresponds to different underlying model and the parameter $\nu$ corresponds to different boundary conditions on the boundary arc $(bc)$. For instance, we consider critical Ising model in $(\Omega_{\delta}; a_{\delta}, b_{\delta}, c_{\delta}, d_{\delta})$ with boundary conditions $\oplus$ on $(a_{\delta}b_{\delta})\cup(c_{\delta}d_{\delta})$ and $\ominus$ on $(d_{\delta}a_{\delta})$ and $\xi\in\{\ominus, \textrm{free}\}$ on $(b_{\delta}c_{\delta})$ and  denote by $\eta_{\delta}$ the interface from $a_{\delta}$ to $d_{\delta}$. Then the scaling limit of $\eta_{\delta}$ has the law of $\hSLE_3(\nu=0)$ if $\xi=\ominus$ and has the law of $\hSLE_3(\nu=-3/2)$ if $\xi=\textrm{free}$. See~\cite[Proposition~1.6]{WuHyperSLE} or~\cite{IzyurovObservableFree}. 
$\hSLE_{\kappa}(\nu)$ is also related to $\SLE_{\kappa}(\rho)$ process: the time-reversal of $\SLE_{\kappa}(\rho)$ is an $\hSLE_{\kappa}(\rho-2)$ process for $\rho\ge \kappa/2-2$. See~\cite[Theorem~1.1]{WuHyperSLE}. 
The continuity and reversibility of $\hSLE_{\kappa}(\nu)$ are addressed in \cite{WuHyperSLE} for $\kappa\in (0,8)$. Our first main result is about the continuity and reversibility of $\hSLE_{\kappa}(\nu)$ with $\kappa=8$.

\begin{theorem}\label{thm::hsle_continuity_reversibility}
Fix $\nu\ge 0$ and $x_1<x_2<x_3<x_4$. The process $\hSLE_{8}(\nu)$ in the upper half-plane $\HH$ from $x_1$ to $x_4$ with marked points $(x_2, x_3)$ is almost surely generated by a continuous curve denoted by $\eta$. Furthermore, the process $\eta$ enjoys reversibility: the time-reversal of $\eta$ has the law of $\hSLE_8(\nu)$ in $\HH$ from $x_4$ to $x_1$ with marked points $(x_3, x_2)$.  
\end{theorem}
The process $\hSLE_8(\nu)$ in general quad is defined via conformal image: 
For a general quad $(\Omega; a, b, c, d)$, let $\phi$ be any conformal map from $\Omega$ onto $\HH$ such that $\phi(a)<\phi(b)<\phi(c)<\phi(d)$. We define $\hSLE_8(\nu)$ in $\Omega$ from $a$ to $d$ with marked points $(b, c)$ to be $\phi^{-1}(\eta)$ where $\eta$ is an $\hSLE_8(\nu)$ in $\HH$ from $\phi(a)$ to $\phi(d)$ with marked points $(\phi(b), \phi(c))$.

In Section~\ref{sec::preSLE}, we will give preliminaries on SLE; and in Section~\ref{sec::hypersle}, we will give definition of $\hSLE_{\kappa}(\nu)$.  In fact, we will address $\hSLE_{\kappa}(\nu)$ for $\kappa\ge 8$ in Section~\ref{sec::hypersle}. As the case of $\kappa>8$ is less relevant, we omit the corresponding conclusion in the introduction. 
In literature, the continuity of $\SLE_{\kappa}$ is first proved in~\cite{RohdeSchrammSLEBasicProperty} for $\kappa\neq 8$ and then proved in~\cite{LawlerSchrammWernerLERWUST} for $\kappa=8$, because the technical analysis in the continuum~\cite{RohdeSchrammSLEBasicProperty} does not apply for $\kappa=8$ whose continuity is proved using convergence of Peano curve in UST~\cite{LawlerSchrammWernerLERWUST}. We encounter a similar situation for $\hSLE$ as well.
The continuity of $\hSLE_{\kappa}(\nu)$ with $\kappa\in (0,8)$ is proved in~\cite{WuHyperSLE} using continuity of $\SLE_{\kappa}$ and $\SLE_{\kappa}(\rho)$ and analysis in the continuum. However, such analysis does not apply to the case with $\kappa=8$.  The full continuity and reversibility results of $\hSLE_8(\nu)$ process are proved using the convergence of Peano curves in UST.  Note that, in the definition of $\hSLE_{8}(\nu)$, the parameter $\nu$ may take values in $\R$, and the continuity and reversibility are believed to hold for all $\nu>-2$, but we are only able to prove them for $\nu\ge 0$. Only the case when $\nu=0$ is relevant for UST in quad.
We denote $\hSLE_8(\nu)$ by $\hSLE_8$ when $\nu=0$, and the rest of the introduction will focus on $\hSLE_8$. 

First of all, we may find $\hSLE_8$ inside $\SLE_8$. 
%
%

\begin{proposition}\label{prop::fromSLEtohSLE}
Fix $x<y$ and suppose $\eta\sim\SLE_8$ in $\HH$ from $x$ to $\infty$. Let $T_y$ be the first time that $\eta$ swallows $y$, and define $\gamma:=\partial(\eta[0,T_y])\cap\overline{\HH}$ (here we view $\eta[0,T_y]$ as a compact set) and we view $\gamma$ as a continuous simple curve starting from $y$ and terminating at some point in $(-\infty, x)$. Let $\tau$ be any stopping time for $\gamma$ before the terminating time. Then the conditional law of $(\eta(t), 0\le  t\le T_y)$ given $\gamma[0,\tau]$ is $\hSLE_8$ in $\HH\setminus\gamma[0,\tau]$ from $x$ to $y^-$ with marked points $(y^+, \infty)$ conditional that its first hitting point on $\gamma[0,\tau]$ is given by $\gamma(\tau)$. 
\end{proposition}

The proof for Proposition~\ref{prop::fromSLEtohSLE} depends on the domain Markov property 
 for UST which will be given in Section~\ref{sec::lerw}. Moreover, the calculation in Section~\ref{sec::lerw} also gives the following consequence Proposition~\ref{prop::hSLE_right_proba}. Recall that given a quad $(\Omega;a,b,c,d)$, there exists a unique $K>0$ and a unique conformal map from $\Omega$ onto the rectangle $[0,1]\times[0,\ii K]$ which sends $a,b,c,d\in\partial\Omega$ to the four corners $0,1,1+\ii K,\ii K$ respectively. We call $K$ the conformal modulus of the quad $(\Omega;a,b,c,d)$.
\begin{proposition}\label{prop::hSLE_right_proba}
Fix a quad $(\Omega; a, b, c, d)$. 
Let $K>0$ be the conformal modulus of the quad $(\Omega; a, b, c, d)$, and let $f$ be the conformal map from $\Omega$ onto the rectangle $(0,1)\times (0, \ii K)$ which sends $(a, b, c, d)$ to $(0, 1, 1+\ii K, \ii K)$.
Suppose $\eta\sim\hSLE_8$ in $\Omega$ from $a$ to $d$ with marked points $(b,c)$. 
Then we have 
\begin{equation}\label{eqn::hSLE_right_proba}
\PP\left[z\not\in\eta\right]=\Re f(z), \quad \forall z\in \Omega. 
\end{equation}
\end{proposition}

We remark that  we do not derive the probability in  \eqref{eqn::hSLE_right_proba} from standard It\^{o}'s calculus. The proof bases on the analysis from UST: we will first prove that~\eqref{eqn::hSLE_right_proba} holds in the discrete setting and then obtain  \eqref{eqn::hSLE_right_proba} by  proving the convergence of the corresponding discrete harmonic  functions with mixed boundary conditions.

\subsection{Uniform spanning tree (UST)}
\label{subsec::intro_ust}
Let us come back to uniform spanning tree. 
We start with definitions and notations. 
The square lattice $\Z^2$ is the graph with vertex set $V(\Z^2):=\{(m, n): m, n\in \Z\}$ and edge set $E(\Z^2)$ given by edges between nearest neighbors. This is our primal lattice. Its dual lattice is denoted by $(\Z^2)^*$. The medial lattice $(\Z^2)^{\diamond}$ is the graph with centers of edges of $\Z^2$ as vertex set and edges connecting nearest vertices. For a finite subgraph $G=(V(G), E(G))\subset\Z^2$, we denote by $\partial G$ the inner boundary of $G$: $\partial G=\{x\in V(G): \exists y\not\in V(G) \text{ such that }\{x, y\}\in E(\Z^2)\}$. 
In this article, when we add the subscript  or superscript $\delta$, we mean scaling subgraphs of the lattices $\Z^2, (\Z^2)^*, (\Z^2)^\diamond$ by $\delta$.

\paragraph{Uniform spanning tree.}
Suppose that $G=(V(G),E(G))$ is a finite connected graph. A forest is a subgraph of $G$ that has no cycles. A tree is a connected forest. A subgraph of $G$ is spanning if it covers $V(G)$. A uniform spanning tree on $G$ is a probability measure on the set of all spanning trees of $G$ in which every tree is chosen with equal probability.  Given a disjoint sequence $(\alpha_k: 1\leq k\leq N)$  of trees of $G$,  a spanning tree with $(\alpha_k: 1\leq k\leq N)$ wired is a spanning tree $T$ such that $\alpha_k\subset T$ for $1\leq k\leq N$.
A uniform spanning tree with  $(\alpha_k: 1\leq k\leq N)$ wired is a probability measure on the set of all spanning trees of $G$ with  $(\alpha_k: 1\leq k\leq N)$ wired  under which every tree is chosen with equal probability.
\paragraph{Space of curves.}
A path is defined by a continuous map from $[0, 1]$ to $\C$. Let $\mathcal{C}$ be the space of unparameterized paths in $\C$.  Define the metric on $\mathcal{C}$ as follows:
\begin{align}\label{eqn::curves_metric}
d(\gamma_1, \gamma_2):=\inf\sup_{t\in[0,1]}\left|\hat{\gamma}_1(t)-\hat{\gamma}_2(t)\right|,
\end{align}
where the infimum is taken over all the choices of  parameterizations  $\hat{\gamma}_1$ and $\hat{\gamma}_2$ of $\gamma_1$ and $\gamma_2$. The metric space $(\mathcal{C}, d)$ is complete and separable, see~\cite{KemppainenSmirnovRandomCurves}.
Let $\mathcal{P}$ be a family of probability measures on $\mathcal{C}$. We say $\mathcal{P}$ is tight if for any $\epsilon>0$, there exists a compact set $K_{\eps}$ such that $\PP[K_{\eps}]\geq 1-\epsilon$ for any $\PP\in \mathcal{P}$. We say $\mathcal{P}$ is relatively compact
if every sequence of elements in $\mathcal{P}$ has a weakly convergent subsequence. As the metric space is complete and separable, relative compactness is equivalent to tightness. 

\paragraph{Convergence of discrete polygons.}
A sequence of discrete polygons $(\Omega_{\delta}; x_1^{\delta}, \ldots, x_p^{\delta})$ on $\delta\Z^2$ is said to converge to a polygon $(\Omega; x_1, \ldots, x_p)$ in the Carath\'{e}odory sense if there exist conformal maps $\phi_{\delta}$ from $\U$ onto $\Omega_{\delta}$ and conformal map $\phi$ from $\U$ onto $\Omega$ such that $\phi_{\delta}\to \phi$ as $\delta\to 0$ uniformly on compact subsets of $\U$ and $\phi_{\delta}^{-1}(x_j^{\delta})\to \phi^{-1}(x_j)$ for $1\le j\le p$.

We will encounter another type of convergence of polygons: 
a sequence of discrete polygons $(\Omega_{\delta}; x_1^{\delta}, \ldots, x_p^{\delta})$ on $\delta\Z^2$ converge to a polgyon $(\Omega; x_1, \ldots, x_p)$ in the following sense: there exists a constant $C>0$, such that 
\begin{equation}\label{eqn::topology}
d((x_i^{\delta}x_{i+1}^{\delta}),(x_ix_{i+1}))\le C\delta,\quad\text{for }i\in\{1, \ldots, p\},
\end{equation}
where $d$ is the metric~\eqref{eqn::curves_metric} and we use the convention that $x_{p+1}=x_1$. 
Note that such convergence implies the convergence in the Carath\'{e}odory sense. 
These two types of convergence apply to polygons in the medial lattice in a similar way.
\medbreak
Suppose a sequence of medial quads $(\Omega_{\delta}^{\diamond}; a_{\delta}^{\diamond}, b_{\delta}^{\diamond}, c_{\delta}^{\diamond}, d_{\delta}^{\diamond})$ on $\delta(\Z^2)^{\diamond}$ converges to a quad $(\Omega; a, b, c, d)$ in the sense~\eqref{eqn::topology}. See details in Section~\ref{sec::ust}. 
Let $\Omega_{\delta}\subset\delta\Z^2$ be the corresponding graph on the primal lattice.
We consider uniform spanning tree (UST) on $\Omega_{\delta}$ with alternating boundary conditions: $(a_{\delta}b_{\delta})$ is wired and $(c_{\delta}d_{\delta})$ is wired (but $(a_{\delta}b_{\delta})$ and $(c_{\delta}d_{\delta})$ are not wired together). Let $\LT_{\delta}$ be the UST on $\Omega_{\delta}$ with such alternating boundary conditions. 
Then there exists a triple of curves $(\eta_{\delta}^L; \gamma_{\delta}^M; \eta_{\delta}^R)$ such that 
$\eta_{\delta}^L$ runs along the tree $\LT_{\delta}$ from $a_{\delta}$ to $d_{\delta}$, 
and $\gamma_{\delta}^M$ is the unique simple path in $\LT_{\delta}$ connecting $(a_{\delta}b_{\delta})$ to $(c_{\delta}d_{\delta})$, 
and $\eta_{\delta}^R$ runs along the tree $\LT_{\delta}$ from $b_{\delta}$ to $c_{\delta}$, 
see Figure~\ref{fig::quadPeano} and see detail in Section~\ref{subsec::ust_quads}. We have the following convergence of the triple 
$(\eta_{\delta}^L; \gamma_{\delta}^M; \eta_{\delta}^R)$. 

\begin{figure}[ht!]
\begin{center}
\includegraphics[width=0.5\textwidth]{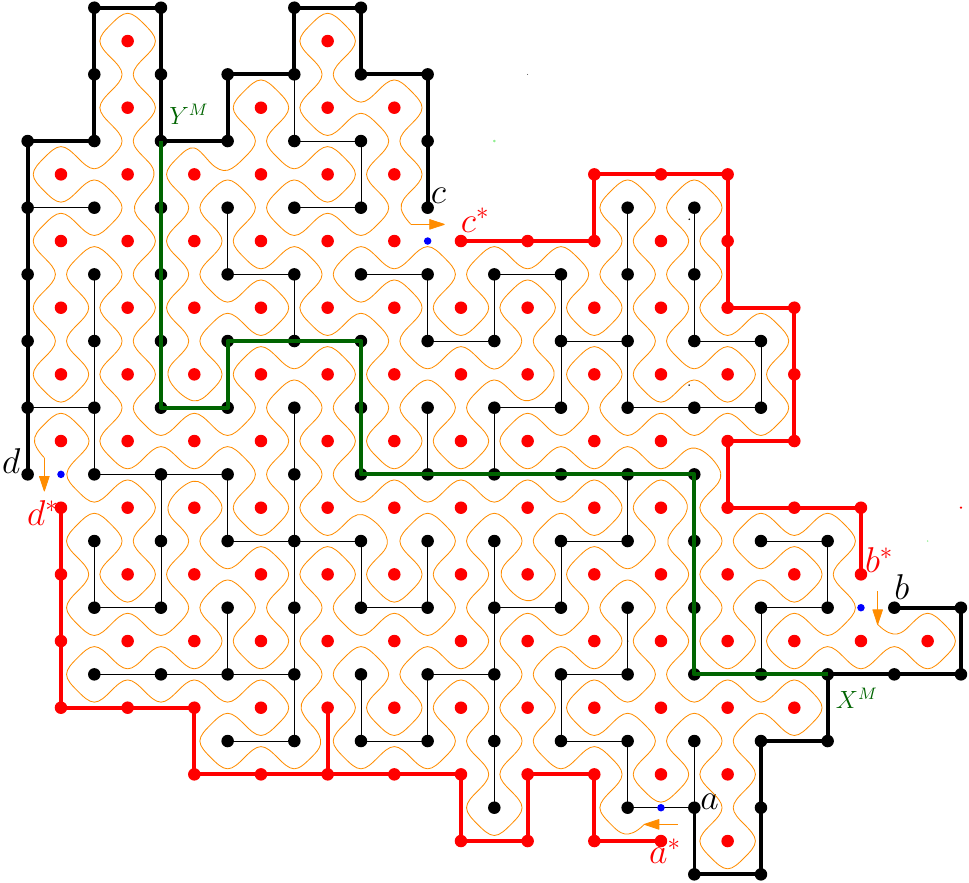}
\end{center}
\caption{\label{fig::quadPeano} The solid edges in black are wired boundary arcs $(ab)$ and $(cd)$, the solid edges in red are dual-wired boundary arcs $(b^*c^*)$ and $(d^*a^*)$. The thin edges are in the tree $\LT$. The solid edges in green are in the branch $\gamma^M$ in $\LT$. This branch intersects $(ab)$ at $X^M$ and intersects $(cd)$ at $Y^M$.  The two curves in orange are the Peano curves $\eta^L$ (from $a$ to $d$) and $\eta^R$ (from $b$ to $c$). }
\end{figure}

\begin{theorem}\label{thm::cvg_triple}
Fix a quad $(\Omega; a, b, c, d)$ such that $\partial\Omega$ is $C^{1}$ and simple. Suppose that a sequence of medial quads $(\Omega_{\delta}^{\diamond}; a_{\delta}^{\diamond}, b_{\delta}^{\diamond}, c_{\delta}^{\diamond}, d_{\delta}^{\diamond})$ converges to $(\Omega; a, b, c, d)$ in the sense~\eqref{eqn::topology}. 
Consider the UST in  $(\Omega_{\delta}; a_{\delta}, b_{\delta}, c_{\delta}, d_{\delta})$ with alternating boundary conditions and consider the triple of curves $(\eta_{\delta}^L; \gamma_{\delta}^M; \eta_{\delta}^R)$ as described above. Then the triple $(\eta_{\delta}^L; \gamma_{\delta}^M; \eta_{\delta}^R)$ converges weakly to a triple of continuous curves $(\eta^L; \gamma^M; \eta^R)$ in metric~\eqref{eqn::curves_metric}. The law of  $(\eta^L; \gamma^M; \eta^R)$ is characterized by the following properties: the marginal law of $\eta^L$ is $\hSLE_8$ in $\Omega$ from $a$ to $d$ with marked points $(b, c)$; given $\eta^L$, the conditional law of $\eta^R$ is $\SLE_8$ in $\Omega\setminus \eta^L$ from $b$ to $c$; and $\gamma^M=\eta^L\cap \eta^R$. Furthermore, given $\gamma^M$, denote by $\Omega^L$ and $\Omega^R$ the two connected components of $\Omega\setminus\gamma^M$ such that $\Omega^L$ has $a, d$ on the boundary and $\Omega^R$ has $b, c$ on the boundary, then the conditional law of $\eta^L$ is $\SLE_8$ in $\Omega^L$ from $a$ to $d$ and the conditional law of $\eta^R$ is $\SLE_8$ in $\Omega^R$ from $b$ to $c$, and $\eta^L$ and $\eta^R$ are conditionally independent given $\gamma^M$. 
\end{theorem}

In Section~\ref{sec::ust},  we prove the convergence of the Peano curve and complete the proof of Theorem~\ref{thm::cvg_triple}. 
The work~\cite{DubedatEulerIntegralsCommutingSLEs} predicts the limiting distribution of Peano curve in UST in general polygon with a special alternating boundary condition. However, the proof there lacks the complete construction of discrete observable and the analysis of the corresponding limit of the observable.
Our proof in Section~\ref{sec::ust} focuses on the Peano curve in UST in quad and follows the standard strategy: we first derive the tightness of the Peano curves and construct a discrete martingale observable, and then identify the subsequential limits through the observable. This part is a generalization of the work in \cite{SchrammScalinglimitsLERWUST} and \cite{LawlerSchrammWernerLERWUST}, and we proceed following a simplification advocated by Smirnov\footnote{H.W. learned this observable from a master course delivered by Smirnov in 2015, but we are not able to identify a published reference.} who provided a different observable from the one in \cite{LawlerSchrammWernerLERWUST}. Smirnov's observable is more suitable for the setup in quad. We remark that the observable here does not have an explicit closed form and the identification of the driving process is a non-trivial step, see Lemma~\ref{lem::drivingfunction}.
As a byproduct, we obtain the continuity and reversibility of $\hSLE_{8}$ and complete the proof of Theorem~\ref{thm::hsle_continuity_reversibility}. 

We emphasize that the simple and $C^1$ regularity on $\partial\Omega$ in the assumption of Theorem~\ref{thm::cvg_triple} is crucial in the proof of the tightness. We derive the tightness following the argument in~\cite[Theorem~11.1]{SchrammScalinglimitsLERWUST} where the simple and $C^1$ regularity is assumed. See Appendix~\ref{app::tightness}.

\subsection{Loop-erased random walk (LERW)}

From Theorem~\ref{thm::cvg_triple}, we see that the Peano curve $\eta^L_{\delta}$ converges weakly to $\hSLE_8$ and the limit of $\gamma^M_{\delta}$ is part of the boundary of $\hSLE_8$. In the following theorem, we provide an explicit characterization of the limiting distribution of $\gamma^M_{\delta}$. 

\begin{theorem}\label{thm::cvg_lerw}
Fix a quad $(\Omega; a, b, c, d)$ such that $\partial\Omega$ is $C^1$ and simple. 
Let $K>0$ be the conformal modulus of the quad $(\Omega; a, b, c, d)$, and let $f$ be the conformal map from $\Omega$ onto the rectangle $(0,1)\times (0, \ii K)$ which sends $(a, b, c, d)$ to $(0, 1, 1+\ii K, \ii K)$. Assume the same setup as in Theorem~\ref{thm::cvg_triple}. Then the law of $\gamma_{\delta}^M$ converges  weakly to a continuous curve $\gamma^M$ whose law is characterized by the following properties. Denote by $X^M=\gamma^M\cap (ab)$ and by $Y^M=\gamma^M\cap (cd)$. 
\begin{itemize}
\item[(1)] The law of the point $f(X^M)$ is uniform in $(0,1)$. 
\item[(2)] Given $X^M$, the conditional law of $\gamma^M$ is an $\SLE_2(-1,-1;-1,-1)$ in $\Omega$ from $X^M$ to $(cd)$ with force points $(d, a; b, c)$ stopped at the first hitting time of $(cd)$. 
\end{itemize}
Furthermore, denote by $x^M=f(X^M)$ and $y^M=\Re f(Y^M)$, the joint density of $(x^M, y^M)$ is given by   
\begin{align}\label{eqn::jointdensity}
\rho_K(x, y)
=\frac{\pi}{4K}\sum_{n\in\Z}\left(\frac{1}{\cosh^2\left(\frac{\pi}{2K}\left(x-y-2n\right)\right)}+\frac{1}{\cosh^2\left(\frac{\pi}{2K}\left(x+y-2n\right)\right)}\right), \quad \forall x, y\in (0,1).  
\end{align}
\end{theorem}


In Section~\ref{sec::lerw}, we work on the LERW branch and complete the proof of Theorem~\ref{thm::cvg_lerw}. The section has three parts.
\begin{itemize}
\item We first derive the joint distribution of the pair $(X^M, Y^M)$ in Section~\ref{subsec::pairrandompoints}. We derive the formula~\eqref{eqn::jointdensity} through discrete observable. The analysis on discrete harmonic function from~\cite{ChelkakWanMassiveLERW} plays an important role. 
\item We then derive the conditional law of $\gamma^M$ given $X^M$ in Section~\ref{subsec::cvg_lerw}. 
In fact, this part of the conclusion is already solved in~\cite{ZhanLERW} in an implicit form with more generality.
The derivation follows the standard strategy: showing the tightness and constructing discrete martingale observable.  We provide details of the proof in a self-contained way for our particular setting and derive the explicit answer in Section~\ref{subsec::cvg_lerw}. 
\item In Section~\ref{subsec::Dobrushintoquad}, we use the domain Markov property
of the UST and complete the proof of Proposition~\ref{prop::fromSLEtohSLE}. We also remark that the proof for the law of $\gamma^M$ in Theorem~\ref{thm::cvg_lerw} also provides an alternative proof for the duality result of $\SLE_8$, see Corollary~\ref{cor::sle8duality}. Such duality relation was previously proved in~\cite{ZhanDuality} and~\cite{MillerSheffieldIG1} in the continuous setting for general $\kappa$ which is substantially more involved. Our proof in Section~\ref{subsec::Dobrushintoquad} is specific for $\kappa=8$ because we use the convergence of UST and LERW.  
\end{itemize}
\begin{corollary}\label{cor::sle8duality}
Fix $x<y$ and suppose $\eta\sim\SLE_8$ in $\HH$ from $x$ to $\infty$. Let $T_y$ be the first time that $\eta$ swallows $y$, and define $\gamma:=\partial(\eta[0,T_y])\cap{\overline{\HH}}$ (here we view $\eta[0,T_y]$ as a compact set) and we view $\gamma$ as a continuous simple curve starting from $y$ and terminating at some point in $(-\infty, x)$. Then the law of $\gamma$ is the same as $\SLE_2(-1, -1; -1, -1)$ in $\HH$ from $y$ to $(-\infty, x)$ with force points $(x, y^-; y^+, \infty)$. 
\end{corollary}


\smallbreak
\paragraph{Acknowledgment.} We thank Titus Lupu, Eveliina Peltola, Yijun Wan, and Dapeng Zhan for helpful discussion on UST and LERW. 
We also thank referees for careful comments which improved the presentation and clarified the proof.


\section{Preliminaries on SLE}
\label{sec::preSLE}
\paragraph{Notations.}
For $z\in \C$ and $r>0$, we denote by $B(z,r)$ the ball with center $z$ and radius $r$ and by $B(z,r)^{c}$ the complement of $B(z,r)$ in $\C$. In particular, we denote $B(0,1)$ by $\U$. 
\paragraph{Loewner chain.}
An $\HH$-hull is a compact subset $K$ of $\overline{\HH}$ such that $\HH\setminus K$ is simply connected. By Riemann's mapping theorem, there exists a unique conformal map $g_K$ from $\HH\setminus K$ onto $\HH$ with normalization $\lim_{z\to\infty}|g_K(z)-z|=0$, and we call $a(K):=\lim_{z\to\infty}z(g_K(z)-z)$ the half-plane capacity of $K$ seen from $\infty$. 
Loewner chain is a collection of $\HH$-hulls $(K_t, t\ge 0)$ associated to the family of conformal maps $(g_t, t\ge 0)$ which solves the following Loewner equation: for each $z\in\HH$,
\[\partial_t g_t(z)=\frac{2}{g_t(z)-W_t},\quad g_0(z)=z,\]
where $(W_t, t\ge 0)$ is a one-dimensional continuous function which we call the driving function. For $z\in\overline{\HH}$, the swallowing time of $z$ is defined to be $\sup\left\{t\ge 0: \min_{s\in[0,t]}|g_s(z)-W_s|>0\right\}$. Let $K_t$ be the closure of $\{z\in\HH: T_z\le t\}$. It turns out that $g_t$ is the unique conformal map from $\HH\setminus K_t$ onto $\HH$ with normalization $\lim_{z\to\infty}|g_t(z)-z|=0$. Since the half-plane capacity of $K_t$ is $\lim_{z\to\infty}z(g_t(z)-z)=2t$, we say that the process $(K_t, t\ge 0)$ is parameterized by the half-plane capacity. We say that $(K_t, t\ge 0)$ can be generated by continuous curve $(\eta(t), t\ge 0)$ if, for any $t$, the unbounded connected component of $\HH\setminus\eta[0,t]$ is the same as $\HH\setminus K_t$. 
\paragraph{Schramm Loewner evolution.}
Schramm Loewner evolution $\SLE_{\kappa}$ is the random Loewner chain driven by $W_t=\sqrt{\kappa}B_t$ where $\kappa> 0$ and $(B_t, t\ge 0)$ is one-dimension Brownian motion starting from 0. 
$\SLE_{\kappa}$ process is almost surely generated by continuous curve $\eta$. The continuity for $\kappa\neq 8$ is proved in \cite{RohdeSchrammSLEBasicProperty}, and the continuity for $\kappa=8$ is proved in \cite{LawlerSchrammWernerLERWUST}. Moreover, the curve $\eta$ is almost surely transient: $\lim_{t\to\infty}|\eta(t)|=\infty$. When $\kappa\in (0,4]$, the curve is simple; when $\kappa\in (4,8)$, the curve is self-touching; when $\kappa\ge 8$, the curve is space-filling. 

In the above, $\SLE_{\kappa}$ is in $\HH$ from $0$ to $\infty$, we may define it in any Dobrushin domain $(\Omega; x, y)$ via conformal image: 
let $\phi$ be any conformal map from $\Omega$ onto $\HH$ such that $\phi(x)=0$ and $\phi(y)=\infty$. We define $\SLE_{\kappa}$ in $\Omega$ from $x$ to $y$ to be $\phi^{-1}(\eta)$ where $\eta$ is an $\SLE_{\kappa}$ in $\HH$ from $0$ to $\infty$. When $\kappa\in (0,8]$, $\SLE_{\kappa}$ enjoys reversibility: suppose $\eta$ is an $\SLE_{\kappa}$ in $\Omega$ from $x$ to $y$, the time-reversal of $\eta$ has the same law as $\SLE_{\kappa}$ in $\Omega$ from $y$ to $x$, proved in \cite{ZhanReversibility}, \cite{MillerSheffieldIG2}, \cite{MillerSheffieldIG3}. When $\kappa>8$, the time-reversal of $\SLE_{\kappa}$ is nolonger $\SLE_{\kappa}$, it becomes $\SLE_{\kappa}(\kappa/2-4; \kappa/2-4)$, see\cite[Theorem~1.19]{MillerSheffieldIG4}. 
\paragraph{$\SLE_{\kappa}(\underline{\rho})$ process.}
$\SLE_{\kappa}(\underline{\rho})$ process is a variant of $\SLE_{\kappa}$ where one keeps track of multiple marked points. Suppose 
$\underline{y}^L=(y^{L, l}<\cdots<y^{L, 1}\le 0)$, $\underline{y}^R=(0\le y^{R, 1}<y^{R, 2}<\cdots<y^{R, r})$ and 
$\underline{\rho}^L=(\rho^{L, l}, \ldots, \rho^{L, 1})$, $\underline{\rho}^R=(\rho^{R, 1}, \ldots, \rho^{R, r})$ with $\rho^{L, i}, \rho^{R, i}\in\R$. An $\SLE_{\kappa}(\underline{\rho}^L;\underline{\rho}^R)$ process with force points $(\underline{y}^L; \underline{y}^R)$ is the Loewner chain driven by $W_t$ which is the solution to the following system of SDEs:
\begin{align*}
\begin{cases}
dW_t=\sqrt{\kappa}dB_t+\sum_{i=1}^l \frac{\rho^{L, i} dt}{W_t-V_t^{L, i}}+\sum_{i=1}^r \frac{\rho^{R, i} dt}{W_t-V_t^{R, i}}, \quad W_0=0; \\
dV_t^{L, i}=\frac{2dt}{V_t^{L, i}-W_t},\quad V_0^{L, i}=y^{L, i}, \quad\text{for }1\le i\le l; \\
dV_t^{R, i}=\frac{2dt}{V_t^{R, i}-W_t},\quad V_0^{R, i}=y^{R, i}, \quad\text{for }1\le i\le r; 
\end{cases}
\end{align*}
where $(B_t, t\ge 0)$ is one-dimensional Brownian motion starting from $0$. We define the continuation threshold of
$\SLE_{\kappa}(\underline{\rho}^L; \underline{\rho}^R)$ to be the infimum of the time $t$ for which 
\[\text{either }\sum_{i: V_t^{L, i}=W_t}\rho^{L, i}\le -2,\quad\text{or } \sum_{i: V_t^{R, i}=W_t}\rho^{R, i}\le -2.\]
$\SLE_{\kappa}(\underline{\rho}^L; \underline{\rho}^R)$ process is well-defined up to the continuation threshold and it is almost surely generated by continuous curve up to and including the continuation threshold, see~\cite{MillerSheffieldIG1}. 

The law of $\SLE_{\kappa}(\underline{\rho}^L; \underline{\rho}^R)$ is absolutely continuous with respect to $\SLE_{\kappa}$, and we will give the Radon-Nikodym derivative below, see also~\cite{SchrammWilsonSLECoordinatechanges}. To simplify the notation for the Radon-Nikodym derivative, we focus on $\SLE_{\kappa}(\underline{\rho})$ process when all force points are located to the same side of the process. Consider $\SLE_{\kappa}(\underline{\rho})$ with force points $\underline{y}$ where $\underline{\rho}=(\rho_1, \ldots, \rho_n)\in\R^n$ and $\underline{y}=(0\le y_1<\cdots<y_n)$. The law of $\SLE_{\kappa}(\underline{\rho})$ with force points $\underline{y}$ is absolutely continuous with respect to $\SLE_{\kappa}$ up to the first time that $y_1$ is swallowed, and the Radon-Nikodym derivative is $M_t/M_0$ where 
\begin{equation}\label{eqn::sle_kappa_rho_mart}
M_t=\prod_{1\le i\le n} \left(g_t'(y_i)^{\rho_i(\rho_i+4-\kappa)/(4\kappa)}(g_t(y_i)-W_t)^{\rho_i/\kappa}\right)\times \prod_{1\le i<j\le n}(g_t(y_j)-g_t(y_i))^{\rho_i\rho_j/(2\kappa)}.
\end{equation} 

$\SLE_{\kappa}(\underline{\rho}^L; \underline{\rho}^R)$ process can be defined in general polygons. Suppose $(\Omega; y^{L, l}, \ldots, y^{L, 1}, x, y^{R, 1}, \ldots, y^{R, r}, y)$ is a polygon with $l+r+2$ marked points. 
Let $\phi$ be any conformal map from $\Omega$ onto $\HH$ such that $\phi(x)=0$ and $\phi(y)=\infty$. 
We define $\SLE_{\kappa}(\underline{\rho}^L; \underline{\rho}^R)$ in $\Omega$ from $x$ to $y$ with force points $(y^{L, l}, \ldots, y^{L, 1}; y^{R, 1}, \ldots, y^{R, r})$ to be $\phi^{-1}(\eta)$ where $\eta$ is an $\SLE_{\kappa}(\underline{\rho}^L; \underline{\rho}^R)$ in $\HH$ from $0$ to $\infty$ with force points $(\phi(y^{L, l}), \ldots, \phi(y^{L, 1}); \phi(y^{R, 1}), \ldots, \phi(y^{R, r}))$. 
\paragraph{$\SLE_2(-1,-1;-1,-1)$ process.}
Let us discuss $\SLE_2(-1,-1;-1,-1)$ mentioned in Theorem~\ref{thm::cvg_lerw}. Suppose $(\Omega; d, a, x, b, c, y)$ is a polygon with six marked points and consider $\SLE_2(-1,-1;-1,-1)$ in $\Omega$ from $x$ to $y$ with force points $(d,a;b,c)$. Note that the total of the force point weights is $-4$ which is $\kappa-6$ for $\kappa=2$. Such process is target independent in the following sense: for distinct $y_1, y_2\in (cd)$, let $\eta_i$ be the $\SLE_2(-1,-1;-1,-1)$ in $\Omega$ from $x$ to $y_i$ with force points $(d,a;b,c)$, and let $T_i$ be the first time that $\eta_i$ hits $(cd)$ (in fact, $T_i$ is the continuation threshold of $\eta_i$) for $i=1,2$. Then the law of $(\eta_1(t), 0\le t\le T_1)$ is the same as the law of 
$(\eta_2(t), 0\le t\le T_2)$. See~\cite{SchrammWilsonSLECoordinatechanges} for the target-independence for a general setup. As the law of $(\eta_i(t), 0\le t\le T_i)$ does not depend on the location of the target point, we say that it is an $\SLE_2(-1, -1; -1, -1)$ in $\Omega$ from $x$ to $(cd)$ with force points $(d, a; b, c)$. 


\section{Hypergeometric SLE with $\kappa\ge 8$}
\label{sec::hypersle}

For $\kappa>0$ and $\nu>(-4)\vee(\kappa/2-6)$, define the hypergeometric function (see Appendix~\ref{appendix_hyper_elliptic}): 
\begin{align}\label{eqn::hyper_def}
F(z):=\hF\left(\frac{2\nu+4}{\kappa}, 1-\frac{4}{\kappa},\frac{2\nu+8}{\kappa};z\right).
\end{align}
Set
\begin{equation}\label{eqn::hab_def}
h=\frac{6-\kappa}{2\kappa},\quad a=\frac{\nu+2}{\kappa}, \quad b=\frac{(\nu+2)(\nu+6-\kappa)}{4\kappa}.
\end{equation}
For $x_1<x_2<x_3<x_4$, define partition function
\begin{equation}\label{eqn::hSLE_partition}
\PartF_{\kappa,\nu}(x_1, x_2, x_3, x_4)=(x_4-x_1)^{-2h}(x_3-x_2)^{-2b}z^aF(z),\quad \text{where}\quad z=\frac{(x_2-x_1)(x_4-x_3)}{(x_3-x_1)(x_4-x_2)}.
\end{equation}

The process $\hSLE_{\kappa}(\nu)$ in $\HH$ from $x_1$ to $x_4$ with marked points $(x_2, x_3)$ is the Loewner chain driven by $W_t$ which is the solution to the following SDEs: 
\begin{align}\label{eqn::hypersle_sde}
\begin{cases}
dW_t=\sqrt{\kappa}dB_t+\kappa(\partial_{1}\log\PartF_{\kappa, \nu})(W_t, V_t^2, V_t^3, V_t^4)dt,\quad W_0=x_1; \\
dV_t^i=\frac{2dt}{V_t^i-W_t},\quad V_0^i=x_i, \quad\text{for }i=2,3,4;
\end{cases}
\end{align}
where $(B_t, t\ge 0)$ is one-dimensional Brownian motion starting from $0$.
Combining~\eqref{eqn::hypersle_sde} and~\eqref{eqn::euler_eq}, the law of $\hSLE_{\kappa}(\nu)$ is the same as $\SLE_{\kappa}$ in $\HH$ from $x_1$ to $\infty$ weighted by the following local martingale:
\begin{equation}\label{eqn::hypersle_mart}
M_t=g_t'(x_2)^{b}g_t'(x_3)^bg_t'(x_4)^{h}\PartF_{\kappa, \nu}(W_t, g_t(x_2), g_t(x_3), g_t(x_4)).
\end{equation}

It is clear that the solution to~\eqref{eqn::hypersle_sde} is well-defined up to the swallowing time of $x_2$. We denote by $T_{x_3}$ the swallowing time of $x_3$. To fully understand solutions to~\eqref{eqn::hypersle_sde}, we will address the following two questions: 
\begin{itemize}
\item Is there a unique solution (in law) to~\eqref{eqn::hypersle_sde} up to and including $T_{x_3}$?
\item Whether the Loewner chain is generated by a continuous curve up to and including $T_{x_3}$?
\end{itemize}
The answers to these questions are positive. The proof turns out to be very different for $\kappa\neq 8$ and for $\kappa=8$. 
These questions are addressed in~\cite{WuHyperSLE} for $\kappa\in (0,8)$. A similar analysis applies to the case when $\kappa>8$, see Section~\ref{subsec::hsle_continuity}. The proof for $\kappa=8$ uses analysis from UST and will be completed in Section~\ref{sec::ust}, more precisely, in the proof of Theorem~\ref{thm::ust_hsle} and in the proof of Corollary~\ref{cor::ust_hsle8}. 

In summary, for $\kappa\ge 8$, we will show that the process is well-defined up to $T_{x_3}$; moreover, it is generated by a continuous curve $\eta$ up to and including $T_{x_3}$. After $T_{x_3}$, we continue the process as a standard $\SLE_{\kappa}$ from $\eta(T_{x_3})$ towards $x_4$ in the remaining domain. The reason for such choice comes from the observation in the discrete setup, see the second last paragraph in the proof of Theorem~\ref{thm::ust_hsle}.

In the above, we have defined $\hSLE$ in $\HH$ and we may extend the definition to general quad via conformal image: For a quad $(\Omega; x_1, x_2, x_3, x_4)$, let $\phi$ be any conformal map from $\Omega$ onto $\HH$ such that $\phi(x_1)<\phi(x_2)<\phi(x_3)<\phi(x_4)$. We define $\hSLE_{\kappa}(\nu)$ in $\Omega$ from $x_1$ to $x_4$ with marked points $(x_2, x_3)$ to be $\phi^{-1}(\eta)$ where $\eta$ is an $\hSLE_{\kappa}(\nu)$ in $\HH$ from $\phi(x_1)$ to $\phi(x_4)$ with marked points $(\phi(x_3), \phi(x_4))$. Moreover, if the marked points $x_1, x_2, x_3, x_4$ lie on sufficiently regular boundary segments (e.g. $C^{1+\eps}$ for some $\eps>0$), we may extend the partition function $\PartF_{\kappa, \nu}$ via conformal image: 
\begin{equation}\label{def::PPF_quad}
\PartF_{\kappa, \nu}(\Omega; x_1, x_2, x_3, x_4)=\phi'(x_1)^{-h}\phi'(x_2)^{-b}\phi'(x_3)^{-b}\phi'(x_4)^{-h}\PartF_{\kappa, \nu}(\phi(x_1), \phi(x_2), \phi(x_3), \phi(x_4)). 
\end{equation}

\subsection{Continuity of $\hSLE_{\kappa}(\nu)$ with $\kappa\ge 8$}
\label{subsec::hsle_continuity}

To derive the continuity of $\hSLE_{\kappa}(\nu)$, it is more convenient to work in $\HH$ with $x_1=0$ and $x_4=\infty$. Consider $\hSLE_{\kappa}(\nu)$ in $\HH$ from $0$ to $\infty$ with marked points $(x,y)$ where $0<x<y$. In this case, the SDEs~\eqref{eqn::hypersle_sde} becomes the following: 
\begin{align}\label{eqn::hsle_sde}
\begin{cases}
dW_t=\sqrt{\kappa}dB_t+\frac{(\nu+2)dt}{W_t-V_t^x}+\frac{-(\nu+2)dt}{W_t-V_t^y}-\kappa\frac{F'(Z_t)}{F(Z_t)}\left(\frac{1-Z_t}{V_t^y-W_t}\right)dt,\quad W_0=0; \\
dV_t^x=\frac{2dt}{V_t^x-W_t},\quad dV_t^y=\frac{2dt}{V_t^y-W_t},\quad V_0^x=x, V_0^y=y; \quad\text{where}\quad Z_t=\frac{V_t^x-W_t}{V_t^y-W_t}. 
\end{cases}
\end{align}
We denote by $T_x$ the swallowing time of $x$ and by $T_y$ the swallowing time of $y$. 
The main result of this section is the continuity of the process up to and including $T_y$. 


From~\eqref{eqn::hsle_sde}, it is clear that the Loewner chain is well-defined up to $T_x$. As in~\eqref{eqn::hypersle_mart}, the process has the same law as $\SLE_{\kappa}$ in $\HH$ from 0 to $\infty$ weighted by the following local martingale: 
\begin{equation}\label{eqn::hsle_mart}
M_t=g_t'(x)^bg_t'(y)^b(g_t(y)-g_t(x))^{-2b}Z_t^a F(Z_t).
\end{equation}
In particular, it is generated by continuous curve up to $T_x$. However, the continuity of the process around $T_x$ or $T_y$ can be problematic in general. The behavior of the process near $T_x$ or $T_y$ depends on the asymptotic of the hypergeometric function $F$ in the definition of the driving function, see Lemma~\ref{lem::hyperfunction_asy}.

\begin{lemma}\label{lem::hsle_continuity_uptoTy}
Fix $\kappa\ge 8, \nu>-2$ and $0<x<y$. Suppose $\eta\sim \hSLE_{\kappa}(\nu)$ in $\HH$ from 0 to $\infty$ with marked points $(x,y)$. Then $\eta$ is generated by a continuous curve up to $T_y$. 
\end{lemma}
\begin{proof}
We compare the law of $\eta$ with $\SLE_{\kappa}(\nu+2, \kappa-6-\nu)$ in $\HH$ from 0 to $\infty$ with force points $(x,y)$. By~\eqref{eqn::sle_kappa_rho_mart} and~\eqref{eqn::hsle_mart}, the Radon-Nikodym derivative is given by $R_t/R_0$ where 
\[R_t=F(Z_t)(V_t^y-W_t)^{4/\kappa-1}.\]

When $\kappa>8$, we write 
\[R_t=F(Z_t)(1-Z_t)^{1-8/\kappa} (V_t^y-V_t^x)^{8/\kappa-1}(V_t^y-W_t)^{-4/\kappa}.\]
By~\eqref{eqn::hyperfunction_asy}, the function $F(z)(1-z)^{1-8/\kappa}$ is uniformly bounded for $z\in [0,1]$. Define, for $n\ge 1$,  
\[S_n=\inf\{t: V_t^y-V_t^x\le 1/n\}.\]
Then $R_t$ is bounded up to $S_n$. Since $\SLE_{\kappa}(\nu+2, \kappa-6-\nu)$ is generated by a continuous curve, the process $\eta$ is generated by a continuous curve up to $S_n$. This holds for any $n$, thus $\eta$ is generated by a continuous curve up to $T_y=\lim_n S_n$. 

When $\kappa=8$, we write
\[R_t=\frac{F(Z_t)}{\log\frac{1}{1-Z_t}}\left(\log\frac{1}{1-Z_t}\right)(V_t^y-W_t)^{4/\kappa-1}.\]
By~\eqref{eqn::hyperfunction_asy_8}, we know that $F(z)/\left(\log\frac{1}{1-z}\right)$ is uniformly bounded for $z\in [0,1]$. Define $S_n$ in the same way as before. Then $R_t$ is bounded up to $S_n$. Similarly, the process $\eta$ is continuous up to $T_y=\lim_n S_n$. 
\end{proof}

In Lemma~\ref{lem::hsle_continuity_uptoTy}, we obtain the continuity of $\hSLE_{\kappa}(\nu)$ up to $T_y$ by showing that the process is absolutely continuous with respect to $\SLE_{\kappa}(\nu+2, \kappa-6-\nu)$. However, the absolute continuity is nolonger true when the process approaches $T_y$. In the following, we will derive the continuity of the process as $t\to T_y$. 
From Lemma~\ref{lem::hyperfunction_asy}, we see that the asymptotic of $F(z)$ as $z\to 1$ is very different between $\kappa>8$ and $\kappa=8$. We will treat the two cases separately: we prove the continuity of $\hSLE$ with $\kappa>8$ in Lemma~\ref{lem::hsle_continuity_includingTy_greater8}, and the continuity with $\kappa=8$ in Corollary~\ref{cor::ust_hsle8}. 

\begin{lemma}\label{lem::hsle_continuity_includingTy_greater8}
Fix $\kappa> 8, \nu\ge \kappa/2-6$ and $0<x<y$. Suppose $\eta\sim\hSLE_{\kappa}(\nu)$ in $\HH$ from 0 to $\infty$ with marked points $(x,y)$. Then $\eta$ is generated by a continuous curve up to and including $T_y$. 
\end{lemma}

\begin{proof}
Since $\hSLE_{\kappa}(\nu)$ is scaling invariant, we may assume $y=1$ and $x\in (0,1)$. We denote $T_y$ by $T$. In this lemma, we discuss the continuity of the process $(K_t, 0\le t\le T)$ as $t\to T$. We need to zoom in around the point $1$. To this end, we perform a standard change of coordinate and parameterize the process according to the capacity seen from the point $1$. See \cite[Theorem~3]{SchrammWilsonSLECoordinatechanges}.

Set $\varphi(z)=z/(1-z)$, this is the M\"{o}bius transformation of $\HH$ that sends the triple $(0,1,\infty)$ to $(0,\infty, -1)$. Denote by $\tilde{x}=\varphi(x)=x/(1-x)>0$. Denote the image of $(K_t, 0\le t\le T)$ under $\varphi$ by $(\tilde{K}_s, 0\le s\le \tilde{S})$ where we parameterize this process by its capacity seen from $\infty$. Let $(\tilde{g}_s, s\ge 0)$ be the corresponding family of conformal maps and $(\tilde{W}_s, s\ge 0)$ be the driving function. From~\eqref{eqn::hypersle_mart} and~\eqref{def::PPF_quad}, the law of $\tilde{K}_s$ is the same as $\SLE_{\kappa}$ in $\HH$ from 0 to $\infty$ weighted by the following local martingale 
\[\tilde{M}_s=\tilde{g}_s'(-1)^h\tilde{g}_s'(\tilde{x})^b(\tilde{W}_s-\tilde{g}_s(-1))^{-2h}\tilde{Z}_s^aF(\tilde{Z}_s),\quad\text{where }\tilde{Z}_s=\frac{\tilde{g}_s(\tilde{x})-\tilde{W}_s}{\tilde{g}_s(\tilde{x})-\tilde{g}_s(-1)}. \]
Compare the law of $\tilde{K}$ with respect to $\SLE_{\kappa}(2;\nu+2)$ in $\HH$ from 0 to $\infty$ with force points $(-1; \tilde{x})$. The Radon-Nikodym derivative is given by $\tilde{R}_s/\tilde{R}_0$ where
\[\tilde{R}_s=(1-\tilde{Z}_s)^{1-8/\kappa}F(\tilde{Z}_s)(\tilde{g}_s(\tilde{x})-\tilde{g}_s(-1))^{1-8/\kappa-2a}.\]

When $\kappa>8$ and $\nu\ge \kappa/2-6>-2$, the function $(1-z)^{1-8/\kappa}F(z)$ is uniformly bounded on $z\in[0,1]$ due to~\eqref{eqn::hyperfunction_asy}. The process $\tilde{g}_s(\tilde{x})-\tilde{g}_s(-1)$ is increasing in $s$, thus $\tilde{g}_s(\tilde{x})-\tilde{g}_s(-1)\ge 1/(1-x)$. Since $\nu\ge \kappa/2-6$, the exponent of the term $\tilde{g}_s(\tilde{x})-\tilde{g}_s(-1)$ is $1-8/\kappa-2a\le 0$. Therefore, $\tilde{R}_s$ is bounded. This implies that the law of $\tilde{K}_s$ is absolutely continuous with respect to the law of $\SLE_{\kappa}(2;\nu+2)$ up to and including the swallowing time of $-1$. Hence $(\tilde{K}_s, 0\le s\le \tilde{S})$ is generated by a continuous curve up to and including $\tilde{S}$. 
In particular, this implies that the original process $(K_t, 0\le t\le T)$ is generated by a continuous curve up to and including $T$.
\end{proof}

To sum up the results in this section for $\kappa>8$, we have the following continuity of $\hSLE_{\kappa}(\nu)$. 
\begin{proposition}
Fix $\kappa>8, \nu\ge \kappa/2-6$ and $x_1<x_2<x_3<x_4$. The process $\hSLE_{\kappa}(\nu)$ in $\HH$ from $x_1$ to $x_4$ is almost surely generated by a continuous curve. 
\end{proposition}
\begin{proof}
Lemma~\ref{lem::hsle_continuity_includingTy_greater8} gives the continuity of $\hSLE_{\kappa}(\nu)$ up to and including $T_{x_3}$. After $T_{x_3}$, we continue the process by standard $\SLE_{\kappa}$ in the remaining domain from $\eta(T_{x_3})$ to $x_4$. Thus, the process $\hSLE_{\kappa}(\nu)$ is continuous for all time.
\end{proof}

\subsection{Non-reversibility of $\hSLE_{\kappa}(\nu)$ with $\kappa>8$}
\label{subsec::hsle_rev_greater8}
The time-reversal of $\SLE_{\kappa}$ with $\kappa>8$ was fully addressed in \cite[Theorem~1.19]{MillerSheffieldIG4}: consider $\SLE_{\kappa}(\rho_1; \rho_2)$ with force points next to the starting point for $\rho_1, \rho_2\in (-2, \kappa/2-2)$, its time-reversal is an $\SLE_{\kappa}(\kappa/2-4-\rho_2; \kappa/2-4-\rho_1)$ process with force points next to the starting point. In particular, the time-reversal of $\SLE_{\kappa}$ is an $\SLE_{\kappa}(\kappa/2-4; \kappa/2-4)$ process. This indicates that the time-reversal of $\hSLE_{\kappa}(\nu)$ with $\kappa>8$ is a variant of $\SLE_{\kappa}$ where one has four extra marked points. In particular, the time-reversal is nolonger in the family of hSLE which is a variant of $\SLE$ with two extra marked points, see Lemma~\ref{lem::hsle_rev_greater8}. Therefore, it is only reasonable to talk about reversibility of $\hSLE_{\kappa}(\nu)$ with $\kappa\le 8$. The reversibility of $\hSLE_{\kappa}(\nu)$ with $\kappa<8$ was addressed in \cite[Section~3.3]{WuHyperSLE}. We will discuss the reversibility of $\hSLE_{\kappa}(\nu)$ with $\kappa=8$ in Section~\ref{subsec::hsle_rev}.

\begin{lemma}\label{lem::hsle_rev_greater8}
When $\kappa>8$ and $\nu\ge \kappa/2-6$. The time-reversal of $\hSLE_{\kappa}(\nu)$ is not an $\hSLE_{\kappa}(\tilde{\nu})$ for any value of $\tilde{\nu}$. 
\end{lemma} 

\begin{proof}
Fix $x_1<x_2<x_3<x_4$. Suppose $\eta$ is an $\hSLE_{\kappa}(\nu)$ in $\HH$ from $x_1$ to $x_4$ and let $\hat{\eta}$ be its time-reversal. Let $\gamma$ be an $\hSLE_{\kappa}(\tilde{\nu})$ in $\HH$ from $x_4$ to $x_1$ with marked points $(x_3, x_2)$. We will compare the laws of $\hat{\eta}$ and $\gamma$ in small neighborhood of $x_4$.
\begin{itemize}
\item In the construction of $\eta$, we know that the process almost surely hits the interval $(x_3, x_4)$ and after the hitting time, we continue the process as a standard $\SLE_{\kappa}$ in the remaining domain. Therefore, the initial segment of $\hat{\eta}$ is the time-reversal of a standard $\SLE_{\kappa}$. By~\cite[Theorem~1.19]{MillerSheffieldIG4}, we know that, in small neighborhood of $x_4$, the law of $\hat{\eta}$ and the law of $\SLE_{\kappa}(\kappa/2-4; \kappa/2-4)$ are absolutely continuous with respect to each other.
\item From the definition of $\hSLE_{\kappa}(\tilde{\nu})$, we know that, in small neighborhood of $x_4$, the law of $\gamma$ and the law of $\SLE_{\kappa}$ are absolutely continuous with respect to each other. 
\end{itemize}
Combining the above two observations, we see that $\hat{\eta}$ can not have the same law as $\gamma$, because $\SLE_{\kappa}(\kappa/2-4; \kappa/2-4)$ and $\SLE_{\kappa}$ are not absolutely continuous with respect to each other. 
\end{proof}
\subsection{Continuity and reversibility of $\hSLE_{\kappa}(\nu)$ with $\kappa= 8$}
\label{subsec::hsle_rev}

The continuity and reversibility of $\hSLE_8$ will be given in Corollary~\ref{cor::ust_hsle8} in Section~\ref{subsec::ust_hsle}. The proof there is based on the convergence of the Peano curve for UST. Assuming this is true, we are able to prove the continuity and reversibility of $\hSLE_8(\nu)$ for $\nu\ge 0$. 

\begin{proof}[Proof of Theorem~\ref{thm::hsle_continuity_reversibility}]
We may assume $x_1=0<x_2=x<x_3=y<x_4=\infty$. Suppose $\eta\sim\hSLE_8(\nu)$ in $\HH$ from $0$ to $\infty$ with marked points $(x, y)$. 
Recall that $\eta$ has the same law as $\SLE_8$ in $\HH$ from 0 to $\infty$ weighted by the following local martingale: 
\begin{equation}\label{eqn::rev_mart_aux1}
M_t=g_t'(x)^bg_t'(y)^b(g_t(y)-g_t(x))^{-2b}Z_t^aF(Z_t),
\end{equation}
where \[h=\frac{-1}{8},\quad 
a=\frac{\nu+2}{8}, \quad b=\frac{(\nu+2)(\nu-2)}{32},\quad 
F(z)=\hF\left(2a, \frac{1}{2}, 2a+\frac{1}{2}; z\right).\]
Suppose $\gamma\sim\hSLE_8$ in $\HH$ from 0 to $\infty$ with marked points $(x,y)$. Then $\gamma$ has the same law as $\SLE_8$ in $\HH$ from 0 to $\infty$ weighted by the following local martingale: 
\begin{equation}\label{eqn::rev_mart_aux2}
N_t=g_t'(x)^hg_t'(y)^h(g_t(y)-g_t(x))^{-2h}Z_t^{1/4}G(Z_t),\quad \text{where }G(z)=\hF\left(\frac{1}{2}, \frac{1}{2}, 1; z\right).
\end{equation}
Combining~\eqref{eqn::rev_mart_aux1} and~\eqref{eqn::rev_mart_aux2}, we see that the law of $\eta$ is the same as the law of $\gamma$ weighted by the following local martingale: 
\[R_t=\frac{M_t}{N_t}=\left(\frac{g_t'(x)g_t'(y)}{(g_t(y)-g_t(x))^2}\right)^{\nu^2/32}\times Z_t^{\nu/8} \times \frac{F(Z_t)}{G(Z_t)}.\]
The term $Z_t$ takes values in $[0,1]$, the term $F(Z_t)/G(Z_t)$ is uniformly bounded due to~\eqref{eqn::hyperfunction_asy_8}. The term 
\[\frac{g_t'(x)g_t'(y)}{(g_t(y)-g_t(x))^2}\]
is the boundary Poisson kernel of the domain $\HH\setminus\gamma[0,t]$ and it is positive and bounded from above by $(y-x)^{-2}$. Thus $R_t$ is uniformly bounded and the law of $\eta$ is absolutely continuous with respect to the law of $\gamma$ up to and including $T_y$. 

By Corollary~\ref{cor::ust_hsle8}, the process $\gamma$ is continuous up to and including $T_y$ and $\gamma\cap [x,y]=\emptyset$. By the absolute continuity, the process $\eta$ is continuous up to and including $T_y$ and $\eta\cap [x,y]=\emptyset$. After $T_y$, we continue the process by standard $\SLE_8$ in the remaining domain from $\eta(T_y)$ to $\infty$. Thus $\eta$ is a continuous curve for all time. 

It remains to show the reversibility. 
We denote by $D$ the connected component of $\HH\setminus\gamma$ with $[x,y]$ on the boundary. Since $\gamma$ is continuous and $\gamma\cap [x,y]=\emptyset$, we have $Z_t\to 1$ as $t\to T_y$. Thus, from~\eqref{eqn::hyperfunction_asy_8}, 
\[R_t\to H_D(x,y)^{\nu^2/32}\times \sqrt{\pi}\frac{(\nu+2)\Gamma(2+\frac{\nu}{4})}{(\nu+4)\Gamma(\frac{3}{2}+\frac{\nu}{4})},\quad\text{as }t\to T_y.\]
From this, we find that the law of $\eta$ is the same as the law of $\gamma$ weighted by the boundary Poisson kernel $H_D(x,y)^{\nu^2/32}$. Since the law of $\gamma$ is reversible due to Corollary~\ref{cor::ust_hsle8}, and the boundary Poisson kernel is conformally invariant, we obtain the reversibility of  $\eta$. This completes the proof.
\end{proof}


\section{Convergence of UST in quads}
\label{sec::ust}
\subsection{UST with Dobrushin boundary conditions}
\label{subsec::ust_dobrushin}
We first introduce Dobrushin domains. Informally speaking, a Dobrushin domain is a simply connected subgraph $\Omega$ of $\Z^2$ with two fixed boundary points $a, b$, and the boundary arc $(ab)$ is in $\Z^2$ and the boundary arc $(ba)$ is in $(\Z^2)^*$.

Consider the medial lattice $(\Z^2)^{\diamond}$. We orient the edges of the medial lattice such that edges of a face containing a vertex in $\Z^2$ are oriented counterclockwise and edges of a face containing a vertex in $(\Z^2)^*$ are oriented clockwise.
Let $a^\diamond, b^\diamond$ be two distinct medial vertices, and $(a^{\diamond}b^{\diamond})$ and $(b^{\diamond}a^{\diamond})$ be two paths of neighboring medial vertices satisfying the following conditions: (1) the edges along $(a^{\diamond}b^{\diamond})$ point in clockwise way with the orientation inherited from the medial lattice; (2) the edges along $(b^{\diamond}a^{\diamond})$ point in counterclockwise way with the orientation inherited from the medial lattice; (3) the two paths are edge-avoiding and $(a^{\diamond}b^{\diamond})\cap(b^{\diamond}a^{\diamond})=\{a^\diamond, b^\diamond\}$. See Figure~\ref{fig::Dobrushin_Peano}. 

Given $(a^{\diamond}b^{\diamond})$ and $(b^{\diamond}a^{\diamond})$, the medial Dobrushin domain $(\Omega^\diamond; a^\diamond, b^\diamond)$ is defined as the subgraph of $(\Z^2)^\diamond$ induced by the vertices enclosed by or on the path $(a^{\diamond}b^{\diamond})\cup(b^{\diamond}a^{\diamond})$. 
Let $\Omega\subset\Z^2$ be the graph with edge set consisting of edges passing through end-points of medial edges in $E(\Omega^{\diamond})\setminus (b^{\diamond}a^{\diamond})$ and with vertex set given by the endpoints of these edges.  Here we always assume that $\Omega$ is simply connected. The vertices of $\Omega$ nearest to $a^\diamond, b^\diamond$ are denoted by $a, b$ and we call $(\Omega; a, b)$ primal Dobrushin domain. Let $(ab)$ be the set of edges corresponding to medial vertices in $(a^{\diamond}b^{\diamond})\cap\partial\Omega^\diamond$.
Let $\Omega^*\subset\Z^2$ be the graph with edge set consisting of edges passing through end-points of medial edges in $E(\Omega^{\diamond})\setminus (a^{\diamond}b^{\diamond})$ and with vertex set given by the endpoints of these edges. Here we always assume that $\Omega^*$ is simply connected.
The vertices of $\Omega^*$ nearest to $a^\diamond, b^\diamond$ are denoted by $a^*, b^*$. Let $(b^*a^*)$ be the set of edges corresponding to medial vertices in $(b^{\diamond}a^{\diamond})\cap\partial\Omega^\diamond$. Note that $a$ is the vertex of $\Omega$ that is nearest to $a^{\diamond}$ and $a^*$ is the vertex of $\Omega^*$ that is nearest to $a^{\diamond}$. See Figure~\ref{fig::Dobrushin_Peano}. 

\begin{figure}[ht!]
\begin{center}
\includegraphics[width=0.4\textwidth]{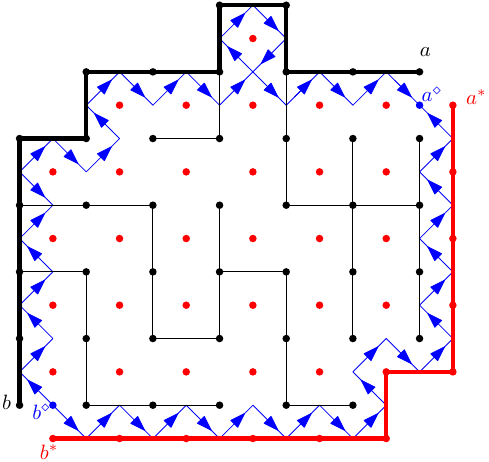}$\quad$
\includegraphics[width=0.4\textwidth]{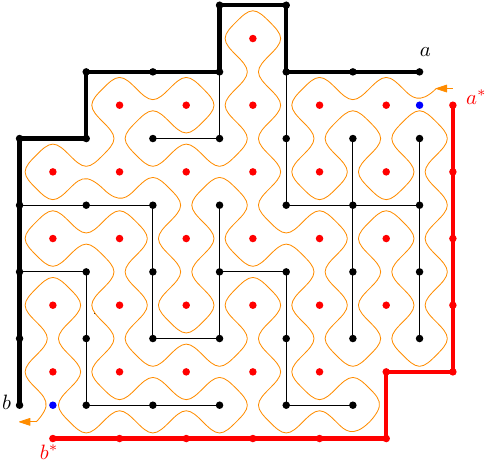}
\end{center}
\caption{\label{fig::Dobrushin_Peano} In the left panel, the solid edges in black are wired boundary arc $(ab)$, the solid edges in red are dual-wired boundary arc $(b^*a^*)$. The edges in blue are the boundary arcs $(a^{\diamond}b^{\diamond})$ and $(b^{\diamond}a^{\diamond})$ on the medial lattice. The thin edges are in the tree $\LT$. In the right panel, the curve in orange is the Peano curve associated to $\LT$.}
\end{figure}

Suppose that $\LT$ is a spanning tree on some primal Dobrushin domain $(\Omega; a, b)$ with $(ab)$ wired. 
Consider its dual configuration $\LT^*\subset E(\Omega^*)$ defined as follows: $\one_{\LT^*}(e^*)=1-\one_{\LT}(e)$ for any $e\in E(\Omega)$ where $e^*$ is the dual edge corresponding to $e$. 
It is clear that $\LT^*$ is a spanning tree on the dual Dobrushin domain $(\Omega^*; a^*, b^*)$ with $(b^*a^*)$ wired. 
There exists a unique  path, called Peano curve, on $(\Z^2)^\diamond$,  running between $\LT$ and $\LT^*$ from $a^{\diamond}$ to $b^{\diamond}$. The following theorem concerns the convergence of the Peano curve of UST.  

\begin{theorem}\label{thm::ust_Dobrushin}
Fix a Dobrushin domain $(\Omega; a, b)$ such that $\partial\Omega$ is $C^{1}$ and simple. Suppose that a sequence of medial Dobrushin domains $(\Omega_\delta^\diamond; a_\delta^\diamond, b_\delta^\diamond)$ converges to $(\Omega; a, b)$  in  the sense of~\eqref{eqn::topology}. 
Consider UST on the primal domain $\Omega_{\delta}$ with $(a_{\delta}b_{\delta})$ wired. Denote by $\eta_{\delta}$ the induced Peano curve. Then the law of $\eta_{\delta}$ converges weakly to $\mathrm{SLE}_{8}$ in $\Omega$ from $a$ to $b$. 
\end{theorem}

This statement is proved in~\cite[Theorems~4.7 and~4.8]{LawlerSchrammWernerLERWUST}. The proof of tightness of Peano curves uses argument in~\cite{SchrammScalinglimitsLERWUST} where the notions of trunk and dual trunk play an important role. The trunk and dual trunk are also important later in this article, and we will give its formal definition below. Roughly speaking, they are the limits of the UST $\LT_{\delta}$ and its dual $\LT_{\delta}^*$ as $\delta\to 0$. 

Consider the UST  $\LT_{\delta}$ on the primal domain. For $\eps>0$, we first define its $\eps$-trunk. For any $x,y\in\LT_{\delta}$, there is a unique simple path on $\LT_{\delta}$ from $x$ to $y$, which we denote by $\eta_{\delta}^{x,y}$. We denote by $x'$ the first point at which $\eta_{\delta}^{x,y}$ hits $\partial B(x,\eps)$ and $y'$ the last point at which $\eta_{\delta}^{x,y}$ hits $\partial B(y,\eps)$. We denote by $\mathcal{I}_{\eps}(x,y)$ the unique simple path on $\LT_{\delta}$ connecting $x'$ and $y'$. If such $x'$ or $y'$ does not exist, we define $\mathcal{I}_{\eps}(x,y)=\emptyset$. Then, the $\eps$-trunk is defined to be 
\[\mathbf{trunk}_{\delta}\left(\eps\right):=\bigcup_{x,y\in\LT_{\delta}}\mathcal{I}_{\eps}(x,y).\] 
We can couple the configurations in the same probability space and choose $\delta_{m}\to 0$ such that $\mathbf{trunk}_{\delta_{m}}(\frac{1}{n})$ converges in Hausdorff  metric for every $n\in\mathbb{N}$ as $m\to\infty$ almost surely.  We define 
\[\mathbf{trunk}_{0}\left(\frac{1}{n}\right):=\lim_{m\to\infty}\mathbf{trunk}_{\delta_{m}}\left(\frac{1}{n}\right),\quad \text{and}\quad\mathbf{trunk}=\bigcup_{n>0}\mathbf{trunk}_{0}\left(\frac{1}{n}\right).\] 
Note that the wired boundary arc $(a_\delta b_\delta)$ belongs to the $\ust$ on $\Omega_\delta$. Thus, by definition, $\mathbf{trunk}$ contains $(ab)$. Similarly, we define $\mathbf{trunk}_{\delta}^*\left(\eps\right)$, $\mathbf{trunk}^{*}_{0}\left(\frac{1}{n}\right)$ and $\mathbf{trunk}^{*}$  for the dual UST $\LT_{\delta}^*$.
\cite[Theorem~11.1]{SchrammScalinglimitsLERWUST} guarantees that, if $\partial\Omega$ is $C^1$ and simple, we have 
\[\PP[\mathbf{trunk}\cap\mathbf{trunk}^*=\emptyset]=1.\]

The strategy of the proof Theorem~\ref{thm::ust_Dobrushin} can be summarized as follows: The first step is to show the tightness of the Peano curves. This is a consequence of the fact that $\mathbf{trunk}\cap\mathbf{trunk}^*=\emptyset$ almost surely. The second step is to show the convergence of driving function by martingale observable~\cite[Theorem~4.4]{LawlerSchrammWernerLERWUST}. This step only requires $\Omega_{\delta}^{\diamond}\to\Omega$ in the Carath\'{e}odory sense and does not need {\color{red}}smoothness regularity assumption on $\partial\Omega$. 

\subsection{UST in quads: tightness}
\label{subsec::ust_quads}

Now we introduce the discrete quad in Theorem~\ref{thm::cvg_triple}. 
Informally speaking, this is a simply connected subgraph $\Omega$ of $\Z^2$ with four fixed boundary points $a,b,c,d$ in counterclockwise order, and the boundary arcs $(ab), (cd)$ are in $\Z^2$ and the boundary arcs $(bc), (da)$ are in $(\Z^2)^*$.

Let $a^\diamond, b^\diamond, c^\diamond, d^\diamond$ be four distinct medial vertices, and $(a^{\diamond}b^{\diamond})$,  $(b^{\diamond}c^{\diamond})$, $(c^{\diamond}d^{\diamond})$ and $(d^{\diamond}a^{\diamond})$ be four paths of neighboring medial vertices satisfying the following conditions: (1) the edges along $(a^{\diamond}b^{\diamond})$ and $(c^{\diamond}d^{\diamond})$ point in clockwise way with the orientation inherited from the medial lattice; (2) the edges along $(b^{\diamond}c^{\diamond})$ and $(d^{\diamond}a^{\diamond})$ point in counterclockwise way with the orientation inherited from the medial lattice; 
(3) all the paths are edge-avoiding and $(a^{\diamond}b^{\diamond})\cap (b^{\diamond}c^{\diamond})=\{b^\diamond\}$, $(b^{\diamond}c^{\diamond})\cap (c^{\diamond}d^{\diamond})=\{c^\diamond\}$, $(c^{\diamond}d^{\diamond})\cap (d^{\diamond}a^{\diamond})=\{d^\diamond\}$, $(d^{\diamond}a^{\diamond})\cap (a^{\diamond}b^{\diamond})=\{a^\diamond\}$. See Figure~\ref{fig::quad}. 

Given $(a^{\diamond}b^{\diamond}), (b^{\diamond}c^{\diamond}), (c^{\diamond}d^{\diamond})$ and $(d^{\diamond}a^{\diamond})$, the medial quad $(\Omega^\diamond; a^\diamond, b^\diamond, c^\diamond, d^\diamond)$ is defined as the subgraph of $(\Z^2)^\diamond$ induced by the vertices enclosed by or on the path $(a^{\diamond}b^{\diamond})\cup(b^{\diamond}c^{\diamond})\cup (c^{\diamond}d^{\diamond})\cup(d^{\diamond}a^{\diamond})$. 
The inner boundary $\partial\Omega^\diamond$ is the set of vertices  of $\Omega^\diamond$ with strictly fewer than four incident edges in $E(\Omega^{\diamond})$. 
Let $\Omega\subset\Z^2$ be the graph with edge set consisting of edges passing through end-points of medial edges in $E(\Omega^{\diamond})\backslash \left((b^{\diamond}c^{\diamond})\cup (d^{\diamond}a^{\diamond})\right)$ and with vertex set given by the endpoints of these edges. Here we always assume that $\Omega$ is simply connected. The vertices of $\Omega$ nearest to $a^\diamond, b^\diamond, c^\diamond, d^\diamond$ are denoted by $a, b, c, d$ and we call $(\Omega; a, b, c, d)$ the primal quad.  
Let $(ab)$ and $(cd)$ be the set of edges corresponding to medial vertices in $\partial\Omega^\diamond$, which are also endpoints of medial edges in $(a^{\diamond}b^{\diamond})$  and $(c^{\diamond}d^{\diamond})$ respectively.  One can define $(bc)$ and $(da)$ to be the two components of $\partial \Omega\backslash \left((ab)\cup (cd)\right)$. See Figure~\ref{fig::quad}. 
Similarly, we can define the dual quad as follows.  Let $\Omega^*\subset(\Z^2)^*$ be the graph with edge set consisting of edges passing through end-points of medial edges in $E(\Omega^{\diamond})\backslash \left((a^{\diamond}b^{\diamond})\cup (c^{\diamond}d^{\diamond})\right)$ and with vertex set given by the endpoints of these edges. Here we always assume that $\Omega^*$ is simply connected. The vertices of $\Omega^*$ nearest to $a^\diamond, b^\diamond, c^\diamond, d^\diamond$ are denoted by $a^*, b^*, c^*, d^*$ and we call $(\Omega^*; a^*, b^*, c^*, d^*)$ the dual quad. Let $(b^*c^*)$ and $(d^*a^*)$ be the set of edges corresponding to medial vertices in $(b^{\diamond}c^{\diamond})$ and $(d^{\diamond}a^{\diamond})$, which are also in $\partial\Omega^\diamond$.  One can define $(a^*b^*)$ and $(c^*d^*)$ to be the two components of $\partial \Omega^*\backslash \left((b^*c^*)\cup (d^*a^*)\right)$. See Figure~\ref{fig::quad}. 

Suppose that $\LT$ is a spanning tree on some primal quad $(\Omega; a, b, c, d)$ with $(ab)$ wired and $(cd)$ wired respectively. 
Its dual configuration $\LT^*$ is a spanning forest with two trees in the dual quad $(\Omega^*; a^*, b^*, c^*, d^*)$ such that one tree contains the dual-wired arc $(b^*c^*)$ and the other tree contains the dual-wired arc $(d^*a^*)$. There exist two paths on $\left(\Z^2\right)^{\diamond}$ running between $\LT$ and $\LT^*$ from $a^{\diamond}$ to $d^{\diamond}$ and from $b^{\diamond}$ to $c^{\diamond}$ respectively, see Figure~\ref{fig::quadPeano}.
We still call them Peano curves as before. 
With the same notations as in Section~\ref{sec::intro}, we denote by $\eta^{L}$ the Peano curve from $a^{\diamond}$ to $d^{\diamond}$, by $\eta^{R}$ the Peano curve from $b^{\diamond}$ to $c^{\diamond}$ and by $\gamma^{M}$ the unique simple path in $\LT$ from $(ab)$ to $(cd)$.
\begin{figure}[ht!]
\begin{center}
\includegraphics[width=0.5\textwidth]{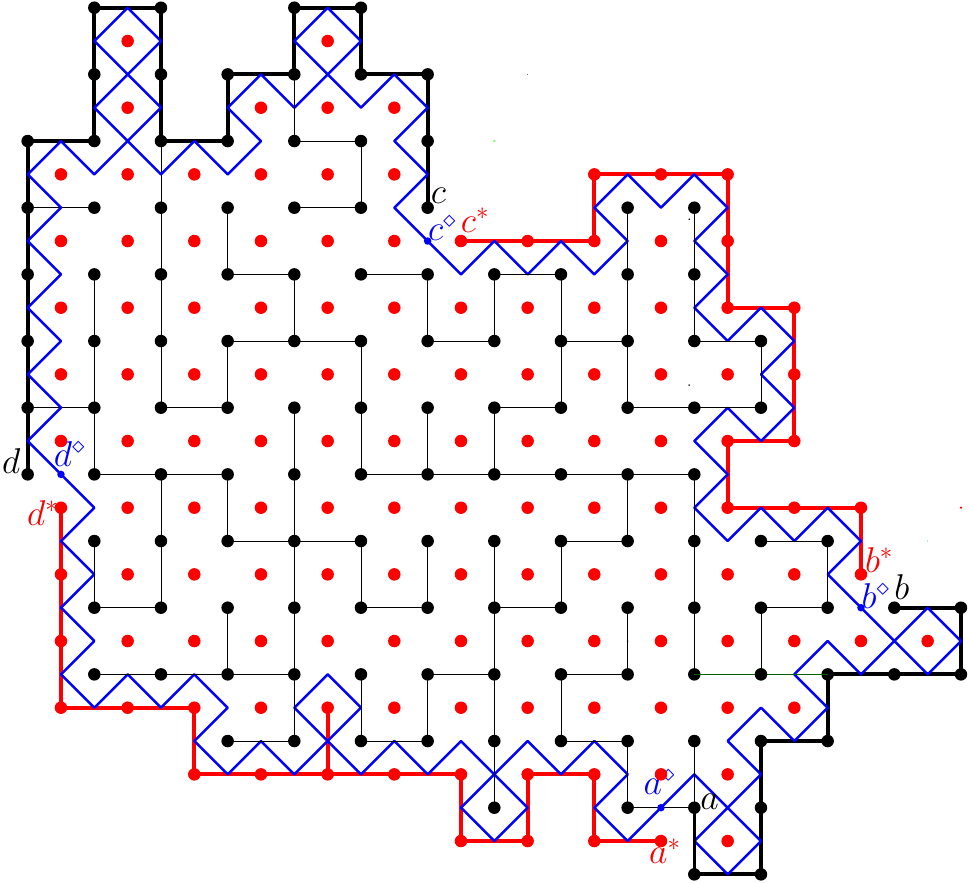}
\end{center}
\caption{\label{fig::quad} The solid edges in black are wired boundary arcs $(ab)$ and $(cd)$, the solid edges in red are dual-wired boundary arcs $(b^*c^*)$ and $(d^*a^*)$. The edges in blue are the boundary arcs $(a^{\diamond}b^{\diamond})$, $(b^{\diamond}c^{\diamond})$, $(c^{\diamond}d^{\diamond})$, and $(d^{\diamond}a^{\diamond})$ on the medial lattice. The thin edges are in the tree $\LT$.}
\end{figure}
\begin{theorem}\label{thm::ust_hsle}
Fix a quad $(\Omega; a, b, c, d)$ such that $\partial\Omega$ is $C^{1}$ and simple. Suppose that a sequence of medial quads $(\Omega_{\delta}^{\diamond}; a_{\delta}^{\diamond}, b_{\delta}^{\diamond}, c_{\delta}^{\diamond}, d_{\delta}^{\diamond})$ converges to $(\Omega; a, b, c, d)$ in the sense of~\eqref{eqn::topology}.
Consider UST on the primal domain $\Omega_{\delta}$ with $(a_{\delta}b_{\delta})$ wired and $(c_{\delta}d_{\delta})$ wired. Denote by $\eta_{\delta}^{L}$ the Peano curve connecting $a_{\delta}^{\diamond}$ and $d_{\delta}^{\diamond}$. Then the law of $\eta^{L}_{\delta}$ converges weakly to $\hSLE_8$ in $\Omega$ from $a$ to $d$ with marked points $(b, c)$ as $\delta\to 0$. 
\end{theorem}

The proof of Theorem~\ref{thm::ust_hsle} consists of two steps: the first step is  showing the tightness of the Peano curves, see Proposition~\ref{prop::tightness}; the second step is constructing an holomorphic observable, see Lemmas~\ref{lem::holomorphic_observable} and~\ref{lem::holo_cvg}. With these two steps at hand, we will complete the proof of Theorem~\ref{thm::ust_hsle} in Section~\ref{subsec::ust_hsle}. Although these steps are standard, 
the proof involves a non-trivial calculation, see Lemma~\ref{lem::drivingfunction}. 

\begin{proposition}\label{prop::tightness}
Assume the same setup as in Theorem~\ref{thm::ust_hsle}. The family of Peano curves $\{\eta^{L}_\delta\}_{\delta>0}$ is tight. Furthermore, $\PP[\eta^L\cap [bc]=\emptyset]=1$ for  any subsequential limit  $\eta^L$ of $\{\eta^{L}_\delta\}_{\delta>0}$.
\end{proposition}
The proof of Proposition~\ref{prop::tightness} has two parts. 
The proof of tightness follows the same strategy as~\cite[Proposition~4.5, Lemma~4.6]{LawlerSchrammWernerLERWUST} where the property of trunk and dual trunk plays an essential role, see Lemma~\ref{lem::tightness_trunk}.  
\begin{lemma}\label{lem::tightness_trunk}
Assume the same setup as in Theorem~\ref{thm::ust_hsle}. 
Define $\mathbf{trunk}_{\delta}\left(\eps\right)$, $\mathbf{trunk}_{0}\left(\frac{1}{n}\right)$ and $\mathbf{trunk}$ for the UST $\LT_{\delta}$ and $\mathbf{trunk}_{\delta}^*\left(\eps\right)$, $\mathbf{trunk}^{*}_{0}\left(\frac{1}{n}\right)$ and $\mathbf{trunk}^{*}$ for  the corresponding dual forest $\LT_{\delta}^*$ as in Section~\ref{subsec::ust_dobrushin}. 
We have 
\begin{equation}\label{eqn::trunktrunkstar_C1regularity}
\PP[\mathbf{trunk}\cap\mathbf{trunk}^{*}\neq\emptyset]=0. 
\end{equation}
\end{lemma}
In the proof of Lemma~\ref{lem::tightness_trunk}, we need that $\partial\Omega$ is $C^1$ and simple.
 To prove  $\PP[\eta^L\cap [bc]=\emptyset]=1$ in Proposition~\ref{prop::tightness}, we first show that $\PP[\eta^L\cap (bc)=\emptyset]=1$, which is a consequence of Lemma~\ref{lem::tightness_trunk}. We then show that $\PP[c\not\in\eta^L]=1$. To this end, we need to estimate the probability $\PP[\eta_{\delta}^L\cap B(c_{\delta}, \eps)\neq\emptyset]$, see Lemma~\ref{lem::subseqlimit_avoids_c}. Similarly, we have $\PP[b\not\in\eta^L]=1$. 
\begin{lemma}\label{lem::subseqlimit_avoids_c}
Assume the same setup as in Theorem~\ref{thm::ust_hsle}. We have
\begin{equation}\label{eqn::subseqlimit_avoids_c}
\lim_{\epsilon\to 0}\varlimsup_{\delta\to 0}\PP[\eta_{\delta}^{L}\text{ hits }B(c_{\delta},\epsilon)]=0.
\end{equation}
\end{lemma}

In the proof of Lemma~\ref{lem::subseqlimit_avoids_c}, we do not need that $\partial\Omega$ is $C^1$ and simple. The details of the proofs of Proposition~\ref{prop::tightness} and Lemma~\ref{lem::tightness_trunk} and Lemma~\ref{lem::subseqlimit_avoids_c} will be given in Appendix~\ref{app::tightness}. 

The following conclusion is about the tightness of the LERW branch in the UST. It will be useful in both this section and Section~\ref{sec::lerw}. 
\begin{proposition}\label{prop::tightness_LERW}
Assume the same setup as in Theorem~\ref{thm::cvg_triple}. Then $\{\gamma_{\delta}^{M}\}_{\delta>0}$ is tight. Moreover, any subsequential limit is a simple curve in $\overline\Omega$ which intersects $\partial\Omega$ only at its two ends. Furthermore, one of the two ends is in $(ab)$ and the other one is in $(cd)$. 
\end{proposition}
The proof of Proposition~\ref{prop::tightness_LERW} is similar to~\cite[Lemma~3.12]{LawlerSchrammWernerLERWUST} and we prove it in Appendix~\ref{app::tightness}.

\subsection{UST in quads: holomorphic observable}
A function $u:\Z^2 \rightarrow \C$ is called (discrete) harmonic at a vertex $x\in \Z^2$ if  
$\sum_{i=1}^4u(x_i)=4u(x)$, where $(x_i: i=1,2,3,4)$ are the four neighbors of $x$ in $\Z^2$.   We say a function $u$ is harmonic on a subgraph of $\Z^2$ if it is harmonic at all vertices in the subgraph.
A function $f: \Z^2\cup (\Z^2)^*\rightarrow \C$  is said to be (discrete) holomorphic around a medial vertex $x^\diamond$  if   one has $f(n)-f(s)=\ii (f(e)-f(w))$, where  $n, s, w, e$ are the vertices incident to $x^\diamond$  in counterclockwise order.  
We say a function $f$ is holomorphic on a subgraph of $\Z^2\cup (\Z^2)^*$ if it is holomorphic at all vertices in the subgraph.
 Note that, for a discrete holomorphic function $f$ on a subgraph of  $\Z^2\cup(\Z^2)^*$, its restriction on $\Z^2$ and its restriction on $\left(\Z^2\right)^*$ are both harmonic (see \cite[Proposition~8.15]{DCParafermionic}). 
We summarize the setup for discrete observable below. 
\begin{itemize}
\item  
Consider the set of spanning trees in the primal domain $\Omega_{\delta}$ with $(a_{\delta}b_{\delta})$ wired and $(c_{\delta}d_{\delta})$ wired. Denote this set by $\mathrm{ST}(\delta)$ and denote its cardinality by $\left|\mathrm{ST}(\delta)\right|$. Let $\LT_{\delta}$ be chosen uniformly among these trees. 
Recall that $\eta_{\delta}^L$ is the Peano curve along $\LT_{\delta}$ from $a_{\delta}^{\diamond}$ to $d_{\delta}^{\diamond}$, and $\eta_{\delta}^R$ is the Peano curve along $\LT_{\delta}$ from $b_{\delta}^{\diamond}$ to $c_{\delta}^{\diamond}$, and $\gamma_{\delta}^M$ is the unique simple path in $\LT_{\delta}$ connecting $(a_\delta b_\delta)$ to $(c_\delta d_\delta)$.  
For a vertex $z^*$ in $\Omega_\delta^*$, define $u_\delta(z^*)$ to be  the probability that $z^*$ lies to the right of $\gamma_{\delta}^M$, i.e. $z^*$ lies in the component of $\Omega^*_\delta\backslash \gamma^M_\delta$ with $b_\delta^*$ and $c_\delta^*$ on the boundary. We abuse notation here and regard $\Omega_\delta^*$ as a planar domain.

\item Consider the set of spanning forests in the primal domain $\Omega_{\delta}$ with $(a_{\delta}b_{\delta})$ wired and $(c_{\delta}d_{\delta})$ wired such that it has only two trees: one of them
contains the wired arc $(a_\delta b_\delta)$ and the other one contains the wired arc $(c_\delta d_\delta)$. Denote this set by $\mathrm{SF}_2(\delta)$ and denote its cardinality by $\left|\mathrm{SF}_2(\delta)\right|$. Let $\LF_{\delta}$ be chosen uniformly among these forests. 
For $z\in \Omega_\delta$, define $v_\delta(z)$ to be the probability that $z$ lies in the same tree as the wired arc $(c_\delta d_\delta)$ in $\LF_{\delta}$.
\end{itemize}
\begin{lemma}\label{lem::holomorphic_observable}
Define 
\[
f_{\delta}(\cdot):=u_\delta(\cdot)+\ii\frac{|\mathrm{SF}_2(\delta)|}{|\mathrm{ST}(\delta)|}v_\delta(\cdot). 
\]
We view it as a function on $\Omega_\delta\cup \Omega_\delta^*$: it equals $u_\delta$ on $\Omega^*_\delta$ and it equals $\ii\frac{|\mathrm{SF}_2(\delta)|}{|\mathrm{ST}(\delta)|}v_\delta$ on $\Omega_\delta$. The function $f_{\delta}$ is discrete holomorphic on $(\Omega_\delta\cup \Omega_\delta^*)\backslash \left((a_{\delta}b_{\delta})\cup(c_{\delta}d_{\delta})\cup(b_{\delta}^*c_{\delta}^*)\cup(d_{\delta}^*a_{\delta}^*)\right)$. Moreover, it has the following boundary data: 
\begin{align*}
\begin{cases}
\mathrm{Re}f_{\delta}=0, \quad\text{on }(d_\delta^*a_\delta^*); \quad 
\mathrm{Re} f_{\delta}=1, \quad\text{on }(b_\delta^*c_\delta^*); \\
\mathrm{Im} f_{\delta}=0, \quad \text{on }(a_{\delta}b_{\delta}); \quad 
\mathrm{Im} f_{\delta}=\frac{|\mathrm{SF}_2(\delta)|}{|\mathrm{ST}(\delta)|},\quad\text{on }(c_{\delta}d_{\delta}).
\end{cases}
\end{align*}
\end{lemma}

\begin{proof}
For $z^*\in \Omega_\delta^*$, denote by $E(\LT_{\delta}; z^*)$ the event that $z^*$ lies to the right of $\gamma_{\delta}^M$.  For $z\in \Omega_\delta$, denote by $E(\LF_{\delta}; z)$ the event that $z$ lies in the same tree as the wired arc $(c_\delta d_\delta)$ in $\LF_{\delta}$. 
 Assume $\{n, s\}$ is a primal edge of $\Omega_\delta$, and the corresponding dual edge is denoted by $\{w,e\}$ such that $w,s,e,n$ are in counterclockwise order (see Figure~\ref{fig::observable}).  Then we have 
\begin{align*}
u_\delta(e)-u_\delta(w)=&\PP\left[E(\LT_{\delta}; e)\right]-\PP\left[E(\LT_{\delta}; w)\right]\\
=&\PP\left[E(\LT_{\delta}; e)\cap E(\LT_{\delta}; w)^c\right]-\PP\left[E(\LT_{\delta}; w)\cap E(\LT_{\delta}; e)^c\right]\\
=&\frac{|\mathrm{SF}_2(\delta)|}{|\mathrm{ST}(\delta)|}\PP\left[E(\LF_{\delta}; n)\cap E(\LF_{\delta}; s)^c\right]-\frac{|\mathrm{SF}_2(\delta)|}{|\mathrm{ST}(\delta)|}\PP\left[E(\LF_{\delta}; s)\cap E(\LF_{\delta}; n)^c\right]\\
=&\frac{|\mathrm{SF}_2(\delta)|}{|\mathrm{ST}(\delta)|}\left(\PP\left[E(\LF_{\delta}; n)\right]-\PP\left[E(\LF_{\delta}; s)\right]\right)\\
=&\frac{|\mathrm{SF}_2(\delta)|}{|\mathrm{ST}(\delta)|}\left(v_\delta(n)-v_\delta(s)\right).
\end{align*}
 The third equal sign is due to the observation explained in Figure~\ref{fig::observable}.  This gives the discrete holomorphicity of $f_{\delta}$. 
The boundary data is clear from the construction.
\end{proof}
\begin{figure}
\includegraphics[scale=0.8]{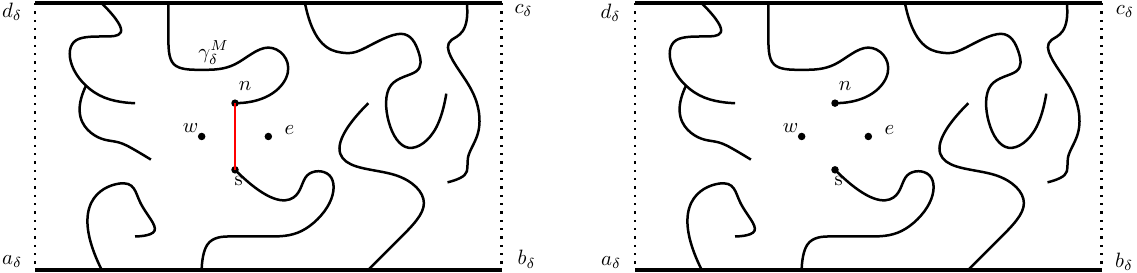}
\caption{\label{fig::observable} We denote by $\mathrm{ST}(\delta; w, e)$ the subset of $\mathrm{ST}(\delta)$ consisting of spanning trees such that $w$ lies to the left of $\gamma_\delta^M$ and $e$ lies to the right of $\gamma_\delta^M$. We denote by $\mathrm{SF}_2(\delta; n, s)$ the subset of $\mathrm{SF}_2(\delta)$ consisting of spanning forests such that $n$ lies in the same tree as $(c_\delta d_\delta)$ and $s$ lies in the same tree as $(a_\delta b_\delta)$. 
Deleting the edge $\{n, s\}$ induces a bijection from $\mathrm{ST}(\delta; w, e)$ (left) to $\mathrm{SF}_2(\delta; n, s)$ (right).}
\end{figure}

\begin{lemma}\label{lem::holo_cvg}
Fix a quad $(\Omega; a, b, c, d)$. 
Suppose that a sequence of medial domains $(\Omega_\delta^\diamond; a_\delta^\diamond, b_\delta^\diamond, c_\delta^\diamond, d_\delta^\diamond)$ converges to $(\Omega; a, b, c, d)$  in the Carath\'{e}odory sense as $\delta\rightarrow 0$. 
Let $K>0$ be the conformal modulus of the quad $(\Omega; a, b, c, d)$, and let $f$ be the conformal map from $\Omega$ onto the rectangle $(0,1)\times (0, \ii K)$ which sends $(a, b, c, d)$ to $(0, 1, 1+\ii K, \ii K)$. 
Then the discrete holomorphic function $f_{\delta}$ in Lemma~\ref{lem::holomorphic_observable} (regarded as a function on $\Omega_\delta$ by interpolating linearly among vertices) converges to $f$ locally uniformly as $\delta\to 0$. 
\end{lemma}
We emphasize that, in Lemma~\ref{lem::holo_cvg}, we do not need that $\partial\Omega$ is $C^1$ and simple. We only assume that $\partial\Omega$ is locally connected and we only require the convergence of polygons in the Carath\'{e}odory sense. 
\begin{proof}
We claim that $ \left\{\frac{\left|\mathrm{SF}_2(\delta)\right|}{|\mathrm{ST}(\delta)|}\right\}_{\delta>0}$ is uniformly bounded. Assuming  this, for any sequence $\delta_n\to 0$, there exists a subsequence, still denoted by $\delta_n$, and a constant $\tilde{K}$ such that 
\[u_{\delta_n}\to u, \quad v_{\delta_n}\to v, \quad \text{locally uniformly;}\quad\text{and} \quad \frac{\left|\mathrm{SF}_2(\delta_n)\right|}{|\mathrm{ST}(\delta_n)|}\to \tilde{K},\quad \text{as }n\to\infty. \]
From the  boundary values of $u_{\delta}$ and $v_{\delta}$, the discrete harmonicity of $u_\delta$ and $v_\delta$, and Beurling estimate (for instance, see~\cite[Proposition 2.11]{ChelkakSmirnovDiscreteComplexAnalysis}), it is clear that $u=1$ on $(bc)$ and $u=0$ on $(da)$, and that $v=1$ on $(cd)$ and $v=0$ on $(ab)$. 
By standard argument, the limit of discrete holomorphic function is holomorphic, see for instance the first step in the proof of Lemma~\ref{lem::harmonic_cvg}. In other words, $g:=u+\ii\tilde{K}v$ is holomorphic on $\Omega$.
Moreover, if we fix a conformal map $\xi$ from $\Omega$ onto $\U$, then the boundary data of $u\circ\xi^{-1}$ is as follows: $u\circ\xi^{-1}=1$ on $(\xi(b)\xi(c))$, and $u=0$ on $(\xi(d)\xi(a))$, and $\partial_n (u\circ\xi^{-1})=0$ on $(\xi(a)\xi(b))\cup (\xi(c)\xi(d))$.
Such boundary data uniquely determines the bounded harmonic function $u$, see Lemma~\ref{lem::harmonic_unique}.
Note that $\Re f$ has the same boundary data, thus $u=\Re f$. Consequently, $g$ and $f$ differ up to a constant. Since $v$ and $\Im f$ both equal to $0$ on $(ab)$, we have $g=f$. This also implies that $\tilde{K}=K$.

It remains to show that $ \left\{\frac{\left|\mathrm{SF}_2(\delta)\right|}{|\mathrm{ST}(\delta)|}\right\}_{\delta>0}$ is uniformly bounded. If this is not the case, there exists a sequence $\delta_n\to 0$ such that $\frac{\left|\mathrm{SF}_2(\delta_n)\right|}{|\mathrm{ST}(\delta_n)|}\to \infty$. 
Then, by the same argument as above, the function $\frac{|\mathrm{ST}(\delta_{n})|}{|\mathrm{SF}_{2}(\delta_{n})|}f_{\delta_n}$ converges to a holomorphic map $h$ on $\Omega$ locally uniformly and $h$ extends continuously to $\overline{\Omega}$. In such case, $\Re h=0$ on $\Omega$, and hence $h$ has to be a constant. But $\Im h=1$ on $(cd)$ and $\Im h=0$ on $(ab)$. This is a contradiction. Therefore, $ \left\{\frac{\left|\mathrm{SF}_2(\delta)\right|}{|\mathrm{ST}(\delta)|}\right\}_{\delta>0}$ is uniformly bounded and we complete the proof.
\end{proof}

As a consequence of Lemmas~\ref{lem::holomorphic_observable} and~\ref{lem::holo_cvg}, we see that 
$\frac{|\mathrm{SF}_2(\delta)|}{|\mathrm{ST}(\delta)|}\to K$ as $\delta\to 0$.
This is a special case of~\cite[Theorem~1.1]{KenyonWilsonBoundaryPartitionsTreesDimers} for the grove with two nodes. 
\subsection{Proof of Theorem~\ref{thm::ust_hsle}}
\label{subsec::ust_hsle}
We fix the following notations in this section.
\begin{itemize}
\item Fix a quad $(\Omega; a, b, c, d)$. Let $K$ be its conformal modulus and denote by $f_{(\Omega; a, b, c, d)}$ the conformal map from $\Omega$ onto $(0,1)\times (0,\ii K)$ sending $(a, b, c, d)$ to $(0,1,1+\ii K, \ii K)$. 
\item Suppose that a sequence of medial quads $(\Omega_\delta^\diamond; a_\delta^\diamond, b_\delta^\diamond, c_\delta^\diamond, d_\delta^\diamond)$ and $(\Omega; a, b, c, d)$ satisfy the assumptions in Theorem~\ref{thm::ust_hsle}. 
We choose conformal maps $\phi_{\delta}:\Omega^{\diamond}_{\delta}\to \HH$ with $\phi_{\delta}(a^{\diamond}_{\delta})=0,\phi_{\delta}(d^{\diamond}_{\delta})=\infty$ and  $\phi:\Omega\to \HH$ with $\phi(a)=0,\phi(d)=\infty$ such that $\phi_{\delta}^{-1}$ converges to $\phi^{-1}$ uniformly on $\overline\HH$. (See~\cite[Corollary 2.4]{Pommerenke}. See also~\cite[the last paragraph above Proposition 4.5]{LawlerSchrammWernerLERWUST}).

\item For a primal quad $(\Omega_{\delta}; a_{\delta}, b_{\delta}, c_{\delta}, d_{\delta})$, we denote by $f^{\delta}_{(\Omega_{\delta}; a_{\delta}, b_{\delta}, c_{\delta}, d_{\delta})}$ the discrete observable as defined in Lemma~\ref{lem::holomorphic_observable} and we linearly interpolate it among vertices. Recall that $f^{\delta}_{(\Omega_{\delta}; a_{\delta}, b_{\delta}, c_{\delta}, d_{\delta})}$ converges to $f_{(\Omega; a, b, c, d)}$ locally uniformly due to Lemma~\ref{lem::holo_cvg}. 
\item The uniform spanning tree $\LT_{\delta}$ and the Peano curve $\eta^L_{\delta}$ are defined in the same way as in Section~\ref{subsec::ust_quads}. 
Define $\hat\eta_{\delta}^{L}:=\phi_{\delta}(\eta_{\delta}^{L})$ and parameterize $\hat\eta_{\delta}^{L}$ by the half-plane capacity and parameterize $\eta_{\delta}^{L}$ so that $\hat{\eta}_{\delta}^L(t)=\phi_{\delta}(\eta_{\delta}^L(t))$.
\item For $\hat{\eta}_{\delta}^L$, denote by $(W^{\delta}_t, t\ge 0)$ the driving function of $\hat\eta^{L}_{\delta}$ and by $(g^{\delta}_{t}, t\ge 0)$ the corresponding conformal maps. 
Define $X_t^\delta:=g_t^\delta(\phi_{\delta}(b_{\delta})), Y_t^\delta:=g_t^\delta(\phi_{\delta}(c_{\delta}))$. Denote by $\mathcal{F}_t^\delta$ the filtration generated by $(W_s^\delta, 0\le s\le t)$. 
Let $K_{\delta}(t)$ be the modulus of the quad $\left(\HH; W_t^\delta, X_t^\delta, Y_t^\delta, \infty\right)$ and define $f^{\delta}_{t}$ to be the conformal map from $\HH$ onto $(0,1)\times (0,\ii K_{\delta}(t))$ sending $\left(W_t^\delta, X_t^\delta, Y_t^\delta, \infty\right)$ to $(0,1,1+\ii K_{\delta}(t),\ii K_{\delta}(t))$. 
\item For $\eta_{\delta}^L$, let $\tau_\delta$ be the first time that $\eta_{\delta}^{L}$ hits $(c^{\diamond}_{\delta}d^{\diamond}_{\delta})$. For every $t<\tau_\delta$,  the slit domain $\Omega^{\diamond}_\delta(t)$ is defined as the component of $\Omega_\delta^\diamond\backslash\eta^{L}_{\delta}[0,t]$  that contains $c_\delta^\diamond$ and $d_\delta^\diamond$ on the boundary. 
We define $(\Omega_\delta(t);a_{\delta}(t),b_{\delta}(t),c_{\delta},d_{\delta})$ to be the primal discrete quad 
as follows: The domain $\Omega_{\delta}(t)$ is the primal domain associated to $\Omega_{\delta}^{\diamond}(t)$. The point $a_{\delta}(t)$ is the primal vertex nearest to $\eta^L_{\delta}(t)$. The definition of the point $b_{\delta}(t)$ is a little complicated: if $\eta^L_{\delta}[0,t]$ does not hit $(b_{\delta}^{\diamond}c_{\delta}^{\diamond})$, then $b_{\delta}(t)=b_{\delta}$; if $\eta^L_{\delta}[0,t]$ hits $(b_{\delta}^{\diamond}c_{\delta}^{\diamond})$, then $b_{\delta}(t)$ is the  primal vertex nearest to  $\eta^L_{\delta}[0,t]\cap (b_{\delta}^{\diamond}c_{\delta}^{\diamond})$. The boundary conditions for $(\Omega_\delta(t);a_{\delta}(t),b_{\delta}(t),c_{\delta},d_{\delta})$ are inherited from $(\Omega_{\delta}; a_{\delta}, b_{\delta}, c_{\delta}, d_{\delta})$ and $\eta^L_{\delta}[0,t]$: the boundary arc $(a_{\delta}(t)b_{\delta}(t))$ is wired and the boundary arc $(c_{\delta}d_{\delta})$ is wired. 
\end{itemize}


\begin{lemma}\label{lem::limit}
For $z\in\Omega_{\delta}\cup\Omega_{\delta}^*$, denote by $\tau^{\delta}_{z}:=\inf\{t:z\notin\Omega^{\diamond}_{\delta}(t)\}$. Then, the process given by the observables
\[(f^{\delta}_{(\Omega_\delta(t); a_{\delta}(t),b_{\delta}(t), c_\delta, d_\delta)}(z), t\ge 0)\]  is a local martingale up to $\tau^{\delta}_{z}\wedge\tau_{\delta}$ with respect to the filtration $(\mathcal{F}_t^\delta, t\ge 0)$. Moreover, for every $\epsilon>0$ and for any compact subset $A$ on $\Omega$, one has almost surely the following  convergence of the conformal maps:
\begin{align}\label{eqn::discretemart_cvg}
\mathbb{E}\left[f^{\delta}_{\tau_2^\delta}\left(g^\delta_{\tau_2^\delta}(\phi_{\delta}(z))\right)-f^{\delta}_{\tau_1^\delta}\left(g^\delta_{\tau_1^\delta}(\phi_{\delta}(z))\right)\mid\mathcal{F}_{\tau_1}^\delta\right]\xrightarrow{\delta\rightarrow 0}0,
\end{align}
uniformly for $z\in A$ and uniformly for any stopping times $0<\tau_1^\delta<\tau_2^\delta<\tau_{\epsilon}^\delta\wedge\tau_{A,\epsilon}^{\delta}$, where $\tau_{A,\epsilon}^{\delta}:=\inf\{t:\dist(\eta_{\delta}^{L}([0,t]),A)=\epsilon\}$ and $\tau^{\delta}_{\epsilon}$ is the first time that $\eta_{\delta}^{L}$ hits the $\epsilon$-neighborhood of $(b_{\delta}^{\diamond}d_{\delta}^{\diamond})$.
\end{lemma}
\begin{proof}
We first show that $(f^\delta_{(\Omega_\delta(t); a_{\delta}(t),b_{\delta}(t), c_\delta, d_\delta)}(z), t\ge 0)$ is a local martingale and we will consider its real part and its imaginary part separately. We fix two stopping times $\tau_1^\delta<\tau_2^\delta<\tau^\delta_z\wedge\tau_\delta$. 
Define $u_{\delta}^{(i)}$ and $v_{\delta}^{(i)}$ similarly as $u_\delta$ and $v_\delta$ in the primal quad $(\Omega_\delta(\tau_i^\delta); a_{\delta}(\tau_i^\delta),b_{\delta}(\tau_i^\delta), c_\delta, d_\delta)$ for $i=1,2$. 

For the real part, for every $z\in\Omega_\delta^*$, we have
\begin{align*}
u_{\delta}^{(1)}(z)&=\E\left[z\text{ lies to the right of }\eta_\delta^L\cond \LF_{\tau_1^\delta}^\delta\right]=\E\left[\E\left[z\text{ lies to the right of }\eta_\delta^L\cond \LF_{\tau_2^\delta}^\delta\right]\cond \LF_{\tau_1^\delta}^\delta\right]=\E\left[u_{\delta}^{(2)}(z)\cond \LF_{\tau_1^\delta}^\delta\right].
\end{align*}
This implies that $(\Re f^{\delta}_{(\Omega_\delta(t); a_{\delta}(t),b_{\delta}(t), c_\delta, d_\delta)}(z), t\ge 0))$ is a martingale  up to $\tau^{\delta}_{z}\wedge\tau_{\delta}$. 

For the imaginary part, define $\mathrm{SF}_2(\tau_i^\delta)$ and $\mathrm{ST}(\tau_i^\delta)$ similarly as $\mathrm{SF}_2(\delta)$ and $\mathrm{ST}(\delta)$ in the primal quad $(\Omega_\delta(\tau_i^\delta); a_{\delta}(\tau_i^\delta),b_{\delta}(\tau_i^\delta), c_\delta, d_\delta)$ for $i=1,2$. Define $\mathrm{SF}_2(\tau_i^\delta;z)$ to be the subset of $\mathrm{SF}_2(\tau_i^\delta)$ such that $z$ lies in the same tree as the wired arc $(c_\delta d_\delta)$ for $i=1,2$. Then, for every $z\in\Omega_\delta$, we have
\begin{align*}
\frac{|\mathrm{SF}_2(\tau_1^\delta)|}{|\mathrm{ST}(\tau_1^\delta)|}v_{\delta}^{(1)}(z)
&=\frac{|\mathrm{SF}_2(\tau_1^\delta;z)|}{|\mathrm{ST}(\tau_1^\delta)|}=\E\left[\frac{|\mathrm{SF}_2(\tau_2^\delta;z)|}{|\mathrm{ST}(\tau_2^\delta)|}\cond\LF_{\tau_1^\delta}^{\delta}\right]
=\E\left[\frac{|\mathrm{SF}_2(\tau_2^\delta)|}{|\mathrm{ST}(\tau_2^\delta)|}v_{\delta}^{(2)}(z)\cond\LF_{\tau_1^\delta}^{\delta}\right].
\end{align*}
This implies that $(\Im f^{\delta}_{(\Omega_\delta(t); a_{\delta}(t),b_{\delta}(t), c_\delta, d_\delta)}(z), t\ge 0))$ is a martingale  up to $\tau^{\delta}_{z}\wedge\tau_{\delta}$. We get the first conclusion.

It remains to show~\eqref{eqn::discretemart_cvg}. 
We will prove that $f^{\delta}_{\tau_{i}^{\delta}}\circ g^{\delta}_{\tau_{i}^{\delta}}\circ\phi_{\delta}-f^{\delta}_{(\Omega_\delta(\tau_{i}^{\delta}); a_\delta(\tau_i^{\delta}), b_\delta, c_\delta, d_\delta)}$ converges to $0$ uniformly on $A$ for $i=1,2$.
If this is not the case, there exists a sequence $\delta_n\to 0$ and a sequence $z_{\delta_n}\in A$ such that 
\begin{equation}\label{eqn::contradiction1}
f^{\delta_n}_{\tau_i^{\delta_n}}\left(g^{\delta_n}_{\tau_i^{\delta_n}}(\phi_{\delta_n}(z_{\delta_n}))\right)-f^{\delta_n}_{(\Omega_{\delta_n}(\tau_{i}^{{\delta_n}}); a_\delta(\tau_i^{{\delta_n}}), b_{\delta_n}, c_{\delta_n}, d_{\delta_n})}(z_{\delta_n})\quad\text{ does not converge to }0.
\end{equation}
Note that, by Carath\'{e}odory kernel theorem (for instance, see~\cite[Theorem 1.8]{Pommerenke}), there exists a subsequence, still denoted by $\delta_n$, such that  $(\Omega_{\delta_{n}}(\tau_{i}^{\delta_{n}}); a_{\delta_{n}}(\tau_{i}^{\delta_{n}}), b_{\delta_{n}}, c_{\delta_{n}}, d_{\delta_{n}})$ converges to a quad $(\Omega_{i};a_{i},b_{i},c_{i},d_{i})$ in the Carath\'{e}odory sense for $i=1,2$. (Here we need definitions of $\tau_{i}^{\delta_{n}}$ to guarantee the limit is still a quad.) By Lemma~\ref{lem::holo_cvg}, the sequence $f^{\delta_n}_{(\Omega_{\delta_{n}}(\tau_{i}^{\delta_n}); a_{\delta_{n}}(\tau_i^{\delta_n}), b_{\delta_{n}}, c_{\delta_{n}}, d_{\delta_{n}})}$ converges to $f_{(\Omega_{i};a_{i},b_{i},c_{i},d_{i})}$ uniformly on  $A$.
Note that $f^{\delta_n}_{\tau_i^{\delta_n}}\circ g^{\delta_n}_{\tau_i^{\delta_n}}\circ\phi_{\delta_n}$ is the conformal map from $\Omega^{\diamond}_{\delta_{n}}(\tau_{i}^{\delta_n})$ to a rectangle which maps $(\eta^{L}_{\delta_n}(\tau_i^{\delta_n}), b_{\delta_{n}}^{\diamond}, c_{\delta_{n}}^{\diamond}, d_{\delta_{n}}^{\diamond})$ to the four corners. By the Carath\'{e}odory convergence and the description of $f_{(\Omega_{i};a_{i},b_{i},c_{i},d_{i})}$ in Lemma~\ref{lem::holo_cvg}, the sequence $f^{\delta_n}_{\tau_i^{\delta_n}}\circ g^{\delta_n}_{\tau_i^{\delta_n}}\circ\phi_{\delta_n}$ converges to $f_{(\Omega_{i};a_{i},b_{i},c_{i},d_{i})}$ on $A$ as well.
This is a contradiction with~\eqref{eqn::contradiction1}. To obtain~\eqref{eqn::discretemart_cvg}, we still need to show that there exists a constant $M=M(\eps)$, such that
\[\left|f^{\delta}_{\tau_i^\delta}\left(g^\delta_{\tau_i^\delta}(\phi_{\delta}(z))\right)\right|\le M \quad\text{and }\left|f^{\delta_n}_{(\Omega_{\delta_n}(\tau_{i}^{{\delta_n}}); a_\delta(\tau_i^{{\delta_n}}), b_{\delta_n}, c_{\delta_n}, d_{\delta_n})}(z)\right|\le M\quad\text{uniformly for }z\in A\text{ and }i=1,2.\]
This can also be obtained by contradiction similarly as above argument. (Here we need the definitions of the stopping times: $0<\tau_1^\delta<\tau_2^\delta<\tau_{\epsilon}^\delta\wedge\tau_{{A},\epsilon}^{\delta}$.) Combining with~\eqref{eqn::contradiction1} and dominated convergence theorem, we get~\eqref{eqn::discretemart_cvg}.
\end{proof}


From Proposition~\ref{prop::tightness} and Proposition~\ref{prop::tightness_LERW}, both the law of $\{\eta_\delta^L\}_{\delta>0}$ and the law of $\{\gamma^M_{\delta}\}_{\delta>0}$ are tight in metric~\eqref{eqn::curves_metric}. In the following, we will consider the distributions of the limits of  $\eta^L_{\delta}$ and $\gamma^M_{\delta}$. Assume the same setup as in Theorem~\ref{thm::cvg_triple} and fix the following notations. 
\begin{itemize}
\item Suppose $(\eta^L; \gamma^M)$ is any subsequential limit of $\{(\eta^L_{\delta}; \gamma^M_{\delta})\}_{\delta>0}$. We may assume $\eta^L_{\delta_n}\to\eta^L$ in law and $\gamma^M_{\delta_n}\to \gamma^M$ in law, both in metric~\eqref{eqn::curves_metric}. By Skorokhod's representation theorem, we may couple them together such that $\eta_{\delta_n}^L\to \eta^L$ and $\gamma_{\delta_n}^M\to \gamma^M$ almost surely, both in metric~\eqref{eqn::curves_metric}. We will show that the conditional law of $\eta^L$ given $\gamma^M$ is $\SLE_8$ in $\Omega^L$ in Lemma~\ref{lem::etaLgivengammaM}.  

\item Define $\hat{\eta}^L=\phi(\eta^L)$ and parameterize $\hat{\eta}^L$ by the half-plane capacity and parameterize $\eta^L$ so that $\hat{\eta}^L(t)=\phi(\eta^L(t))$. We will show that $\hat{\eta}^L$ has continuous driving function in Corollary~\ref{cor::continuousdriving} and denote its driving function by $(W_t, t\ge 0)$. We will derive the law of $(W_t, t\ge 0)$ in Lemma~\ref{lem::drivingfunction}. 
\end{itemize}

\begin{lemma}\label{lem::etaLgivengammaM}
The conditional law of $\eta^L$ given $\gamma^M$ is $\SLE_8$ in $\Omega^L$ from $a$ to $d$. 
\end{lemma}
\begin{proof}

We denote by $\Omega^{L}$ the connected component of $\Omega\setminus\gamma^{M}$ which contains $[da]$ on its boundary. This is well-defined since we have $\gamma^M\cap [da]=\emptyset$ almost surely due to Proposition~\ref{prop::tightness_LERW}. 
Denote by $\Omega_{\delta_n}^{L}$ the connected component of $\Omega_{\delta_n}\setminus\gamma_{\delta_n}^{M}$ which contains $[d_{\delta_n}a_{\delta_n}]$ on its boundary. Denote by $\Omega^{\diamond,L}_{\delta_n}$ the the medial graph associated with $\Omega_{\delta_n}^{L}$. Since $\gamma_{\delta_n}^{M}\to\gamma^{M}$ as curves, the medial (random) Dobrushin domains $(\Omega^{\diamond,L}_{\delta_n};a_{\delta_n}^{\diamond},d_{\delta_n}^{\diamond})$ converges to $(\Omega^{L};a,d)$ in the following sense: there exists a constant $C>0$, 
\[d((d_{\delta_n}^{\diamond}a_{\delta_n}^{\diamond}),(da))\le C\delta_n,\quad d((a_{\delta_n}^{\diamond}d_{\delta_n}^{\diamond}),(ad))\to 0,\quad \text{as }n\to \infty,\]
where $d$ is the metric~\eqref{eqn::curves_metric}. Note that in the proof of Theorem~\ref{thm::ust_Dobrushin}, the conditions that $\partial\Omega$ is $C^{1}$ and simple and the convergence of boundaries are only used in~\cite[Section 4.3]{LawlerSchrammWernerLERWUST} to get the uniform continuity, that is, the tightness of the discrete Peano curves in metric~\eqref{eqn::curves_metric}. See the begining of~\cite[Section 4.3]{LawlerSchrammWernerLERWUST}. These conditions are used to ensure that $\mathbf{trunk}\cap\mathbf{trunk}^{*}=\emptyset$ almost surely. We will summarize the proof of tightness in the proof the Proposition~\ref{prop::tightness}.
Although we do not have $C^1$ regularity on $\partial\Omega^{L}$ neither the required convergence of domains, we already have $\mathbf{trunk}\cap\mathbf{trunk}^{*}=\emptyset$ almost surely due to Lemma~\ref{lem::tightness_trunk}.
This is because $\mathbf{trunk}_{\delta_n}$ in the primal graph associated with the medial Dobrushin domain $(\Omega^{\diamond,L}_{\delta_n};a_{\delta_n}^{\diamond},d_{\delta_n}^{\diamond})$ is a subset of $\mathbf{trunk}_{\delta_n}$ in $(\Omega_{\delta_n};a_{\delta_n},b_{\delta_n},c_{\delta_n},d_{\delta_n})$ and $\mathbf{trunk}^{*}_{\delta_n}$ in the dual graph associated with the medial Dobrushin domain $(\Omega^{\diamond,L}_{\delta_n};a_{\delta_n}^{\diamond},d_{\delta_n}^{\diamond})$ is a subset of $\mathbf{trunk}^{*}_{\delta_n}$ in $(\Omega^*_{\delta_n};a^*_{\delta_n},b^*_{\delta_n},c^*_{\delta_n},d^*_{\delta_n})$. In the part of the proof of Theorem~\ref{thm::ust_Dobrushin} to obtain the convergence of the driving functions, the authors only need two requirements:
\begin{itemize}
\item
the radius of the domains seen from an interior point is uniformly large if we rescale the discrete domain so that they are subgraphs of $\Z^2$;
\item
the harmonic measure of the wired arc seen from the same interior point belongs to $(\eps,1-\eps)$ for $\eps>0$. 
\end{itemize}
See more details in~\cite[Theorem 4.4]{LawlerSchrammWernerLERWUST}. Note that in our case, we have the convergence of domains in the Carath\'eodory sense, which satisfies the above two assumptions. Thus, by the same proof of Theorem~\ref{thm::ust_Dobrushin}, we have that $\eta^{L}$ is $\SLE_{8}$ in $\Omega^{L}$ from $a$ to $d$.
\end{proof}
\begin{corollary}\label{cor::continuousdriving}
The process $\hat{\eta}^L=\phi(\eta^L)$ has continuous driving function when parameterized by the half-plane capacity.
Moreover, when parameterized by the half-plane capacity, $\hat{\eta}^L_{\delta_n}\to \hat{\eta}^L$ locally uniformly and 
the driving functions $(W^{\delta_n}_t, t\ge 0)\to (W_t, t\ge 0)$ locally uniformly.  
\end{corollary}
\begin{proof}
We only need to show that $\hat{\eta}^L$ has continuous driving function when parameterized by the half-plane capacity. The rest of the statement is true due to~\cite[Proposition~4.10]{LupuWuLevellineGFF}. 
Recall that $\phi$ is a conformal map from $\Omega$ onto $\HH$. Since $\partial\Omega$ is simple, $\phi$ can be extended injectively and continuously to $\partial\Omega$. This implies that $\hat{\eta}^L=\phi(\eta^L)$ is a continuous curve.
Recall that $\Omega^L$ is the connected component of $\Omega\setminus\gamma^M$ which contains $[da]$ on its boundary. 
Fix a conformal map $\phi^L$ from $\Omega^L$ onto $\HH$. We parameterize $\eta^L$ by $[0,+\infty)$ such that for every $t>0$, the half-plane capacity of $\phi^L(\eta^L[0,t])$ equals $t$. Note that under this parameterization, $\eta^L$ is a continuous function on $[0,+\infty)$.  
From Lemma~\ref{lem::etaLgivengammaM}, $\phi^L(\eta^L)$ is an $\SLE_8$. Thus, we have the following two observations:
\begin{itemize}
\item
For every $0<t<T$, we have $\eta^L(t,T)$ is contained in the closure of the connected component of $\Omega^L\setminus\eta^L(0,t)$. This implies that $\phi(\eta^L(t,T))$ is contained in the closure of the unbounded connected component of $\HH\setminus\phi(\eta^L(0,t))$.
\item
The set $\{s\in (t,T): \eta^L(s)\in \eta^L[0,t]\cup\partial\Omega^L\}$ has empty interior. In paticular, this implies that the set $\{s\in (t,T): \phi(\eta^L(s))\in \phi(\eta^L[0,t])\cup\R\}$ has empty interior.
\end{itemize} 
Thus, by~\cite[Proposition 4.3]{Lawler-book} and~\cite[Proposition 6.12]{MillerSheffieldIG1}, almost surely, $\hat{\eta}^L=\phi(\eta^L)$ is driven by a continuous function after being reparameterized by the half-plane capacity. This completes the proof.
\end{proof}

For $x<y<w$, define 
\begin{align}\label{eqn::driftfunction}
\Theta(x, y, w):=\frac{2}{w-x}+\frac{-2}{w-y}-8\frac{y-x}{(y-w)^2}\frac{F'\left(\frac{x-w}{y-w}\right)}{F\left(\frac{x-w}{y-w}\right)},\quad\text{where }F(z)=\hF\left(\frac{1}{2}, \frac{1}{2}, 1; z\right).
\end{align}
Note that $F$ is the hypergeometric function in~\eqref{eqn::hyper_def} with $\kappa=8, \nu=0$. 
\begin{lemma}\label{lem::drivingfunction}
For every $\epsilon>0$, we denote by $\tau_{\epsilon}$ the first time that $\eta^L$ hits the $\epsilon$-neighbourhood of $(bd)$. Then, the law of the driving function $(W_t, t\ge 0)$ of $\hat\eta^{L}$ is given by the following SDEs up to $\tau_{\eps}$:
\begin{align}\label{eqn::hsle}
\begin{cases}
dW_t=\sqrt{8}dB_t+\Theta(X_{t}, Y_{t}, W_{t})dt,\quad W_{0}=0;\\
dX_{t}=\frac{2dt}{X_{t}-W_t},\quad X_{0}=\phi(b);\\
dY_{t}=\frac{2dt}{Y_{t}-W_t},\quad Y_{0}=\phi(c); 
\end{cases}
\end{align}
where $(B_t, t\ge 0)$ is one-dimensional Brownian motion starting from 0 and $\Theta$ is defined in~\eqref{eqn::driftfunction}.  
\end{lemma}
\begin{proof}
Recall that $\hat\eta^L_{\delta_n}=\phi_{\delta_n}(\eta_{\delta_n}^{L})$ and $\hat\eta^L=\phi(\eta^L)$ are parameterized by half-plane capacity. From Corollary~\ref{cor::continuousdriving}, we see that $\hat\eta^L$ is driven by a continuous function $W$; and that $\hat\eta^L_{\delta_n}$ converges to $\hat\eta^L$ locally uniformly. We define $T_{M}:=\inf\{t: \hat\eta^{L}[0,t]\cap\partial B(0,M)\neq\emptyset\}$ for every $M>0$ and $\tau_{\epsilon',z}:=\inf\{t:\hat\eta^{L}[0,t]\cap\partial B(z,\epsilon')\neq\emptyset\}$ for every $z\in\HH$ and $\epsilon'>0$. It suffices to prove that~\eqref{eqn::hsle} holds up to $\tau_{\epsilon}\wedge T_{M}$ since we may get the result by letting $M\to\infty$. We define $T^{\delta_n}_{M}$ and $\tau^{\delta_n}_{\epsilon',z}$ similarly for $\hat\eta_{\delta_n}^{L}$. We may assume $T^{\delta_n}_{M}\to T_{M}$, $\tau^{\delta_n}_{\epsilon',z}\to\tau_{\epsilon',z}$ and $\tau^{\delta_n}_{\epsilon}\to\tau_{\epsilon}$ by considering a continuous modification, see details in~\cite[Appendix B]{KarrilaMultipleSLELocalGlobal} and~\cite{KarrilaUSTBranches}.
Then, Lemma~\ref{lem::limit} implies that 
\[M_{t}(z):=f_{(\HH; 0,X_t-W_{t}, Y_t-W_{t},\infty)}(g_t(z)-W_t)\] 
is a martingale up to $\tau_{\epsilon}\wedge T_{M}\wedge\tau_{\epsilon',z}$. 

First, we prove that $W_t$ is a semimartingale (similar argument already appeared in~\cite{KarrilaUSTBranches}). Define $g(w,x,y; z):=f_{(\HH; 0, x-w, y-w,\infty)}(z-w)$ on $\{(w,x,y)\in\R^3:w<x<y\}\times\HH$. 
Note that $\partial_w g(w,x,y;\cdot)$ is also an analytic function on $\HH$. We now  show that the zero set of $\partial_w g(w,x,y;\cdot)$ is isolated.
Otherwise, $\partial_w g(w,x,y;\cdot)$ equals a constant $C$ on $\HH$. By letting $z\to w$, we have  $C=0$. Thus $g(w,x,y;\cdot)$ is independent of $w$. This contradicts  the fact that $g(w,x,y;\cdot)$ is the conformal map from $\HH$ onto $(0,1)\times (0,\ii K)$ sending $(w, x, y,\infty)$ to $(0,1,1+\ii K, \ii K)$. Thus,
for every $w<x<y$, the zero set of $\partial_w g(w,x,y; \cdot)$ is isolated. Consequently,
there exists $z\in\HH$ such that $\partial_w g(w,x,y; z)\neq 0$. By continuity, $\partial_w g(\cdot,\cdot,\cdot;z)\neq 0$ on an interval containing $(w,x,y)$. Combining with implicit function theorem, there exists a smooth function $\psi$ such that $w=\psi(x,y,z,g)$ on an open neighborhood of $(x,y,z,g)$. 
Consequently, there exists  a sequence of triple sets $\{(O_i;z_i,\psi_i)\}_{i\ge 1}$ satisfying: 
\begin{itemize}
\item $O_i$ is an open set of $\R^3$ for each $i\geq 1$; 
\item $\cup_i O_i=\{(w, x,y)\in\R^3:w<x<y\}$;
\item For each $O_i$, there exist $z_i\in\HH$ and a smooth function $\psi_i$ such that 
\[w=\psi_i(x,y,z_i,g), \text{ for all }(w, x,y)\in O_i.\]
\end{itemize}
Define a sequence of stopping times $\{T_i\}_{i\geq 1}$ as follows: Define $T_1:=0$ and define $(O_{T_1};z_{T_1},\psi_{T_1})$ to be any element in $\{(O_i;z_i,\psi_i)\}_{i\ge 1}$ such that $O_{T_1}$ contains $(0,x,y)$. Suppose that $T_n$ and $(O_{T_n};z_{T_n},\psi_{T_n})$ are well-defined, we set 
\[T_{n+1}:=\inf\{t>T_n:(W_t, X_t,Y_t)\notin O_{T_n}\}\]
and define $(O_{T_{n+1}};z_{T_{n+1}},\psi_{T_{n+1}})$ to be any element in $\{(O_i;z_i,\psi_i)\}_{i\ge 1}$ such that $O_{T_{n+1}}$ contains the point $(W_{T_{n+1}}, X_{T_{n+1}}, Y_{T_{n+1}})$. Then, we have
\[W_t=\sum_{i=1}^{\infty}\one_{\{T_i\le t< T_{i+1}\}}\psi_{T_i}(X_t,Y_t,g_t(z_{T_i}),M_t(z_{T_i})).\]
This implies that $W_t$ is a semimartingale.

Next, let us calculate the drift term of $M_t(z)$. From Schwarz-Christorffel formula (see e.g.~\cite[Chapter~6-Section~2.2]{AhlforsComplexAnalysis}), we have, for $0<x<y$ and for $z\in\HH$, 
\begin{align*}
f_{(\HH; 0, x, y,\infty)}(z)=\frac{\int_0^{z/x} \left(s(s-1)(s-\frac{y}{x})\right)^{-1/2}ds}{\int_0^{1} \left(s(s-1)(s-\frac{y}{x})\right)^{-1/2}ds}.
\end{align*}
Denote by $\LK$ the elliptic integral of the first kind~\eqref{eqn::elliptic_def}. 
By changing of variable $s=\sin^2\theta$, we have
\begin{equation}\label{eqn::exactform}
f_{(\HH; 0, x, y,\infty)}(z)=\frac{\LK(\arcsin\sqrt{z/x}, x/y)}{\LK(x/y)}. 
\end{equation}
Therefore, 
\begin{align}\label{eqn::mart_aux1}
M_t(z)=\frac{\LK\left(S_t, U_t\right)}{\LK\left(U_t\right)},\quad\text{where}\quad S_t=\arcsin\sqrt{\frac{g_t(z)-W_t}{X_t-W_t}},\quad U_t=\frac{X_t-W_t}{Y_t-W_t}. 
\end{align}

We first calculate $dU_t$ and $dS_t$: 
\begin{equation}\label{eqn::mart_aux2}
\begin{split}
dU_t&=\frac{1}{(Y_t-W_t)^2}\left(\left(\frac{2}{U_t}-2U_t\right)dt +(X_t-Y_t)dW_t-(1-U_t)d\langle W\rangle_t\right),\\
dS_t&=\frac{\cot S_t}{(X_t-W_t)^2}\left((2+\cot^2 S_t)dt-\frac{1}{2}(X_t-W_t)dW_t-\frac{1}{8}\left(3+\cot^2 S_t\right)d\langle W\rangle_t\right). 
\end{split}
\end{equation}
Applying It\^{o}'s formula in~\eqref{eqn::mart_aux1}, we have
\begin{equation}\label{eqn::Ito}
dM_{t}(z)=\frac{1}{\LK(U_{t})}d\LK(S_{t},U_{t})-\frac{\LK(S_{t},U_{t})}{\LK^{2}(U_{t})}d\LK(U_{t})-\frac{1}{\LK^{2}(U_{t})}d\langle \LK(S,U),\LK(U)\rangle_{t}+\frac{\LK(S_{t},U_{t})}{\LK^{3}(U_{t})}d\langle \LK(U)\rangle_{t}.
\end{equation}
We denote by $L_{t}$ the drift term of $W_{t}$. Since the drift term of $M_{t}(z)$ is zero,  plugging~\eqref{eqn::mart_aux2} into~\eqref{eqn::Ito}, we have
\begin{equation}\label{eqn::drift}
\begin{split}
0=&\frac{\partial_{\varphi}\LK(S_{t},U_{t})}{\LK(U_{t})}\frac{\cot S_{t}}{(X_{t}-W_{t})^{2}}\left((2+\cot^{2}S_{t})dt-\frac{1}{2}(X_{t}-W_{t})dL_t-\frac{1}{8}(3+\cot^{2}S_{t})d\langle W\rangle_{t}\right)\\
&+\frac{\partial_{x}\LK(S_{t},U_{t})}{\LK(U_{t})}\frac{1}{(Y_{t}-W_{t})^{2}}\left(\left(\frac{2}{U_{t}}-2U_{t}\right)dt+(X_{t}-Y_{t})dL_{t}-(1-U_{t})d\langle W\rangle_{t}\right)\\
&+\frac{1}{8}\frac{\partial_{\varphi}^2\LK(S_{t},U_{t})}{\LK(U_{t})}\frac{\cot^{2}S_{t}}{(X_{t}-W_{t})^{2}}d\langle W\rangle_{t}+\frac{1}{2}\frac{\partial_{x}^2\LK(S_{t},U_{t})}{\LK(U_{t})}\frac{(X_{t}-Y_{t})^{2}}{(Y_{t}-W_{t})^{4}}d\langle W\rangle_{t}\\
&-\frac{1}{2}\frac{\partial_{x}\partial_{\varphi}\LK(S_{t},U_{t})}{\LK(U_{t})}\frac{(X_{t}-Y_{t})\cot S_{t}}{(Y_{t}-W_{t})^{2}(X_{t}-W_{t})}d\langle W\rangle_{t}\\
&-\frac{\LK(S_{t},U_{t})}{\LK^2(U_{t})}\frac{\partial_{x}\LK(U_{t})}{(Y_{t}-W_{t})^{2}}\left(\left(\frac{2}{U_{t}}-2U_{t}\right)dt+(X_{t}-Y_{t})dL_{t}-(1-U_{t})d\langle W\rangle_{t}\right)\\
&-\frac{1}{2}\frac{\LK(S_{t},U_{t})}{\LK^{2}(U_{t})}\frac{(X_{t}-Y_{t})^{2}}{(Y_{t}-W_{t})^{4}}\partial_{x}^2\LK(U_{t})d\langle W\rangle_{t}-\frac{\partial_{x}\LK(U_{t})\partial_{x}\LK(S_{t},U_{t})}{\LK^{2}(U_{t})}\frac{(X_{t}-Y_{t})^{2}}{(Y_{t}-W_{t})^{4}}d\langle W\rangle_{t}\\
&+\frac{\partial_{x}\LK(U_{t})\partial_{\varphi}\LK(S_{t},U_{t})}{\LK^{2}(U_{t})}\frac{(X_{t}-Y_{t})\cot S_{t}}{2(Y_{t}-W_{t})^{2}(X_{t}-W_{t})}d\langle W\rangle_{t}\\
&+\frac{\LK(S_{t},U_{t})}{\LK^{3}(U_{t})}\frac{(X_{t}-Y_{t})^{2}}{(Y_{t}-W_{t})^{4}}(\partial_{x}\LK(U_{t}))^{2}d\langle W\rangle_{t}.
\end{split}
\end{equation}
Note that
\[g_t(z)-W_t\to 0,\quad S_t\to 0,  \quad \sqrt{g_{t}(z)-W_{t}}\cot S_{t}\to \sqrt{X_{t}-W_{t}},\quad \text{as }z\to \hat{\eta}^L(t). \]
Combining with~\eqref{eqn::elliptic_derivatives} and the trivial facts $\partial_x\LK(\varphi, x), \partial^2_x\LK(\varphi, x)\to 0$ as $\varphi\to 0$, we have
\begin{align*}
&\LK(S_{t},U_{t})\to 0,\quad\partial_{\varphi}\LK(S_t, U_t)\to 1,\quad \partial_{x}\LK(S_{t},U_{t})\to 0,\\
&\frac{\partial_{\varphi}^2\LK(S_{t},U_{t})}{\sqrt{g_{t}(z)-W_{t}}}\to \frac{U_{t}}{\sqrt{X_{t}-W_{t}}},\quad\frac{\partial_{x}\partial_{\varphi}\LK(S_{t},U_{t})}{g_{t}(z)-W_{t}}\to\frac{1}{2(X_{t}-W_{t})},\quad \partial_{x}^2\LK(S_{t},U_{t})\to 0,\quad \text{as }z\to \hat{\eta}^L(t). 
\end{align*}
In the right hand-side of~\eqref{eqn::drift}, the leading term is of order $\cot^3 S_{t}$. 
Dividing~\eqref{eqn::drift} by $\cot^3 S_{t}$ and letting $z\to \hat{\eta}^L(t)$, we have
\begin{equation}\label{eqn::variation}
d\langle W\rangle_{t}=8dt.
\end{equation}
Plugging~\eqref{eqn::variation} into~\eqref{eqn::drift}, the leading term now is of order $\cot S_t$. Dividing~\eqref{eqn::drift} by $\cot S_{t}$ and letting $z\to \hat{\eta}^L(t)$, we have
\[-\frac{1}{X_{t}-W_{t}}\left(dt+\frac{1}{2}(X_{t}-W_{t})dL_{t}\right)+\frac{U_{t}}{X_{t}-W_{t}}dt+\frac{4\partial_{x}\LK(U_{t})}{\LK(U_{t})}\frac{X_{t}-Y_{t}}{(Y_{t}-W_{t})^{2}}dt=0.\]
By~\eqref{eqn::elliptic_hyper}, we have
\begin{equation}\label{eqn::correct_drift}
dL_{t}=\Theta(X_t, Y_t, W_t)dt. 
\end{equation}
Combining~\eqref{eqn::variation} and~\eqref{eqn::correct_drift}, we obtain the result. 
\end{proof}
\begin{proof}[Proof of Theorem~\ref{thm::ust_hsle}]
Recall that $\hat{\eta}^L_{\delta_n}=\phi_{\delta_n}(\eta^L_{\delta_n})$ and $\hat{\eta}^L=\phi(\eta^L)$ are parameterized by the half-plane capacity. From Corollary~\ref{cor::continuousdriving}, $\hat{\eta}^L_{\delta_n}\to \hat{\eta}^L$ locally uniformly as $n\to\infty$ almost surely.
Recall the definitions of $\mathbf{trunk}_{\delta}\left(\eps\right)$, $\mathbf{trunk}_{0}\left(\frac{1}{n}\right)$ and $\mathbf{trunk}$ for $\LT_{\delta}$ and $\mathbf{trunk}_{\delta}^*\left(\eps\right)$, $\mathbf{trunk}^{*}_{0}\left(\frac{1}{n}\right)$ and $\mathbf{trunk}^{*}$ for $\LT_{\delta}^*$ in Lemma~\ref{lem::tightness_trunk}. We may assume that (by subtracting a further subsequence), in such coupling, the sequence of trunks $\{\mathbf{trunk}_{\delta_{n}}^{*}\left(\frac{1}{m}\right)\}_n$ converges to $\mathbf{trunk}^{*}_{0}\left(\frac{1}{m}\right)$ in  Hausdorff distance for every $m\in\N$. 
We define $\tau:=\inf\{t>0:\eta^{L}[0,t]\cap(cd)\neq\emptyset\}$. By Lemma~\ref{lem::drivingfunction}, the driving function of $\hat\eta^{L}$ has the same law as the one for $\hSLE_{8}$ up to $\tau$. (Note that $\eta^L\cap [bc]=\emptyset$ almost surely.)
In order to show that $\hat\eta^{L}$ has the same law as $\hSLE_{8}$ as a whole process, 
it remains to analyze the continuity of the process as $t\to \tau$ and to derive the limit after the time $\tau$. 

Define 
$\tau_{\delta_n,\epsilon}:=\inf\{t>0:\dist(\eta_{\delta_{n}}^{L}(t),(c^{\diamond}_{\delta_{n}}d^{\diamond}_{\delta_{n}}))=\epsilon\}$.
Recall that $\tau_{\delta_n}$ is the first time that $\eta_{\delta_n}^{L}$ hits $(c_{\delta_n}^{\diamond}d_{\delta_n}^{\diamond})$. Firstly, we will show 
\begin{align}\label{eqn::hittingtime}
\lim_{n\to\infty}\tau_{\delta_n}=\tau \quad \text{almost surely}. 
\end{align} 
It is clear that
\[\lim_{\epsilon\to 0}\varlimsup_{n\to \infty}\tau_{\delta_n,\epsilon}\le \tau\le\varliminf_{n\to \infty}\tau_{\delta_n}\le\varlimsup_{n\to \infty}\tau_{\delta_n}\quad \text{almost surely}. \]
Denote by $T=\lim_{\epsilon\to 0}\varlimsup_{n\to \infty}\tau_{\delta_n,\epsilon}$. From Lemma~\ref{lem::subseqlimit_avoids_c}, we have $\hat\eta^{L}(T)\in (cd)$. 
If~\eqref{eqn::hittingtime} doesn't hold, there exists $t$  with $T<t<\varlimsup_{n\to \infty}\tau_{\delta_n}$ such that $\rho:=\dist(\eta^{L}(t),\partial\Omega)>0$. By the locally uniform convergence,  there exists a subsequence of $\{\delta_{n}\}_{n\ge 1}$ (still denoted by $\{\delta_{n}\}_{n\ge 1}$) such that  $\dist(\eta^{L}_{\delta_{n}}(t),\partial\Omega_{\delta_{n}})>\frac{\rho}{2}$ and $t\in (\tau_{\delta_n,\epsilon},\tau_{\delta_n})$ for a  for some $\epsilon>0$. In such case, we can choose $v\in\Omega_{\delta_{n}}^{*}$ adjacent to $\eta_{\delta_{n}}^{L}(t)$ on  the right-hand side such that it is connected to $(b^{*}_{\delta_{n}}c^{*}_{\delta_{n}})$ in the dual forest by a unique simple path which we denote by $T_{v,n}$. Then, we have $\diam(T_{v,n})>\frac{\rho}{2}$. Since the trunk $\{\mathbf{trunk}_{\delta_{n}}^{*}\left(\frac{1}{m}\right)\}_n$ converges to $\mathbf{trunk}^{*}_{0}\left(\frac{1}{m}\right)$ in  Hausdorff distance for every $m\in\N$, this implies $\{\mathbf{trunk}\cap\mathbf{trunk}^{*}\neq\emptyset\}$. But Lemma~\ref{lem::tightness_trunk} says that  almost surely $\mathbf{trunk}\cap\mathbf{trunk}^{*}=\emptyset $. This implies~\eqref{eqn::hittingtime}.

Secondly, we see that the driving function of $\hat\eta^{L}$ solves~\eqref{eqn::hsle} up to $\tau$ by Lemma~\ref{lem::drivingfunction}. From Proposition~\ref{prop::tightness}, the curve $\eta^{L}$ does not hit $[bc]$.  We define $x:=\phi(b)$ and $y:=\phi(c)$. By Lemma~\ref{lem::drivingfunction}, we can couple $W$ and a one-dimensional Brownian motion $B$ starting from 0  together such that, for $t<\tau$, 
\begin{align*}\label{eqn::coupling}
W_t=\sqrt{8}B_t+\int_{0}^{t}\frac{2ds}{W_s-V_s^x}+\int_{0}^{t}\frac{-2ds}{W_s-V_s^y}-8\int_{0}^{t}\frac{F'(Z_s)}{F(Z_s)}\left(\frac{1-Z_s}{V_s^y-W_s}\right)ds.
\end{align*}

Thirdly, we prove that $W$ solves~\eqref{eqn::hsle} up to and including $\tau$. We may assume that $\tau<\infty$. Note that, for any $t<\tau$, 
\[W_t=\sqrt{8}B_t-2\int_{0}^{t}\frac{V_{s}^{y}-V_{s}^{x}}{(W_s-V_s^x)(W_s-V_s^y)}ds-8\int_{0}^{t}\frac{F'(Z_s)}{F(Z_s)}\left(\frac{1-Z_s}{V_s^y-W_s}\right)ds.\]
By Corollary \ref{cor::continuousdriving}, we get that $W:[0,\infty)\to\R$ is a continuous function.
Thus, we have
\[8\int_{0}^{t}\frac{F'(Z_s)}{F(Z_s)}\left(\frac{1-Z_s}{V_s^y-W_s}\right)ds\le \max_{t\in [0,\tau)}|W_{t}-\sqrt{8}B_{t}|<\infty.\]
Then, by monotone convergence theorem, we have
\[\lim_{t\to\tau}\int_{0}^{t}\frac{F'(Z_s)}{F(Z_s)}\left(\frac{1-Z_s}{V_s^y-W_s}\right)ds=\int_{0}^{\tau}\frac{F'(Z_s)}{F(Z_s)}\left(\frac{1-Z_s}{V_s^y-W_s}\right)ds<\infty.\] 
Take $z\in\R$ such that $z>\hat\eta^{L}(\tau)$. For any $t<\tau$, we have $g_{t}(y)\le g_{t}(z)$. This implies
\[\int_{0}^{t}\frac{2ds}{V_{s}^{y}-W_{s}}\le g_{t}(z).\]
By monotone convergence theorem, we have
\[\lim_{t\to\tau}\int_{0}^{t}\frac{2ds}{V_{s}^{x}-W_{s}}=\int_{0}^{\tau}\frac{2ds}{V_{s}^{x}-W_{s}}\le \lim_{t\to\tau}\int_{0}^{t}\frac{2ds}{V_{s}^{y}-W_{s}}=\int_{0}^{\tau}\frac{2ds}{V_{s}^{y}-W_{s}}\le g_{\tau}(z)<\infty.\]
Therefore, letting $t\to\tau$, we have
\[W_\tau=\sqrt{8}B_\tau+\int_{0}^{\tau}\frac{2ds}{W_s-V_s^x}+\int_{0}^{\tau}\frac{-2ds}{W_s-V_s^y}-8\int_{0}^{\tau}\frac{F'(Z_s)}{F(Z_s)}\left(\frac{1-Z_s}{V_s^y-W_s}\right)ds.\]
In other words, the driving function $W$ of $\hat{\eta}^L$ has the same law as the one for $\hSLE_8$ up to and including $\tau$. In this step, it is important that $W$ is continuous up to and including $\tau$ due to Corollary~\ref{cor::continuousdriving}. 

Finally, we will show that the driving function of $\eta^L[\tau,\infty]$ given $\eta^L[0,\tau]$ is $\sqrt{8}$ times a one-dimensional Brownian motion starting from $0$. Denote by $\Omega(\tau)$ the connected component of $\Omega\setminus\eta^{L}[0,\tau]$ having $d$ on its boundary. From above, we know that   $\Omega_{\delta_n}(\tau_{\delta_n})$ converges to $\Omega(\tau)$ in the Carath\'eodory sense by Carath\'eodory kernel theorem.  By the domain Markov property, conditioning on $\eta_{\delta_{n}}[0,\tau_{\delta_n}]$, the remaining curve $\eta_{\delta_{n}}[\tau_{\delta_n},\infty]$ has the same law as the Peano curve from $\eta_{\delta_{n}}(\tau_{\delta_n})$ to $d^{\diamond}_{\delta_{n}}$ in $\Omega_{\delta_n}(\tau_{\delta_n})$ with Dobrushin boundary conditions. By~\cite[Theorem~4.4]{LawlerSchrammWernerLERWUST}, the driving function of $\hat\eta^{L}[\tau,\infty]$ has the same law as the driving function of $\SLE_{8}$ in $\Omega_{\tau}$ from $\eta(\tau)$ to $d$. It is important that the convergence of driving function only requires the convergence of domains in the Carath\'eodory sense and there is no smoothness regularity requirement on the boundary of the limiting domain. Thus, the driving function of $\hat{\eta}^L[\tau,\infty]$ given $\hat{\eta}^L[0,\tau]$ is $\sqrt{8}$ times a one-dimensional Brownian motion starting from $0$. 

In summary, the driving function of $\hat{\eta}^L$ is the same as the one for $\hSLE_8$ as  a whole process. 
This completes the proof.
\end{proof}
As a consequence of Theorem~\ref{thm::ust_hsle}, we have the following.
\begin{corollary}\label{cor::ust_hsle8}
Fix a quad $(\Omega; a, b, c, d)$. The process $\eta\sim \hSLE_8$ in $\Omega$ from $a$ to $d$ with marked points $(b, c)$ has the following properties: It is almost surely generated by continuous curve and $\eta\cap [bc]=\emptyset$ almost surely. Moreover, it is reversible: the time-reversal of $\eta$ has the law of $\hSLE_8$ in $\Omega$ from $d$ to $a$ with marked points $(c, b)$.  
\end{corollary}

\begin{proof}
We may assume that $\partial\Omega$ is $C^1$ and simple. Note that, if the conclusion holds under such assumption, it would also hold for a general quad with locally connected boundary via conformal mapping. 

First of all, we argue that there exists a unique solution in law to the SDE~\eqref{eqn::hsle} up to and including $\tau$---the first time that $\phi(b)$ is swallowed.
From the proof of Theorem~\ref{thm::ust_hsle}, we see that there exists a version of solution $W$ to the SDE~\eqref{eqn::hsle} up to $\tau$. Moreover, it is generated by a continuous curve $\eta$ up to and including $\tau$ and $\eta\cap [\phi(b), \phi(c)]=\emptyset$. Suppose $\tilde{W}$ is another solution. Denote by $\tau_{\eps}$ the first time that the process gets within $\eps$-neighborhood of $(\phi(b), \infty)$. As the SDE~\eqref{eqn::hsle} has a unique solution up to $\tau_{\eps}$, the two processes $W$ and $\tilde{W}$ have the same law up to $\tau_{\eps}$. We may couple them so that $W_t=\tilde{W}_t$ for $t\le \tau_{\eps}$. 
 We denote by $\PP_{\eps}$ the probability measure corresponding to this coupling. Then, the family $\{\PP_\eps\}_{\eps>0}$ is consistent on the product space of continuous functions on $[0,+\infty)$. Thus, we can construct a new probability measure $Q$ such that under $Q$, two driving functions $W_t=\tilde{W}_t$ for all $t<T$ where $T$ is the first hitting time of $[\phi(b), \infty)$. Since $\eta\cap [\phi(b), \phi(c)]=\emptyset$, we have $\tau=T$. This implies that the SDE~\eqref{eqn::hsle} has a unique solution up to and including $\tau$, which is given by the limit of the Peano curve in Theorem~\ref{thm::ust_hsle}. 
Then, the continuity of $\hSLE_8$ is a consequence of Proposition~\ref{prop::tightness}. For the reversibility, we denote by $\LR(\eta_{\delta_n}^{L})$ the time-reversal of $\eta_{\delta_n}^{L}$. By Theorem~\ref{thm::ust_hsle}, the law of $\LR(\eta_{\delta_n}^{L})$ converges to $\hSLE_8$ in $\Omega$ from $d$ to $a$ with marked points $(c, b)$ as $\delta_n\to 0$. This implies the reversibility and completes the proof.
\end{proof}

\subsection{Proof of Theorem~\ref{thm::cvg_triple}}
In this section, we will complete the proof of Theorem~\ref{thm::cvg_triple}. 



\begin{proof}[Proof of Theorem~\ref{thm::cvg_triple}]
Note that, the families $\{\eta^L_{\delta}\}_{\delta>0}$ and $\{\eta^R_{\delta}\}_{\delta>0}$ are tight due to Proposition~\ref{prop::tightness}; and the family $\{\gamma^M_{\delta}\}_{\delta>0}$ is tight due to Proposition~\ref{prop::tightness_LERW}. 
Hence for any sequence $\delta_{n}\to 0$, there exists a subsequence, still denoted by $\delta_n$, such that $\{\left(\eta_{\delta_{n}}^{L};\gamma_{\delta_{n}}^{M};\eta_{\delta_{n}}^{R}\right)\}$ converges in law as $n\to\infty$. We couple $\{\left(\eta_{\delta_n}^{L};\gamma_{\delta_n}^{M};\eta_{\delta_n}^{R}\right)\}$ together such that $\eta_{\delta_n}^{L}\to\eta^{L}$ and $\gamma_{\delta_n}^{M}\to\gamma^{M}$ and $\eta_{\delta_n}^{R}\to\eta^{R}$ as curves almost surely. We will prove that the law of the triple $\left(\eta^{L}; \gamma^{M}; \eta^{R}\right)$ is the one in Theorem~\ref{thm::cvg_triple}. 

The law of $\eta^{L}$ is $\hSLE_8$ in $\Omega$ from $a$ to $d$ with marked points $(b, c)$ due to Theorem~\ref{thm::ust_hsle}. 
By Lemma~\ref{lem::etaLgivengammaM}, the conditional law of $\eta^L$ given $\gamma^M$ is $\SLE_8$. Similarly, the conditional law of $\eta^R$ given $\gamma^M$ is $\SLE_8$. These imply that $\eta^L$ and $\eta^R$ are conditionally independent given $\gamma^M$. 
Since $\SLE_{8}$ is space-filling, we have $\gamma^{M}=\eta^{L}\cap\eta^{R}$. Therefore, the conditional law of $\eta^R$ given $\eta^L$ is the same as the conditional law of $\eta^R$ given $\gamma^M$ which is $\SLE_8$. This completes the proof.
\end{proof}

\begin{corollary}\label{cor::gamma_lebzero}
Consider the continuous curve $\gamma^M$ in the triple of Theorem~\ref{thm::cvg_triple}. We have $\PP[z\in\gamma^M]=0$ for any $z\in\Omega$ and $\Leb(\gamma^M)=0$ almost surely. 
\end{corollary}

\begin{proof}
By Theorem~\ref{thm::ust_hsle}, the law of $\eta^L$ is $\hSLE_8$ in $\Omega$ from $a$ to $d$ with marked points $(b, c)$. From Theorem~\ref{thm::cvg_triple}, the curve $\gamma^M$ is the part of the boundary of $\eta^L$ inside $\Omega$. We parameterize $\gamma^M$ so that $\gamma^M(0)=X^M$ and $\gamma^M(1)=Y^M$. 
Let $\eta$ be an $\SLE_8$ in $\Omega$ from $a$ to $d$ and denote by $\tau$ the first time that $\eta$ swallows $b$. From~\eqref{eqn::rev_mart_aux2}, the law of $\eta^L$ is absolutely continuous with respect to $\eta$ up to $\tau$. 
For every $z\in\Omega$, the probability that $z$ belongs to the frontier of $\SLE_8$ equals $0$, thus $\PP[z\in\gamma^M[0,t]]=0$ for any $t\in (0,1)$.   
As $\gamma^M=\cup_n \gamma^M[0,1-1/n]$, we have $\PP[z\in\gamma^M]=0$. 
This implies that $\PP[\Leb(\gamma^M)]=0$,
which implies that almost surely, we have $\Leb(\gamma^M)=0$ as desired.
\end{proof}

\section{Convergence of LERW in quads}
\label{sec::lerw}
\subsection{The pair of random points $(X^M, Y^M)$}
\label{subsec::pairrandompoints}
The goal of this section is to derive the limiting distribution of the pair of random points $(X^M, Y^M)$ in Theorem~\ref{thm::cvg_lerw}. We summarize the setup below.
\begin{itemize}
\item Fix a quad $(\Omega; a, b, c, d)$ such that $\partial\Omega$ is $C^{1}$ and simple. Suppose that a sequence of medial quads $(\Omega_{\delta}^{\diamond}; a_{\delta}^{\diamond}, b_{\delta}^{\diamond}, c_{\delta}^{\diamond}, d_{\delta}^{\diamond})$ converges to $(\Omega; a, b, c, d)$ in the sense of~\eqref{eqn::topology}.
Assume the same setup as in Section~\ref{subsec::ust_quads}. 
We consider the UST $\LT_{\delta}$ in $\Omega_{\delta}$ with $(a_{\delta}b_{\delta})$ wired and $(c_{\delta}d_{\delta})$ wired. 
There are two Peano curves running along $\LT_{\delta}$, and we denote by $\eta^L_{\delta}$ the one from $a_{\delta}^{\diamond}$ to $d_{\delta}^{\diamond}$ and by $\eta^R_{\delta}$ the one from $b_{\delta}^{\diamond}$ to $c_{\delta}^{\diamond}$. 
There exists a unique simple path in $\LT_{\delta}$, denoted by $\gamma^M_{\delta}$, connecting $(a_{\delta}b_{\delta})$ to $(c_{\delta}d_{\delta})$. Recall that $X_{\delta}^{M}:=\gamma_{\delta}^{M}\cap(a_{\delta}b_{\delta})$ and $Y_{\delta}^{M}:=\gamma_{\delta}^{M}\cap(c_{\delta}d_{\delta})$. 
\item For the quad $(\Omega; a, b, c, d)$, let $K$ be its conformal modulus and denote by $f=f_{(\Omega;a,b,c,d)}$ the conformal map from $\Omega$ onto $(0,1)\times (0,\ii K)$ which sends $(a,b,c,d)$ to $(0,1,1+\ii K,\ii K)$ and extend its definition continuously to the boundary. 
\item Consider the Poisson kernel for the rectangle $f(\Omega)=(0,1)\times (0,\ii K)$. 
Define, for all $r\in (f(a)f(b))\cup(f(c)f(d))$ and for all $z\in ([0,1]\times[0,\ii K])\setminus\{r\}$, 
\begin{equation}\label{eqn::Poissonkernel_rect}
P_K(z, r)=\Im\sum_{n\in\Z}\left(\frac{1}{\exp(\frac{\pi}{K}(2n-r+z))-1}+\frac{1}{\exp(\frac{\pi}{K}(2n-r-\bar{z}))-1}\right). 
\end{equation}
Note that $P_K(\cdot, r)$ is continuous on $[0,1]\times[0,\ii K]\setminus\{r\}$, and it is harmonic on $(0,1)\times (0,\ii K)$ with the following boundary data: 
\begin{equation}\label{eqn::Poissonkernel_def}
P_K(\cdot, r)=0 \text{ on } (f(a)f(b))\cup (f(c)f(d))\setminus\{r\},\quad  
\partial_{n}P_K(\cdot, r)=0 \text{ on } (f(b)f(c))\cup (f(d)f(a)),
\end{equation}
where $n$ is the outer normal vector. 
\end{itemize}

Before we proceed, let us explain how we figure out the formula for Poisson kernel in~\eqref{eqn::Poissonkernel_rect} as a harmonic solution with boundary data~\eqref{eqn::Poissonkernel_def}. The Poisson kernel $P_K(z,r)$ with boundary data~\eqref{eqn::Poissonkernel_def} can be regarded as the ``renormalized probability" that the Brownian motion from $z$  reflecting on the edges $(0, \ii K)\cup (1, 1+\ii K)$  first hits  $(0, 1)\cup (\ii K, 1+\ii K)$ at $r$. By considering the strip $\mathcal{S}:=\mathbb{R}\times[0,\ii K]$, this ``renormalzied probability" is equal to the ``renormalized probability" that the Brownian motion from $z$ exits $\mathcal{S}$ at the points
 $\{r+n, -r+n\}_{n\in \mathbb{Z}}$. 
Hence we have
\[P_K(z,r)=\sum_{n\in\mathbb{Z}} \left(P_{\mathcal{S}}(z,r+n)+P_{\mathcal{S}}(z,-r+n)\right), \]
where $P_{\mathcal{S}}(z,q)$ is the Poisson kernel of the strip $\mathcal{S}$ with Dirichlet boundary condition. Note that the Poisson kernel $P_{\mathcal{S}}(z,q)$ can be obtained from the Poisson kernel of the upper half plane $\mathbb{H}$ by composing with the conformal map $\varphi(z)= e^z$. This gives the desired formula in~\eqref{eqn::Poissonkernel_rect}. 

Let us come back to the distribution of the limits of the pairs $(X_{\delta}^M; Y_{\delta}^M)_{\delta>0}$. 
\begin{proposition}\label{prop::cvg_pair_points}
The pair $(X_{\delta}^{M},Y_{\delta}^{M})$ converges weakly  as a pair of points in $\R^2$ to a random pair of points $(X^{M},Y^{M})$ as $\delta\to 0$. Denote by $x^M:=f(X^{M})$ and $y^M:=\Re f(Y^M)$. The law of the pair $(x^{M},y^{M})$ is characterized by the following. 
\begin{itemize}
\item[(1)] The law of $x^{M}$ is uniform on $(0,1)$. 
\item[(2)] The conditional density of $y^{M}$ given $x^{M}\in (0,1)$ is the following: 
\begin{equation}\label{eqn::conditionalYgivenX}
\rho_K(x^M, y)=\partial_n P_K(z, y+\ii K)|_{z=x^M}, \quad \forall y\in (0,1). 
\end{equation}
\end{itemize}
In particular, the joint density of $(x^M, y^M)$ is given by~\eqref{eqn::jointdensity}. 
\end{proposition}

The proof of Proposition~\ref{prop::cvg_pair_points} is split into three lemmas. In Lemma~\ref{lem::X}, we first derive the limiting distribution of $X^M_{\delta}$. This step is immediate from the convergence of the observable in Lemmas~\ref{lem::holomorphic_observable} and~\ref{lem::holo_cvg}.  
We then derive the conditional law of $y^M$ given $x^M$. 
To this end, we first analyze the conditional probability in the discrete setting in Lemma~\ref{lem::YconditionalXdiscrete} and use good control on discrete harmonic functions proved in~\cite{ChelkakWanMassiveLERW}; and then we derive the limit of the conditional probability in Lemma~\ref{lem::YconditionalX}. 

\begin{lemma}\label{lem::X}
The pair $(X_{\delta}^{M},Y_{\delta}^{M})$ converges weakly  as a pair of points in $\R^2$ to a random pair of points $(X^{M},Y^{M})$ as $\delta\to 0$. Moreover, the law of $x^M=f(X^{M})$ is uniform on $(0,1)$. 
\end{lemma}
\begin{proof}
From Theorem~\ref{thm::cvg_triple}, the curve $\gamma_{\delta}^M$ converges weakly to $\gamma^M$ as $\delta\to 0$. Since $X^M_\delta=\gamma_\delta^M(0)$, $Y_\delta^M=\gamma_\delta^M(1)$ and $X^M=\gamma^M(0)$, $Y^M=\gamma^M(1)$, this implies that $(X_{\delta}^M, Y_{\delta}^M)$ converges weakly to $(X^M, Y^M)$ as a pair of points in $\R^2 $ as $\delta\to 0$, where $X^M=\gamma^M\cap (ab)$ and $Y^M=\gamma^M\cap (cd)$. It remains to show that $f(X^M)$ is uniform in $(0,1)$. 

Recall from Lemma~\ref{lem::holomorphic_observable} that $u_\delta(z^*)$ is the probability that $z^*$ lies to the right of $\eta_{\delta}^L$ for every $z^{*}\in\Omega^{*}_{\delta}$ and that $u_{\delta}$ converges to $\Re f$ locally uniformly due to Lemma~\ref{lem::holo_cvg}. 
We denote by $\Omega^{R}$ the connected component of $\Omega\setminus\gamma^{M}$ which contains $[bc]$ on its boundary. It is also the connected component of $\Omega\setminus\eta^{L}$ which contains $[bc]$ on its boundary. For every $z\in\Omega$ and $z_{\delta_n}^{*}\in\Omega_{\delta_n}^{*}$ such that $z_{\delta_n}^{*}\to z$, we have
\begin{align*}
\{z\in\Omega^{R}\}&\subset\left\{\cup_{i=1}^{\infty}\cap_{n=i}^{\infty}\{z_{\delta_n}^{*}\text{ lies to the right of }\eta_{\delta_n}^L\}\right\}\\
&\subset\left\{\cap_{i=1}^{\infty}\cup_{n=i}^{\infty}\{z_{\delta_n}^{*}\text{ lies to the right of }\eta_{\delta_n}^L\}\right\}\subset\{z\in\overline\Omega^{R}\}.
\end{align*}
This implies that
\[\PP[z\in\Omega^{R}]\le \varliminf_{n\to\infty}u_{\delta_n}(z_{\delta_{n}}^{*})=\Re f(z)=\varlimsup_{n\to\infty}u_{\delta_n}(z_{\delta_{n}}^{*})\le\PP[z\in\overline\Omega^{R}].\]
Note that $\PP[z\in\Omega^R]=\PP[z\in\overline{\Omega}^R]$ as $\PP[z\in\gamma^M]=0$ due to Corollary~\ref{cor::gamma_lebzero}. Therefore, 
\begin{equation}\label{eqn::hSLE_right_proba_aux}
\PP[z\in\Omega^{R}]=\Re f(z),\quad \forall z\in\Omega.
\end{equation}
For every $\theta\in(ab)$, we choose $\{w_n\}\subset\Omega$ such that $w_{n}\to\theta$ as $n\to\infty$. Then, we have 
\[\{X^{M}\in(a\theta)\}\subset\left\{\cup_{i=1}^{\infty}\cap_{n=i}^{\infty}\{w_{n}\in\Omega^{R}\}\right\}\subset\left\{\cap_{i=1}^{\infty}\cup_{n=i}^{\infty}\{w_{n}\in\Omega^{R}\}\right\}\subset\{X^{M}\in(a\theta]\}.\]
Since $f$ is continuous on $\overline\Omega$, we have
\[\PP[X^{M}\in(a\theta)]\le\varliminf_{n\to\infty}\Re f(w_{n})=\Re f(\theta)\le\varlimsup_{n\to\infty}\Re f(w_{n})\le\PP[X^{M}\in(a\theta]].\]
Furthermore, for every $\tilde\theta\in(\theta b)$, we have
\[\Re f(\theta)\le\PP[X^{M}\in(a\theta]]\le\PP[X^{M}\in(a\tilde\theta)]\le\Re f(\tilde\theta).\]
By letting $\tilde\theta\to\theta$, we have 
\[\PP[X^{M}\in(a,\theta]]=\Re f(\theta).\]
This implies that $\PP[f(X^{M})\in(0,\Re f(\theta)]]=\PP[X^{M}\in(a,\theta]]=\Re f(\theta)$ as desired.  Note that $f(a,b)=(0,1)$ and we complete the proof.
\end{proof}
\begin{figure}[ht!]
\includegraphics[width=0.7\textwidth]{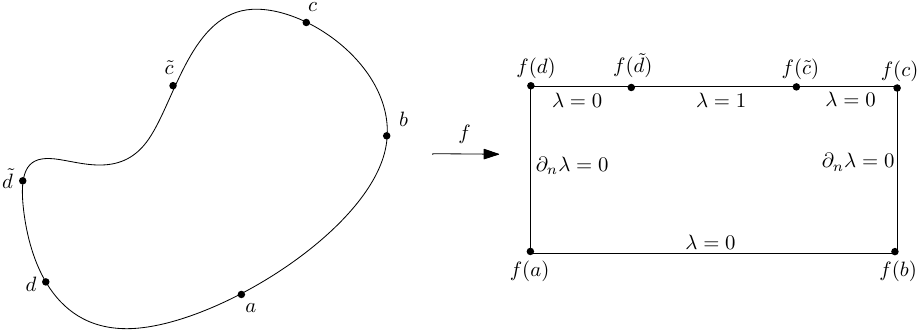}
\caption{\label{fig::boundary-conditions} The function $\lambda=\lambda_{\tilde{c}, \tilde{d}}$ is the unique bounded harmonic function on $(0,1)\times (0,\ii K)$ with the following boundary data: $\lambda=0$ on $(f(a)f(b))\cup(f(c)f(\tilde c))\cup(f(\tilde d)f(d))$; $\lambda=1$ on $(f(\tilde c)f(\tilde d))$; and $\partial_n\lambda=0$ on $(f(b)f(c))\cup(f(d)f(a))$ where $n$ is the outer normal vector. }
\end{figure}

\begin{lemma}\label{lem::YconditionalXdiscrete}
Fix a polygon $(\Omega; a, \tilde{a}, \tilde{b}, b, c, \tilde{c}, \tilde{d}, d)$ with eight marked points. Suppose that a sequence of medial polygons $(\Omega_{\delta}^{\diamond}; a_{\delta}^{\diamond}, \tilde{a}_{\delta}^{\diamond}, \tilde{b}_{\delta}^{\diamond}, b_{\delta}^{\diamond}, c_{\delta}^{\diamond}, \tilde{c}_{\delta}^{\diamond}, \tilde{d}_{\delta}^{\diamond}, d_{\delta}^{\diamond})$ converges to $(\Omega; a, \tilde{a}, \tilde{b}, b, c, \tilde{c}, \tilde{d}, d)$ in the sense of~\eqref{eqn::topology}. Denote by $\lambda_{\tilde c,\tilde d}$ the unique bounded harmonic function on $(0,1)\times (0,\ii K)$ with the boundary data as shown in Figure~\ref{fig::boundary-conditions}.  Denote by $\lambda_{c, d}$ when $\tilde{c}=c$ and $\tilde{d}=d$.  
Then, for every $\epsilon>0$, there exist $\delta_0>0$ and $s>0$ such that for all $\delta\le\delta_0$ and $x_{\delta}\in(\tilde a_{\delta}\tilde b_{\delta})$ and $x\in (\tilde a\tilde b)$ with $\dist(x_{\delta},x)\le s$, we have 
\begin{align}\label{eqn::discretelaw}
\left|\PP\left[Y_{\delta}^{M}\in(\tilde c_{\delta}\tilde d_{\delta})\cond X_{\delta}^{M}=x_{\delta}\right]-\frac{\partial_n \lambda_{\tilde c,\tilde d}(f(x))}{\partial_{n}\lambda_{c,d}(f(x))}\right|\le\epsilon.
\end{align}
\end{lemma}

\begin{proof}
First, we show that $\partial_{n}\lambda_{c,d}(f(x))\neq0$ for all $x\in(ab)$. Let $g$ be the bounded harmonic function on $(0,1)\times (0,\ii K)$ with the following boundary data: $g=0$ on $(f(d)f(c))$ and $g=1$ on $(f(c)f(d))$. By maximum principle, we have $\lambda_{c,d}(y)\ge g(y)$ for every $y\in[0,1]\times[0,\ii K]$. Thus, we have 
\[
\partial_n\lambda_{c,d}(f(x))\le\partial_ng(f(x))<0, \quad \text{for all }x\in (ab). 
\] 

Next, we prove~\eqref{eqn::discretelaw}. 
To this end, Wilson's algorithm allows us to use estimates on loop erased random walk to get estimations on the branch $\gamma_\delta^M$.
To be precise, denote by $\tilde\Omega_{\delta}$ the graph obtained from $\Omega_{\delta}$ by regarding $(a_{\delta}b_{\delta})$ as a single vertex (and keep all edges connecting $(a_{\delta}b_{\delta})$ to interior vertices). Let $\tilde{\gamma}_{\delta}$ be the loop-erasure of the random walk on $\tilde\Omega_{\delta}$ starting from $(a_{\delta}b_{\delta})$ and stopped whenever it hits $(c_{\delta}d_{\delta})$. 
Wilson's algorithm~\cite{PemantleSpanningTree} tells that $\gamma_{\delta}^{M}$ (viewed in $\tilde{\Omega}_{\delta}$) has the same law as $\tilde{\gamma}_{\delta}$. Going back to $\Omega_{\delta}$, we obtain that, for $v_{\delta}\in\Omega_{\delta}$ with $v_{\delta}\sim(a_{\delta}b_{\delta})$ and $\PP[\tilde{\gamma}_{\delta}(1)=v_{\delta}]>0$,  it holds that  $\PP[\gamma_{\delta}^M(1)=v_{\delta}]=\PP[\tilde{\gamma}_{\delta}(1)=v_{\delta}]$. Moreover, the conditional law of $\gamma_{\delta}^M$ given $\{\gamma_{\delta}^M(1)=v_{\delta}\}$ is the same as the loop-erasure of the random walk in $\Omega_{\delta}$ starting from $v_{\delta}$ conditioned to hit $(a_{\delta}b_{\delta})\cup(c_{\delta}d_{\delta})$ at  $(c_{\delta}d_{\delta})$. 

For any $w_{\delta}\in \Omega_{\delta}$, denote by $\PP^{w_{\delta}}$ the law of  random walk $\LR$ in $\Omega_{\delta}$ starting from $w_{\delta}$. Define
\begin{align*}
u_{\delta}(w_{\delta}):=&\PP^{w_{\delta}}\left[\LR\text{ hits }(a_{\delta}b_{\delta})\cup(c_\delta d_\delta)\text{ at }\left(c_{\delta}d_{\delta}\right)\right],\\
\tilde u_{\delta}(w_{\delta}):=&\PP^{w_{\delta}}\left[\LR\text{ hits }(a_{\delta}b_{\delta})\cup(c_\delta d_\delta)\text{ at }\left(\tilde c_{\delta}\tilde d_{\delta}\right)\right]. 
\end{align*}
By the above relation  between $\gamma_{\delta}^M$ and the random walk, we have
\begin{align*}
\PP\left[Y_{\delta}^{M}\in(\tilde c_{\delta}\tilde d_{\delta})\cond X_{\delta}^{M}=x_{\delta}\right]&=\frac{\sum_{\substack{v_{\delta}\sim x_{\delta}\\ v_{\delta}\in \Omega_{\delta}}}\PP^{v_{\delta}}\left[\LR\text{ hits }(a_{\delta}b_{\delta})\cup(c_\delta d_\delta)\text{ at }\left(\tilde c_{\delta}\tilde d_{\delta}\right)\right]}{\sum_{\substack{v_{\delta}\sim x_{\delta}\\ v_{\delta}\in \Omega_{\delta}}}\PP^{v_{\delta}}\left[\LR\text{ hits }(a_{\delta}b_{\delta})\cup(c_\delta d_\delta)\text{ at }(c_{\delta} d_{\delta})\right]}
=\frac{\sum_{\substack{v_{\delta}\sim x_{\delta}\\ v_{\delta}\in \Omega_{\delta}}}\tilde u_{\delta}(v_{\delta})}{\sum_{\substack{v_{\delta}\sim x_{\delta}\\ v_{\delta}\in \Omega_{\delta}}}u_{\delta}(v_{\delta})}.
\end{align*}
The function $\tilde u_{\delta}$ is a discrete harmonic function on $\Omega_{\delta}\setminus (a_{\delta}b_{\delta})\cup(c_{\delta}d_{\delta})$ with the following boundary data: $\tilde u_{\delta}=0$ on $(a_{\delta}b_{\delta})\cup(c_{\delta}\tilde c_{\delta})\cup(\tilde d_{\delta}d_{\delta})$ and $\tilde u_{\delta}=1$ on $(\tilde c_{\delta} \tilde d_{\delta})$. Similarly, $u_{\delta}$ is a discrete harmonic function on $\Omega_{\delta}\setminus (a_{\delta}b_{\delta})\cup(c_{\delta}d_{\delta})$ with the following boundary data: $u_{\delta}=0$ on $(a_{\delta}b_{\delta})$ and $u_{\delta}=1$ on $(c_{\delta}d_{\delta})$. By~\cite[Corollary~3.8]{ChelkakWanMassiveLERW}, for every $\epsilon>0$, there exists $s_1>0$ such that 
\[1-\epsilon\le\frac{\tilde u_{\delta}(v_{\delta})}{u_{\delta}(v_{\delta})}\times\frac{u_{\delta}(z_{\delta})}{\tilde u_{\delta}(z_{\delta})}\le 1+\epsilon, \quad \text{for all } z_{\delta}\in \Omega_{\delta}\cap B(x_{\delta},s_1)\setminus \partial\Omega_{\delta} \text{ and all } x_{\delta}\in(\tilde a_{\delta}\tilde b_{\delta}).\]
This implies that
\[1-\epsilon\le\PP\left[Y_{\delta}^{M}\in(\tilde c_{\delta}\tilde d_{\delta})\cond X_{\delta}^{M}=x_{\delta}\right]\times\frac{u_{\delta}(z_{\delta})}{\tilde u_{\delta}(z_{\delta})}\le 1+\epsilon.\]
We choose $s<\frac{s_1}{4}$. Since $\partial\Omega$ is a curve, we can choose a simple curve $L$ such that $\frac{s}{4}< \dist(L,(\tilde a\tilde b))< \frac{s}{2}$. Note that there exists $\delta_{1}>0$ such that for all $\delta<\delta_1$, we can choose a discrete simple curve $L_{\delta}\subset\Omega_{\delta}$ with $\frac{s}{4}< \dist(L_{\delta},(\tilde a_{\delta}\tilde b_{\delta}))< \frac{s}{2}$ and  $L_{\delta}\to L$ as curves. 
By the same argument as in the proof of Lemma~\ref{lem::harmonic_cvg}, we have  $\tilde u_{\delta}\to\lambda_{\tilde c,\tilde d}\circ f$ and $u_{\delta}\to\lambda_{c, d}\circ f$ locally uniformly in $\Omega$. Then, there exists $\delta_2>0$ such that the following holds: for every $\delta<\delta_2$ and $\dist(x_{\delta},x)<s$, there exists $z\in B(x,s)\cap L$ such that 
\[
\left|\frac{\tilde u_{\delta}(z_{\delta})}{u_{\delta}(z_{\delta})}-\frac{\lambda_{\tilde c,\tilde d}(f(z))}{\lambda_{c, d}( f(z))}\right|<\epsilon,\quad \forall z_{\delta}\in B(x_{\delta},s)\cap L_{\delta}.
\]
Since $f$ is continuous on $\overline\Omega$, we have $\diam(f(B(x,s)))\to 0$ as $s\to 0$.
By Taylor expansion, we can choose $s$ small enough such that 
\[\left|\frac{\lambda_{\tilde c,\tilde d}(f(z))}{\lambda_{c,d}(f(z))}-\frac{\partial_n\lambda_{\tilde c,\tilde d}(f(x))}{\partial_n\lambda_{c,d}(f(x))}\right|\le\epsilon, \quad \text{for all } x\in(\tilde a\tilde b) \text{ and } z\in B(x,s)\cap L. \]
This implies that, if $\delta<\delta_1\wedge\delta_2$ and $\dist(x_{\delta},x)<s$, we have 
\[(1-3\epsilon)\frac{\partial_n\lambda_{\tilde c,\tilde d}(f(x))}{\partial_n\lambda_{c,d}(f(x))}\le\PP\left[Y_{\delta}^{M}\in(\tilde c_{\delta}\tilde d_{\delta})\cond X_{\delta}^{M}=x_{\delta}\right]\le (1+3\epsilon)\frac{\partial_n\lambda_{\tilde c,\tilde d}(f(x))}{\partial_n\lambda_{c,d}(f(x))}.\]
This completes the proof.
\end{proof}
\begin{lemma}\label{lem::YconditionalX}
The conditional law of $Y^{M}$ given $X^{M}$ is given by
\begin{equation}\label{eqn::YconditionalX_aux}
\PP[Y^{M}\in (\tilde{c}\tilde{d})\cond X^{M}]=\frac{\partial_n\lambda_{\tilde c,\tilde d}(f(X^{M}))}{\partial_n\lambda_{c,d}(f(X^{M}))}. 
\end{equation}
\end{lemma}
\begin{proof}
By  conformal invariance, we may assume $\Omega=(0,1)\times (0,\ii K)$. We couple $(X_{\delta}^{M},Y_{\delta}^{M})$ and $(X^{M},Y^{M})$ together such that $X_{\delta}^{M}\to X^{M}$ and $Y_{\delta}^{M}\to Y^{M}$ almost surely. 
Fix a polygon $(\Omega; a, \tilde{a}, \tilde{b}, b, c, \tilde{c}, \tilde{d}, d)$ with eight marked points. Suppose that a sequence of medial polygons $(\Omega_{\delta}^{\diamond}; a_{\delta}^{\diamond}, \tilde{a}_{\delta}^{\diamond}, \tilde{b}_{\delta}^{\diamond}, b_{\delta}^{\diamond}, c_{\delta}^{\diamond}, \tilde{c}_{\delta}^{\diamond}, \tilde{d}_{\delta}^{\diamond}, d_{\delta}^{\diamond})$ converges to $(\Omega; a, \tilde{a}, \tilde{b}, b, c, \tilde{c}, \tilde{d}, d)$ in the sense of~\eqref{eqn::topology}. 
For any $\delta_n\to 0$, we have 
\begin{align*}
\{X^{M}\in(\tilde a\tilde b),Y^{M}\in(\tilde c\tilde d)\}&\subset\left\{\cup_{j=1}^{\infty}\cap_{n=j}^{\infty}\{X^{M}_{\delta_n}\in(\tilde a_{\delta_n}\tilde b_{\delta_n}),Y^{M}_{\delta_n}\in(\tilde c_{\delta_n}\tilde d_{\delta_n})\}\right\}\\
&\subset\left\{\cap_{j=1}^{\infty}\cup_{n=j}^{\infty}\{X^{M}_{\delta_n}\in(\tilde a_{\delta_n}\tilde b_{\delta_n}),Y^{M}_{\delta_n}\in(\tilde c_{\delta_n}\tilde d_{\delta_n})\}\right\}\subset\{X^{M}\in[\tilde a\tilde b], Y^{M}\in[\tilde c\tilde d]\}. 
\end{align*}
Thus, 
\begin{equation}\label{eqn::upboundlowerbound}
\begin{split}
\PP[X^{M}\in(\tilde a\tilde b),Y^{M}\in(\tilde c\tilde d)]&\le \varliminf_{n\to\infty}\PP[X^{M}_{\delta_n}\in(\tilde a_{\delta_n}\tilde b_{\delta_n}),Y^{M}_{\delta_n}\in(\tilde c_{\delta_n}\tilde d_{\delta_n})]\\
&\le  \varlimsup_{n\to\infty}\PP[X^{M}_{\delta_n}\in(\tilde a_{\delta_n}\tilde b_{\delta_n}),Y^{M}_{\delta_n}\in(\tilde c_{\delta_n}\tilde d_{\delta_n})]\le \PP[X^{M}\in[\tilde a\tilde b],Y^{M}\in[\tilde c\tilde d]].
\end{split}
\end{equation}

For every $\epsilon>0$, we choose $s$ the same as in Lemma~\ref{lem::YconditionalXdiscrete}. Divide $(\tilde a\tilde b)$ into $\cup_{j=0}^{m}[x^{j}x^{j+1}]$ with $x^{0}=\tilde a$ and $x^{m+1}=\tilde b$ such that the length of $[x^{j}x^{j+1}]$ is less than $s$. Denote by $\{x_{\delta_n}^{0}=\tilde a_{\delta_n},x_{\delta_n}^{1},\ldots,x_{\delta_n}^{m+1}=\tilde b_{\delta_n}\}\subset(\tilde a_{\delta_n}\tilde b_{\delta_n})$ the corresponding discrete approximation. By Lemma~\ref{lem::YconditionalXdiscrete}, for $n$ large enough, we have
\[\left|\PP[X^{M}_{\delta_n}\in(\tilde a_{\delta_n}\tilde b_{\delta_n}),Y^{M}_{\delta_n}\in(\tilde c_{\delta_n}\tilde d_{\delta_n})]- \sum_{j=0}^{m}\frac{\partial_n\lambda_{\tilde c,\tilde d}(f(x^j))}{\partial_n\lambda_{c,d}(f(x^j))}\PP[X_{\delta_n}^{M}\in[x_{\delta_{n}}^{j}x_{\delta_{n}}^{j+1}]]\right|\le \epsilon.\]
By Lemma~\ref{lem::X}, the law of $X^{M}$ is uniform on $(ab)$. This implies that
\[\lim_{n\to\infty}\PP[X_{\delta_n}^{M}\in[x_{\delta_{n}}^{j}x_{\delta_{n}}^{j+1}]]=\PP[X^{M}\in[x^{j}x^{j+1}]], \quad \text{for }0\le j\le m. \]
Thus, we have
\[\varlimsup_{n\to\infty}\left|\PP[X^{M}_{\delta_n}\in(\tilde a_{\delta_n}\tilde b_{\delta_n}),Y^{M}_{\delta_n}\in(\tilde c_{\delta_n}\tilde d_{\delta_n})]-\sum_{j=0}^{m}\frac{\partial_n\lambda_{\tilde c,\tilde d}(f(x^j))}{\partial_n\lambda_{c,d}(f(x^j))}\PP[X^{M}\in[x^{j}x^{j+1}]]\right|\le \epsilon.\]
By letting $s\to 0 \,(m\to\infty)$ and $\epsilon\to 0$, we obtain
\[\lim_{n\to\infty}\PP[X^{M}_{\delta_n}\in(\tilde a_{\delta_n}\tilde b_{\delta_n}),Y^{M}_{\delta_n}\in(\tilde c_{\delta_n}\tilde d_{\delta_n})]=\int_{f(\tilde a)}^{f(\tilde b)}\frac{\partial_n\lambda_{\tilde c,\tilde d}(x)}{\partial_n\lambda_{c,d}(x)}dx.\]
where we used  that  $X^M$ is uniform distributed on $(0,1)$ (see Lemma~\ref{lem::X}).
Plugging into~\eqref{eqn::upboundlowerbound} implies that
\begin{equation*}
\PP[X^{M}\in(\tilde a\tilde b),Y^{M}\in(\tilde c\tilde d)]
\le \int_{f(\tilde a)}^{f(\tilde b)}\frac{\partial_n\lambda_{\tilde c,\tilde d}(x)}{\partial_n\lambda_{c,d}(x)}dx
\le \PP[X^{M}\in[\tilde a\tilde b],Y^{M}\in[\tilde c\tilde d]]. 
\end{equation*}
Note that the marginal law of $X^M$ is uniform on $(ab)$ and the marginal law of $Y^M$ is uniform on $(cd)$, we have 
\[\PP[X^{M}\in(\tilde a\tilde b), Y^{M}\in(\tilde c\tilde d)]=\PP[X^{M}\in[\tilde a\tilde b],Y^{M}\in[\tilde c\tilde d]].\]
Therefore, 
\[\PP[X^{M}\in(\tilde a\tilde b), Y^{M}\in(\tilde c\tilde d)]
=\int_{f(\tilde a)}^{f(\tilde b)}\frac{\partial_n\lambda_{\tilde c,\tilde d}(x)}{\partial_n\lambda_{c,d}(x)}dx. \]
This gives~\eqref{eqn::YconditionalX_aux} and completes the proof. 
\end{proof}

\begin{proof}[Proof of Proposition~\ref{prop::cvg_pair_points}]
The convergence of $(X^M_{\delta}, Y^M_{\delta})$ and the law of $x^M=f(X^M)$ is derived Lemma~\ref{lem::X}. The conditional law of $Y^M$ given $X^M$ is derived in Lemma~\ref{lem::YconditionalX}. 
Since $\lambda_{\tilde{c}, \tilde{d}}(\cdot )$ is a bounded harmonic function on $[0,1]\times [0,\ii K]$ with boundary data illustrated in Figure~\ref{fig::boundary-conditions}, Poisson integral formula implies 
\[
\lambda_{\tilde{c}, \tilde{d}}(z )=\int_{f(\tilde{d})}^{f(\tilde{c})}P_K(z, i+\ii K)dy,\quad  z\in  (0,1)\times (0,\ii K).
\]
Similarly, 
\[
\lambda_{c, d}(z )=\int_{0}^{1}P_K(z, i+\ii K)dy,\quad z\in (0,1)\times (0,\ii K).
\]
Let us calculate the outer normal derivative of $P_K$: for $x, y\in (0,1)$, 
\begin{equation}\label{eqn::Poissonkernel_rect_partialn}
\partial_{n}P_K(z, y+\ii K)|_{z=x} =\frac{\pi}{4K}\sum_{n\in\Z}\left(\frac{1}{\cosh^2\left(\frac{\pi}{2K}\left(x-y-2n\right)\right)}+\frac{1}{\cosh^2\left(\frac{\pi}{2K}\left(x+y-2n\right)\right)}\right), 
\end{equation}
where the right-hand side is the same as $\rho_K(x,y)$ defined in~\eqref{eqn::jointdensity}. 
Hence for $x\in (0,1)$, we have
\[
\partial_n\lambda_{\tilde c,\tilde d}(x)=\int_{\Re f(\tilde{d})}^{\Re f(\tilde{c})} \left(\partial_{n}P_K(z, y+\ii K)|_{z=x} \right)dy=\int_{\Re f(\tilde{d})}^{\Re f(\tilde{c})}\rho_K(x, y)dy
\]
and
\begin{align*}
\partial_n\lambda_{c,d}(x)&=\int_0^1 \left(\partial_{n}P_K(z, y+\ii K)|_{z=x} \right)dy\\
&=\frac{\pi}{4K}\sum_{n\in\Z}\left(\int_0^1\frac{dy}{\cosh^2\left(\frac{\pi}{2K}\left(x-y-2n\right)\right)}+\int_0^1\frac{dy}{\cosh^2\left(\frac{\pi}{2K}\left(x+y-2n\right)\right)}\right)\\
&=\frac{\pi}{4K}\sum_{n\in\Z}\left(\int_{2n}^{2n+1}\frac{dr}{\cosh^2\left(\frac{\pi}{2K}\left(x-r\right)\right)}+\int_{2n-1}^{2n}\frac{dr}{\cosh^2\left(\frac{\pi}{2K}\left(x-r\right)\right)}\right)\\
&=\frac{\pi}{4K}\int_{-\infty}^{+\infty}\frac{dr}{\cosh^2\left(\frac{\pi}{2K}\left(x-r\right)\right)}=\frac{1}{2}\int_{-\infty}^{+\infty}\frac{dr}{\cosh^2(r)}=1. 
\end{align*}
Therefore, combing with~\eqref{eqn::YconditionalX_aux} we get
\begin{equation}\label{eqn::YconditionalXcontinuumaux}
\PP[Y^{M}\in (\tilde{c}\tilde{d})\cond X^{M}]=\frac{\partial_n\lambda_{\tilde c,\tilde d}(f(X^{M}))}{\partial_n\lambda_{c,d}(f(X^{M}))}=\partial_n\lambda_{\tilde c,\tilde d}(f(X^{M}))=\int_{\Re f(\tilde{d})}^{\Re f(\tilde{c})}\rho_K(x^M, y)dy. 
\end{equation}
This gives the density in~\eqref{eqn::jointdensity} and completes the proof. 
\end{proof}

\begin{corollary}\label{cor::YconditionalXdiscrete}
Fix a polygon $(\Omega; a, x, b, c, \tilde{c}, \tilde{d}, d)$ with seven marked points. Suppose that a sequence of medial polygons $(\Omega_{\delta}^{\diamond}; a_{\delta}^{\diamond}, x_{\delta}^{\diamond}, b_{\delta}^{\diamond}, c_{\delta}^{\diamond}, \tilde{c}_{\delta}^{\diamond}, \tilde{d}_{\delta}^{\diamond}, d_{\delta}^{\diamond})$ converges to $(\Omega; a, x, b, c, \tilde{c}, \tilde{d}, d)$ in the sense of~\eqref{eqn::topology}. Denote by $f=f_{(\Omega;a,b,c,d)}$ the conformal map from $\Omega$ onto $(0,1)\times (0,\ii K)$ which sends $(a,b,c,d)$ to $(0,1,1+\ii K,\ii K)$. We  extend $f$ continuously to the boundary. Then,
\begin{align*}\label{eqn::YconditionalXdiscretecor}
\lim_{\delta\to 0}\PP\left[Y_{\delta}^{M}\in(\tilde c_{\delta}\tilde d_{\delta})\cond X_{\delta}^{M}=x_{\delta}\right]=\int_{\Re f(\tilde{d})}^{\Re f(\tilde{c})}\rho_K(f(x), y)dy. 
\end{align*}
\end{corollary}
\begin{proof}
First let $\delta\to 0$ and then let $\eps\to 0$ in~\eqref{eqn::discretelaw}, combining with~\eqref{eqn::YconditionalXcontinuumaux}, we obtain the conclusion. 
\end{proof}
We emphasize that  in Corollary~\ref{cor::YconditionalXdiscrete}  we only need the assumption that $\partial\Omega$ is locally connected and we do not require extra regularity. 

\subsection{Proof of Theorem~\ref{thm::cvg_lerw}}
\label{subsec::cvg_lerw}
The joint law of $(X^M, Y^M)$ in Theorem~\ref{thm::cvg_lerw} is given in Proposition~\ref{prop::cvg_pair_points}, and to complete the proof of Theorem~\ref{thm::cvg_lerw}, it remains to show that the conditional law of $\gamma^M$ given $X^M$ is $\SLE_2(-1, -1; -1, -1)$. 
We follow the strategy in~\cite{ZhanLERW}. 
We fix the following notation in this section. 
\begin{itemize}
\item Fix $d=-\infty<a<b<c$.  
Denote by $K$ the conformal modulus of the quad $(\HH; a, b, c, \infty)$ and by $f(\cdot; a, b, c)$ the conformal map from $\HH$ onto $(0,1)\times (0,\ii K)$ sending $(a, b, c, \infty)$ to $(0,1,1+\ii K, \ii K)$.
\item Fix $d=-\infty<a<w<b<c$. Define 
\begin{equation}\label{eqn::PoissonH_def}
P(z;a,w,b,c):=P_{K}(f(z;a,b,c),f(w;a,b,c)),\quad \forall z\in \HH, 
\end{equation}
where $P_{K}$ is given in~\eqref{eqn::Poissonkernel_rect}. Note that $P(\cdot;\cdot,\cdot,\cdot,\cdot)$ is smooth on $\HH\times\{(a,w,b,c)\in\R^{4}:a<w<b<c\}$. It is the Poisson kernel on $\HH$ with the boundary data: 
\begin{equation}\label{eqn::PoissonH_boundarydata}
P(\cdot; a, w, b, c)=0,\text{ on }(a,w)\cup (w, b)\cup (c, \infty); \quad \partial_{n}P(\cdot; a, w, b, c)=0,\text{ on }(-\infty, a)\cup (b, c); 
\end{equation}
and the normalization: 
\begin{equation}\label{eqn::PoissonH_normalization}
\int_{c}^{\infty}(\partial_{n}P(z;a,w,b,c)|_{z=x})dx=1.
\end{equation}
\end{itemize}  

The strategy is as follows: first, we show that the Poisson kernel satisfies a certain PDE in Lemma~\ref{lem::PoissonH_pde}; then we show that the conditional density in~\eqref{eqn::conditionalYgivenX} gives a martingale observable for $\gamma^M$ in Lemma~\ref{lem::lerw_observable}. With these two lemmas at hand, we obtain the driving function from the martingale observable. This last step involves a non-trivial calculation where Lemma~\ref{lem::PoissonH_pde} plays a crucial role. 

\begin{lemma}\label{lem::PoissonH_pde}
For $a<w<b<c$ and $z=x+\ii y\in\HH$, 
consider the function $P(z; a, w, b, c)$ in~\eqref{eqn::PoissonH_def}. Denote by $\partial_x$ the partial derivative with respect to the real part of the first (complex) variable and by $\partial_y$ the partial derivative with respect to the imaginary part of the first (complex) variable. Define 
\[\LD:=\frac{2}{a-w}\partial_a+\frac{2}{b-w}\partial_b+\frac{2}{c-w}\partial_c+2\frac{f''(w;a,b,c)}{f'(w;a,b,c)}\partial_w+\partial_w^2+\Re\left(\frac{2}{z-w}\right)\partial_x+\Im\left(\frac{2}{z-w}\right)\partial_y.\]
Then, we have $\LD P(z;a,w,b,c)=0$.
\end{lemma}

To prove Lemma~\ref{lem::PoissonH_pde}, we define 
\begin{equation}\label{eqn::PoissonH_auxV}
\LV(z):=\LD P(z;a,w,b,c). 
\end{equation}
The goal is to show $\LV\equiv0$. 
To this end, we will show that $\LV$ is harmonic in $\HH$ and has the same boundary data as $P(\cdot; a,w,b,c)$ in Lemma~\ref{lem::PoissonH_aux_harm}. We will show that 
$\LV$ is bounded near $a,b,c,\infty$ in Lemma~\ref{lem::PoissonH_aux_bounded}.  This implies that $\LV$
is a bounded harmonic function such that
\[
\LV=0,\text{ on }(a,b)\cup (c, \infty); \quad \partial_{n}\LV=0,\text{ on }(-\infty, a)\cup (b, c).
\]
 By the same argument used  in the proof of Lemma~\ref{lem::harmonic_unique}, we get that $\LV\equiv0$ and then complete the proof of Lemma~\ref{lem::PoissonH_pde}. 

\begin{lemma}\label{lem::PoissonH_aux_harm}
The function $\LV$ in~\eqref{eqn::PoissonH_auxV} is harmonic in $\HH$ and has the same boundary data as $P(\cdot; a,w,b,c)$. 
\end{lemma}
\begin{proof}
First, we show that $\LV(\cdot)$ is harmonic in $\HH$. Note that by the explicit form of $P(z;a,w,b,c)$ in~\eqref{eqn::PoissonH_def} and $P_K$ in~\eqref{eqn::Poissonkernel_rect}, the sum over terms $n\neq 0$ is smooth on $\overline\HH\times\{(a,w,b,c)\in\R^{4}:a<w<b<c\}$ and it is analytic on $\overline\HH\setminus\{a,b,c\}$ when $a<w<b<c$ is fixed. For the term $n=0$, by subtracting $\Im \frac{K}{\pi f'(w;a,b,c)}\frac{1}{(z-w)}$, we get a function which is smooth on $\overline\HH\times\{(a,w,b,c)\in\R^{4}:a<w<b<c\}$ and  analytic on $\overline\HH\setminus\{a,b,c\}$ when $a<w<b<c$ is fixed.
Therefore we may write
\[
P(z;a,w,b,c)=\Im \left(G(z; a, w, b,c)+\frac{K}{\pi f'(w;a,b,c)}\frac{1}{(z-w)}\right),\quad\forall z\in\HH, w\in (a, b),
\]
where $G(\cdot;\cdot,\cdot,\cdot,\cdot)$ is a smooth function function on $\overline\HH\times\{(a,w,b,c)\in\R^{4}:a<w<b<c\}$. Moreover,  $G$ is analytic on $\overline\HH\setminus\{a,b,c\}$ when $a<w<b<c$ is fixed,
Then, we get
\begin{align*}
&\left(\Re\left(\frac{2}{z-w}\right)\partial_x+\Im\left(\frac{2}{z-w}\right)\partial_y\right)P(z;a,w,b,c)\\
=&\Im \left(\frac{2G'(z; a,w,b,c)}{z-w}-\frac{2K}{\pi f'(w; a,b,c)(z-w)^{3}}\right),\\
&\partial_wP(z; a, w, b,c)\\
=&\Im\left(\partial_w G(z;a,w,b,c)+\frac{K}{\pi f'(w;a,b,c)}\frac{1}{(z-w)^2}-\frac{K}{\pi}\frac{f''(w;a,b,c)}{f'(w;a,b,c)^2}\frac{1}{z-w}\right),\\
&\partial_w^2P(z; a, w, b, c)\\
=&\Im\left(\partial^2_{w}G(z;a,w,b,c)-\frac{K}{\pi}\frac{f'''(w;a,b,c)f'(w;a,b,c)-2f''(w;a,b,c)^2}{f'(w;a,b,c)^3}\frac{1}{z-w}\right)\\
&+\Im\left(\frac{2K}{\pi f'(w;a,b,c)}\frac{1}{(z-w)^3}-\frac{2K}{\pi}\frac{f''(w;a,b,c)}{f'(w;a,b,c)^2}\frac{1}{(z-w)^2}\right).
\end{align*}
Therefore, we can write
\begin{equation}\label{eqn::PoissonH_auxV_harm}
\LV(z)=\Im\left(G_{1}(z)+\frac{G_{2}(z)}{z-w}\right),
\end{equation}
where $G_1$ and $G_2$ are analytic functions on $\overline\HH\setminus\{a,b,c\}$. This implies that $\LV$ is harmonic.

Next, we show that $\LV(\cdot)$ has the same boundary data as $P(\cdot;a,w,b,c)$. From~\eqref{eqn::PoissonH_boundarydata}, we have $P(z; a, w, b, c)=0$ for all $a<z\neq w<b<c$. Thus 
\begin{align*}
&\ell P(z;a,w,b,c)=0,\quad \forall a<z\neq w<b<c,\quad\text{for all }\ell=\partial_a, \partial_b, \partial_c, \partial_w,\partial^2_w, \partial_x.
\end{align*}
Therefore, $\LV(\cdot)=0$ on $(a,w)\cup(w,b)$. Similarly, since $P(z; a,w,b,c)=0$ for all $a<w<b<c<z$, we have $\LV(\cdot)=0$ on $(c,+\infty)$. Since $\partial_{n}P(x; a,w,b,c)=0$ for all $a<w<b<x<c$ or $x<a<w<b<c$, we have
\begin{align*}
\partial_{n}\ell P(x; a,w,b,c)=&\ell\partial_{n} P(x; a,w,b,c)=0,\quad\text{for all }\ell=\partial_a,\partial_b, \partial_c,\partial_w,\partial_w^2,\\
\partial_{n}\left(\Re\frac{2}{z-w}\partial_xP(z; a, w, b, c)\right)\Big|_{z=x}=&\frac{2}{x-w}\partial_x\partial_{n}P(x;a,w,b,c)=0,\\
\partial_{n}\left(\Im\frac{2}{z-w}\partial_y P(z; a,w,b,c)\right)\Big|_{z=x}=&\frac{-2}{(x-w)^2}\partial_{n}P(x;a,w,b,c)=0, \\&\forall a<w<b<x<c\text{ or }x<a<w<b<c. 
\end{align*}
Here the interchange of $\partial_n$ and $\partial_a, \partial_b, \partial_c, \partial_w, \partial_x$ is legal due to the explicit form of $P(z;a,w,b,c)$ in~\eqref{eqn::PoissonH_def}. Thus, $\partial_{n}\LV(\cdot)=0$ on $(b,c)\cup(-\infty,a)$. This completes the proof. 
\end{proof}

\begin{lemma}\label{lem::PoissonH_aux_bounded}
The function $\LV$ in~\eqref{eqn::PoissonH_auxV} is bounded near $a, b, c, \infty$.  
\end{lemma}
\begin{proof}
We first investigate its behavior around $a$. 
From~\eqref{eqn::exactform}, we have
\[f(z; a,b,c)=\LK\left(\arcsin\sqrt{\frac{z-a}{b-a}}, \frac{b-a}{c-a}\right)/\LK\left(\frac{b-a}{c-a}\right),\]
and
\[K=K(a,b,c)=\Im\, \LK\left(\arcsin\sqrt{\frac{c-a}{b-a}}, \frac{b-a}{c-a}\right)/\LK\left(\frac{b-a}{c-a}\right).\]
Note that $f(\cdot;\cdot,\cdot,\cdot)$ is smooth on $\overline\HH\setminus\{a,b,c\}\times\{(a,b,c)\in \R^3:a<b<c\}$ and $K(\cdot,\cdot,\cdot)$ is smooth on $\{(a,b,c)\in \R^3:a<b<c\}$. This implies that $\partial_w P$ and $\partial^2_w P$ are continuous at $a$. Moreover, for all $\ell\in\{\partial_a,\partial_b, \partial_c,\partial_x,\partial_y\}$, we have
\begin{align*}
\ell P(z;a,w,b,c)=&\Im \sum_{n\in \Z}\frac{1}{\sinh^2\left(\frac{\pi}{2K}(2n-f(w;a,b,c)+f(z;a,b,c))\right)}\ell\left(\frac{\pi}{4K}(f(z;a,b,c)-f(w;a,b,c))\right)\\
&-\Im \sum_{n\in \Z}\frac{1}{\sinh^2\left(\frac{\pi}{2K}(2n-f(w;a,b,c)-\overline f(z;a,b,c))\right)}\ell\left(\frac{\pi}{4K}(\overline f(z;a,b,c)+f(w;a,b,c))\right).
\end{align*}
Denote by $\tilde z:=\arcsin\sqrt{(z-a)/(b-a)}$ and $s:=(b-a)/(c-a)$. We have
\begin{align*}
&\partial_b f(z;a,b,c)=-\frac{\sqrt{(c-a)(z-a)}}{2(b-a)\LK(s)\sqrt{(c-z)(b-z)}}+\frac{\partial_x \LK(\tilde z,s)}{(c-a)\LK(s)}-\frac{\LK(\tilde z,s)\LK'(s)}{(c-a)\LK^{2}(s)};\\
&\partial_c f(z;a,b,c)=-\frac{(b-a)\partial_x\LK(\tilde z,s)}{(c-a)^2\LK(s)}+\frac{(b-a)\LK(\tilde z,s)\LK'(s)}{(c-a)^2\LK^2(s)}.
\end{align*}
This implies that 
\begin{equation}\label{eqn::PoissonH_aux_bounded1}
\partial_b P(z;a,w,b,c)\to 0, \quad \partial_c P(z;a,w,b,c)\to 0, \quad \text{as }z\to a.
\end{equation}
A direct calculation implies 
\begin{align*}
\left|\partial_y P(z;a,w,b,c)\Im\left(\frac{2}{z-w}\right)\right|&=\left|\frac{\sqrt{c-a}}{2\sqrt{(b-z)(c-z)(z-a)}\LK(s)}\Im \frac{2}{z-w}\right|\\
&=\frac{\sqrt{c-a}\times\Im z}{|\sqrt{(b-z)(c-z)(z-a)}|\LK(s)|z-w|^2}.
\end{align*}
Thus
\begin{equation}\label{eqn::PoissonH_aux_bounded2}
\partial_y P(z;a,w,b,c)\Im\left(\frac{2}{z-w}\right)\to 0,\quad \text{as }z\to a. 
\end{equation}
Note that 
\begin{align*}
&\left(\frac{2}{a-w}\partial_a+\Re\left(\frac{2}{z-w}\right)\partial_x\right)f(z;a,w,b,c)\\
=&\frac{\sqrt{c-a}}{\sqrt{(c-z)(b-z)(z-a)}\LK(s)}\left(\frac{z-a}{(b-a)(a-w)}-\Re \frac{z-a}{(z-w)(a-w)}\right)\\
&+\frac{(b-c)\partial_x \LK(\tilde z,s)}{(c-a)^2\LK(s)}-\frac{(b-c)\LK(\tilde z,s)\LK'(s)}{(c-a)^2\LK^2(s)}.
\end{align*}
Therefore we have
\begin{equation}\label{eqn::PoissonH_aux_bounded3}
\left(\frac{2}{a-w}\partial_a+\Re\left(\frac{2}{z-w}\partial_x\right)\right)P(z;a,w,b,c)\to 0, \quad\text{as }z\to a. 
\end{equation}
Recall that $\partial_w P$ and $\partial^2_w P$ are continuous at $a$, combining with~\eqref{eqn::PoissonH_aux_bounded1}, \eqref{eqn::PoissonH_aux_bounded2} and~\eqref{eqn::PoissonH_aux_bounded3}, we see that $\LV(z)$ remains bounded as $z\to a$. Similarly, we may show that $\LV(z)$ is also bounded near $b, c, \infty$.
\end{proof}

\begin{proof}[Proof of Lemma~\ref{lem::PoissonH_pde}]
When there is no ambiguity, we write $\partial_{n}P(z; a, w, b, c)|_{z=x}$ as $\partial_{n}P(x; a,w,b,c)$. 
The goal is to show $\LV(z)=0$ for every $z\in\HH$. To this end, we evaluate the value of $\int_c^{\infty}\partial_{n}\LV(x)dx$. On the one hand,  we consider the following function:
\begin{equation*}\label{eqn::PoissonH_aux_bounded}
\widetilde{\LV}(\cdot):=\LV(\cdot)-\frac{\pi}{K}G_{2}(w)f'(w;a,b,c)P(\cdot;a,w,b,c),
\end{equation*}
where $G_2$ is defined as in~\eqref{eqn::PoissonH_auxV_harm}. By Lemma~\ref{lem::PoissonH_aux_harm}, $\widetilde{\LV}$ is harmonic on $\HH$ and has the same boundary data as $P(\cdot;a,w,b,c)$. Moreover,  by the construction $\widetilde{\LV}$ is bounded near $w$ and by Lemma~\ref{lem::PoissonH_aux_bounded} $\widetilde{\LV}$ is bounded near $a, b, c, \infty$. Thus, $\widetilde{\LV}$ is bounded in $\overline{\HH}$ and hence $\widetilde{\LV}\equiv 0$.
This implies that 
\begin{align}\label{eqn::Poisson_pde_aux}
\int_{c}^{+\infty}\partial_{n}\LV(x)dx=\frac{\pi G_{2}(w)f'(w;a,b,c)}{K}\int_{c}^{+\infty}\partial_{n}P(x;a,w,b,c)dx=\frac{\pi}{K}G_{2}(w)f'(w;a,b,c),  
\end{align}
where the second equality  is due to~\eqref{eqn::PoissonH_normalization}. 
On the other hand,  we have  
\begin{align*}
&\int_c^{\infty}\partial_{n}\ell P(x; a, w, b, c)dx =\ell \int_c^{\infty}\partial_{n}P(x; a, w, b, c)dx=0,\quad \text{for }\ell=\partial_a, \partial_b, \partial_w, \partial_w^2; \tag{by~\eqref{eqn::PoissonH_normalization}}
\\
&\int_c^{\infty}\partial_{n}\frac{2}{c-w}\partial_c P(x; a, w, b, c)dx=\frac{2}{c-w}\int_c^{\infty}\partial_c\partial_{n} P(x; a, w, b, c)dx=\frac{2}{c-w}\partial_{n}P(c; a, w, b, c); \tag{by~\eqref{eqn::PoissonH_normalization}}\\
&\int_c^{\infty}\left(\partial_{n}\left(\Re\frac{2}{z-w}\partial_xP(z; a, w, b, c)\right)\right)\Big|_{z=x} dx\\
&=\int_c^{\infty}\frac{2}{x-w}\partial_x\partial_{n}P(x; a, w, b, c)dx=\frac{-2}{c-w}\partial_{n}P(c; a, w, b, c)+\int_c^{\infty}\frac{2}{(x-w)^2}\partial_{n}P(x; a, w, b, c)dx; \\
&\int_c^{\infty}\left(\partial_{n}\left(\Im\frac{2}{z-w}\partial_y P(z; a,w,b,c)\right)\right)\Big|_{z=x}dx =\int_c^{\infty}\frac{-2}{(x-w)^2}\partial_{n}P(x;a,w,b,c)dx. 
\end{align*}
Therefore, 
\begin{align*}
\int_{c}^{+\infty}\partial_{n}\LV(x)dx=\int_c^{\infty}\partial_{n}\LD P(z; a,w,b,c)\Big|_{z=x}dx=0.
\end{align*}
Comparing with~\eqref{eqn::Poisson_pde_aux}, we have $G_{2}(w)=0$. Consequently, $\LV\equiv 0$ as desired.  
\end{proof}

\begin{corollary}
For $a<w<b<c<x$, define 
\begin{equation}\label{eqn::Poisson_partialn_def}
F(x; a,w,b,c)=\partial_{n}P(z; a, w, b, c)|_{z=x}. 
\end{equation}
Then we have 
\begin{equation}\label{eqn::Poisson_partialn_PDE}
\left(\frac{2}{a-w}\partial_a+\frac{2}{b-w}\partial_b+\frac{2}{c-w}\partial_c+2\frac{f''(w; a,b,c)}{f'(w; a,b,c)}\partial_w+\partial_w^2+\frac{2}{x-w}\partial_x+\frac{-2}{(x-w)^2}\right)F=0.
\end{equation}
\end{corollary}
\begin{proof}
The PDE~\eqref{eqn::Poisson_partialn_PDE} can be obtained by taking $\partial_{n}$ in $\LD P=0$ from Lemma~\ref{lem::PoissonH_pde}. 
\end{proof}

\begin{lemma}\label{lem::lerw_observable}
Assume the same setup as in Theorem~\ref{thm::cvg_lerw}. Choose a conformal map $\phi$ from $\Omega$ onto $\HH$ such that $\phi(d)=\infty$ and $\phi(a)<\phi(b)<\phi(c)$. Denote by $(W_t, t\ge 0)$ the driving function of $\phi\left(\gamma^{M}\right)$ and by $(g_{t}, t\ge 0)$ the corresponding conformal maps. 
For any $x\in(\phi(c),+\infty)$, the process 
\[\left(g_t'(x)F(g_t(x);g_t(\phi(a)),W_t,g_t(\phi(b)),g_t(\phi(c))),\, t\ge 0\right)\] 
is a martingale up to the first time that $\gamma^M$ hits $(cd)$, where $F$ is defined in~\eqref{eqn::Poisson_partialn_def}. 
\end{lemma}
\begin{proof}
Fix two boundary points $\tilde{c}, \tilde{d}$ such that $a, b, c, \tilde{c}, \tilde{d}, d$ are in counterclockwise order. Choose a sequence of medial polygons $(\Omega_{\delta}^{\diamond} ;a_\delta^{\diamond},b_{\delta}^{\diamond},c_{\delta}^{\diamond}, \tilde{c}_{\delta}^{\diamond},\tilde{d}_{\delta}^{\diamond}, d_{\delta}^{\diamond})$ converges to $(\Omega;a,b,c,\tilde{c},\tilde{d}, d)$ in the sense of~\eqref{eqn::topology} and choose a sequence of conformal maps $\phi_{\delta}:\Omega_\delta\to\HH$ with $\phi_\delta(d_\delta)=\infty$ such that $\phi_\delta^{-1}$ converges to $\phi^{-1}$ uniformly on $\overline\HH$ as $\delta\to 0$.
By Theorem~\ref{thm::cvg_triple}, we have that $\gamma_\delta^M\to\gamma^M$ in law as $\delta\to 0$.  Couple $\{\gamma_\delta^M\}_{\delta>0}$ and $\gamma^M$ together such that $\gamma_\delta^M\to\gamma^M$ almost surely as $\delta\to 0$. Recall that $X^M=\gamma^M\cap (ab), Y^M=\gamma^M\cap (cd)$ and $X^M_{\delta}=\gamma^M_{\delta}\cap(a_{\delta}b_{\delta}), Y^M_{\delta}=\gamma^M_{\delta}\cap (c_{\delta}d_{\delta})$. 

We parameterize $\phi(\gamma^M)$ by the half-plane capacity and parameterize $\gamma^M$ such that $\phi(\gamma^M(t))=\phi(\gamma^M)(t)$. 
Denote by $T$ the first time that $\gamma^M$ hits $(cd)$. For $\eps>0$, define $T_{\epsilon}=\inf\{t:\dist(\gamma^M(t),(ba))=\epsilon\}$. 
For $t<T$, denote by $K_t$ the conformal modulus of the quad $(\HH; g_t(\phi(a)), g_t(\phi(b)), g_t(\phi(c)),\infty)$ and by $f_t$ the conformal map from $\HH$ onto $(0,1)\times (0,\ii K_t)$ sending $(g_t(\phi(a)), g_t(\phi(b)), g_t(\phi(c)),\infty)$ to $(0,1,1+\ii K_t, \ii K_t)$. 

We parameterize $\gamma_{\delta}^M$ similarly, define $T^{\delta}_{\eps}=\inf\{t: \dist(\gamma^M_{\delta}, (b_{\delta}a_{\delta}))=\eps\}$.  
We may assume $T_\epsilon^\delta\to T_{\epsilon}$ almost surely as $\delta\to 0$ by considering the continuous modification, see details in~\cite[Appendix B]{KarrilaMultipleSLELocalGlobal}and~\cite{KarrilaUSTBranches}. For every $t<T^{\delta}_{\eps}$, define $\Omega_{\delta}(t):=\Omega_{\delta}\setminus\gamma_{\delta}^M[0,t]$. The boundary conditions of $\Omega_{\delta}(t)$ are inherited from $(\Omega_{\delta}; a_{\delta}, b_{\delta}, c_{\delta}, d_{\delta})$ and $\gamma_{\delta}^M[0,t]$ as follows: $(a_{\delta} b_{\delta})\cup\gamma_{\delta}^M[0,t]$ is wired and $(c_{\delta} d_{\delta})$ is wired. Consider the UST in the quad $(\Omega_{\delta}(t); a_{\delta}, b_{\delta}, c_{\delta}, d_{\delta})$ with such boundary condition.  Let $\gamma^M_{\delta,t}$ be the unique simple path from $(a_\delta b_\delta)\cup\gamma_\delta^M[0,t]$ to $(c_\delta d_\delta)$ in the UST on $\Omega_{\delta}(t)$. Denote by $X_{\delta}^M(t)$ the starting point of $\gamma^M_{\delta,t}$ on $(a_\delta b_\delta)\cup\gamma_\delta^M[0,t]$ and by $Y_{\delta}^M(t)$ the ending point of $\gamma^M_{\delta,t}$ on $(c_\delta d_\delta)$.
For any bounded continuous function $R$ on curves, we have
\begin{align*}
&\E\left[\one_{\{Y_\delta^M\in(\tilde c_\delta\tilde d_\delta)\}}R(\gamma_\delta^M[0,t\wedge T_\epsilon^\delta])\right]\\
&=\E\left[\PP\left[Y_{\delta}^M(t\wedge T_\epsilon^\delta)\in(\tilde c_\delta\tilde d_\delta)\cond X_{\delta}^M(t\wedge T_{\eps}^{\delta})=\gamma^M_\delta(t\wedge T_\epsilon^\delta)\right]R(\gamma_\delta^M[0,t\wedge T_\epsilon^\delta])\right].
\end{align*}
From the convergence of $\gamma_{\delta}^M$ to $\gamma^M$, we have
\begin{align}\label{eqn::leftside}
\E\left[\one_{\{Y_\delta^M\in(\tilde c_\delta\tilde d_\delta)\}}R(\gamma_\delta^M[0,t\wedge T_\epsilon^\delta])\right]\to\E\left[\one_{\{Y^M\in(\tilde c\tilde d)\}}R(\gamma^M[0,t\wedge T_\epsilon])\right],\quad \text{as }\delta\to 0. 
\end{align}
From Corollary~\ref{cor::YconditionalXdiscrete}, we have
\[\PP\left[Y_{\delta}^M(t\wedge T_\epsilon^\delta)\in(\tilde{c}_{\delta}\tilde{d}_{\delta})\cond X_{\delta}^M(t\wedge T_\epsilon^\delta)=\gamma^M_\delta(t\wedge T_\epsilon^\delta)\right]\to\int_{f_{t\wedge T_\epsilon}(g_{t\wedge T_\epsilon}(\phi(\tilde d)))}^{f_{t\wedge T_\epsilon}(g_{t\wedge T_\epsilon}(\phi(\tilde c)))}\rho_{K_{t\wedge T_\epsilon}}(f_{t\wedge T_\epsilon}(W_{t\wedge T_\epsilon}),\Re\, y)dy,\]
where $\rho_K$ is defined in~\eqref{eqn::jointdensity}. 
Thus, by  bounded convergence theorem, we have
\begin{equation}\label{eqn::rightside}
\begin{split}
&\E\left[\PP\left[Y_{\delta}^M(t\wedge T_\epsilon^\delta)\in(\tilde c_\delta\tilde d_\delta)\cond X_{\delta}^M(t\wedge T_\epsilon^\delta)=\gamma_\delta^M(t\wedge T_\epsilon^\delta)\right]R(\gamma_\delta^M[0,t\wedge T_\epsilon^\delta])\right]\\
&\to\E\left[R(\gamma^{M}[0,t\wedge T_\epsilon])\int_{f_{t\wedge T_\epsilon}(g_{t\wedge T_\epsilon}(\phi(\tilde d)))}^{f_{t\wedge T_\epsilon}(g_{t\wedge T_\epsilon}(\phi(\tilde c)))}\rho_{K_{t\wedge T_{\epsilon}}}(f_{t\wedge T_\epsilon}(W_{t\wedge T_\epsilon}), \Re\, y)dy\right],\quad \text{as }\delta\to 0.
\end{split}
\end{equation}
Combining~\eqref{eqn::leftside} and~\eqref{eqn::rightside}, we have
\begin{align*}
&\E\left[\one_{\{Y^M\in(\tilde c\tilde d)\}}R(\gamma^M([0,t\wedge T_\epsilon]))\right]\\
&=\E\left[R(\gamma^{M}([0,t\wedge T_\epsilon]))\int_{f_{t\wedge T_\epsilon}(g_{t\wedge T_\epsilon}(\phi(\tilde d)))}^{f_{t\wedge T_\epsilon}(g_{t\wedge T_\epsilon}(\phi(\tilde c)))}\rho_{K_{t\wedge T_\epsilon}}(f_{t\wedge T_\epsilon}(W_{t\wedge T_\epsilon}),\Re\, y)dy\right].
\end{align*}
This implies that the process 
\[\left(\int_{f_{t}(g_{t}(\phi(\tilde d)))}^{f_{t}(g_{t}(\phi(\tilde c)))}\rho_{K_{t}}(f_{t}(W_{t}),\Re\, y)dy, \, t\ge 0\right)\] 
is a martingale up to $T_{\eps}$. Thus, the process
\[\left((f_{t}\circ g_{t})'(x)\rho_{K_{t}}(f_{t}(W_{t}),\Re f_{t}(g_{t}(x))), \, t\ge 0\right)\] 
is a martingale up to $T_{\eps}$ for every $x\in(\phi(c),+\infty)$. Combining~\eqref{eqn::jointdensity}, \eqref{eqn::Poissonkernel_rect_partialn}, \eqref{eqn::PoissonH_def} and~\eqref{eqn::Poisson_partialn_def}, we have
\begin{align*}
F(x; \phi(a), W_0, \phi(b), \phi(c))=f'(x)\rho_K(f(W_0), \Re f(x)). 
\end{align*}
Thus, the process 
\begin{align*}
\left(g_t'(x)F(g_t(x); g_t(\phi(a)), W_t, g_t(\phi(b)), g_t(\phi(c))), \, t\ge 0\right)
\end{align*}
is a martingale up to $T_{\eps}$ for every $x\in(\phi(c),+\infty)$. From Proposition~\ref{prop::tightness_LERW}, the curve $\gamma^M$ intersects $\partial\Omega$ only at its two ends almost surely. Thus $T_\epsilon\to T$ as $\epsilon\to 0$. This completes the proof.
\end{proof}
\begin{proof}[Proof of Theorem~\ref{thm::cvg_lerw}]
The joint law of $(X^M, Y^M)$ is derived in Proposition~\ref{prop::cvg_pair_points}. It remains to show that the conditional law of $\gamma^M$ given $X^M$ is $\SLE_2(-1,-1;-1,-1)$. To this end, we may assume $\Omega=\HH$ with $d=\infty$ and $a<b<c$ and parameterize $\gamma^M$ by the half-plane capacity. Denote by $T$ the first time that $\gamma^M$ hits $(c,\infty)$. 
Denote by $(W_t, t\ge 0)$ the driving function of $\gamma^M$ and by $(g_t, t\ge 0)$ the corresponding conformal maps. For $a<w<b<c<x$, define $F(x; a, w, b, c)$ as in~\eqref{eqn::Poisson_partialn_def}. Lemma~\ref{lem::lerw_observable} tells that the process 
\[(g_t'(x)F(g_t(x); g_t(a), W_t, g_t(b), g_t(c)),\, t\ge 0)\]
is a martingale up to $T$. 
By the same argument in the proof of Lemma~\ref{lem::drivingfunction}, we can deduce that $(W_t, t\ge 0)$ is a semimartingale. Denote by $L_t$ the drift term of $W_t$. By It\^{o}'s formula, we have
\begin{align*}
\left(\frac{2}{g_t(x)-W_t}\partial_x+\frac{2}{g_t(a)-W_t}\partial_a+\frac{2}{g_t(b)-W_t}\partial_b+\frac{2}{g_t(c)-W_t}\partial_c+\frac{-2}{(g_t(x)-W_t)^2}\right)Fdt\\+\partial_wFdL_t+\frac{1}{2}\partial_w^2Fd\langle W\rangle_t=0.
\end{align*}
Combining with~\eqref{eqn::Poisson_partialn_PDE}, we have 
\begin{align*}
\partial_w F\left(dL_t-2\frac{f''(W_t; g_t(a), g_t(b), g_t(c))}{f'(W_t; g_t(a), g_t(b), g_t(c))}dt\right)+\frac{1}{2}\partial_w^2 F\left(d\langle W\rangle_t-2dt\right)=0. 
\end{align*}
From~\eqref{eqn::exactform} and~\eqref{eqn::elliptic_derivatives}, we have
\begin{equation*}\label{eqn::lerw_observable_drift_aux}
2\frac{f''(w;a,b,c)}{f'(w;a,b,c)}=\frac{1}{a-w}+\frac{1}{b-w}+\frac{1}{c-w}.
\end{equation*}
Thus, it simplifies as 
\begin{equation}\label{eqn::lastfunction}
\partial_wF\left(dL_t-\left(\frac{1}{g_t(a)-W_t}+\frac{1}{g_t(b)-W_t}+\frac{1}{g_t(c)-W_t}\right)dt\right)+\frac{1}{2}\partial_w^2F(d\langle W\rangle_t-2dt)=0. 
\end{equation}
Note that almost surely, ~\eqref{eqn::lastfunction} holds for all $x\in\QQ\cap(c,+\infty)$. By the continuity, almost surely, it holds for all $x\in(c,+\infty)$. Now we fix $(a,b,c)$ and $t$. Define \[S_1(x):=\partial_wF(g_t(x);g_t(a),W_t,g_t(b),g_t(c)),\quad S_2(x):=\partial^2_wF(g_t(x);g_t(a),W_t,g_t(b),g_t(c)).\] It suffices to prove 
\begin{equation}\label{eqn::nolinear}
\exists x, x'\in (c, \infty)\text{ such that }
S_1(x)S_2(x')\neq S_2(x)S_1(x'). 
\end{equation}
Assume this is true, then we have
\[dL_t=\left(\frac{1}{g_t(a)-W_t}+\frac{1}{g_t(b)-W_t}+\frac{1}{g_t(c)-W_t}\right)dt\quad\text{ and }\quad d\langle W\rangle_t=2dt.\]
This shows that $\gamma^M$ is $\SLE_2(-1; -1, -1)$ as desired. 

It remains to show~\eqref{eqn::nolinear}. 
Denote $f(\cdot)=f(\cdot; a, b, c)$. Note that
{\tiny\begin{align*}
S_1(x)=&-2f’(g_t(x))\frac{\pi}{4K}\times\frac{\pi}{2K}f'(W_t)\times
\sum_{n\in\Z}\left(\frac{\sinh\left(\frac{\pi}{2K}\left(f(W_t)-\Re f(g_t(x))-2n\right)\right)}{\cosh^3\left(\frac{\pi}{2K}\left(f(W_t)-\Re f(g_t(x))-2n\right)\right)}+\frac{\sinh\left(\frac{\pi}{2K}\left(f(W_t)+\Re f(g_t(x))-2n\right)\right)}{\cosh^3\left(\frac{\pi}{2K}\left(f(W_t)+\Re f(g_t(x))-2n\right)\right)}\right),\\
S_2(x)=&\frac{f''(W_t)}{f'(W_t)}S_1(x)\\
&+2f'(g_t(x))\frac{\pi}{4K}\times\left(\frac{\pi}{2K}f'(W_t)\right)^2
\times\sum_{n\in\Z}\left(\frac{2\sinh^2\left(\frac{\pi}{2K}\left(f(W_t)-\Re f(g_t(x))-2n\right)\right)-1}{\cosh^4\left(\frac{\pi}{2K}\left(f(W_t)-\Re f(g_t(x))-2n\right)\right)}+\frac{2\sinh^2\left(\frac{\pi}{2K}\left(f(W_t)+\Re f(g_t(x))-2n\right)\right)-1}{\cosh^4\left(\frac{\pi}{2K}\left(f(W_t)+\Re f(g_t(x))-2n\right)\right)}\right).
\end{align*}}
Define
\[R_1(z):=\sum_{n\in\Z}\left(\frac{\cosh\left(\frac{\pi}{2K}\left(f(W_t)-z-2n\right)\right)}{\sinh^3\left(\frac{\pi}{2K}\left(f(W_t)-z-2n\right)\right)}+\frac{\cosh\left(\frac{\pi}{2K}\left(f(W_t)+z-2n\right)\right)}{\sinh^3\left(\frac{\pi}{2K}\left(f(W_t)+z-2n\right)\right)}\right)\]
and
\[R_2(z):=-\sum_{n\in\Z}\left(\frac{2\cosh^2\left(\frac{\pi}{2K}\left(f(W_t)-z-2n\right)\right)+1}{\sinh^4\left(\frac{\pi}{2K}\left(f(W_t)-z-2n\right)\right)}+\frac{2\cosh^2\left(\frac{\pi}{2K}\left(f(W_t)+z-2n\right)\right)+1}{\sinh^4\left(\frac{\pi}{2K}\left(f(W_t)+z-2n\right)\right)}\right).\]
If~\eqref{eqn::nolinear} is false, then $S_1(x)/S_2(x)$ is constant for $x>c$. Thus, there exists $\lambda$ (which is random) such that $(R_1-\lambda R_2)(z)|_{z\in(\ii K,1+\ii K)}=0$. Since $R_1$ and $R_2$ are analytic functions in $(0,1)\times (0,\ii K)$, this implies that $R_1=\lambda R_2$ in $(0,1)\times (0,\ii K)$. This is a contradiction by considering the asymptotic of $R_1$ and $R_2$ when $z\to f(W_t)$: 
\[\lim_{z\to f(W_{t})}R_{1}(z)(z-f(W_{t}))^{3}=\left(\frac{-2K}{\pi}\right)^{3},\quad \lim_{z\to f(W_{t})}R_{2}(z)(z-f(W_{t}))^{4}=-3\left(\frac{2K}{\pi}\right)^{4}.\]
This completes the proof.
\end{proof} 

\subsection{Consequences}
\label{subsec::Dobrushintoquad}
In this section, we complete the proof for Propositions~\ref{prop::fromSLEtohSLE} and~\ref{prop::hSLE_right_proba} and Corollary~\ref{cor::sle8duality}. 
\begin{proof}[Proof of Proposition~\ref{prop::hSLE_right_proba}]
The conclusion is immediate from Theorem~\ref{thm::cvg_triple} and~\eqref{eqn::hSLE_right_proba_aux}. 
\end{proof}

The proof for Proposition~\ref{prop::fromSLEtohSLE} and Corollary~\ref{cor::sle8duality} bases on the following observation in the discrete for UST. 
Fix a Dobrushin domain $(\Omega; c, d)$ such that $\partial\Omega$ is $C^1$ and simple. Suppose $(\Omega_{\delta}; c_{\delta}, d_{\delta})$ is an approximation of $(\Omega; c, d)$ on $\delta\Z^2$ as in Section~\ref{subsec::ust_dobrushin}. Let $\LT_{\delta}$ be the UST in $(\Omega_{\delta}; c_{\delta}, d_{\delta})$ with $(c_{\delta}d_{\delta})$ wired. Denote by $\eta_{\delta}$ the associated Peano curve along $\LT_{\delta}$ from $d_{\delta}^{\diamond}$ to $c_{\delta}^{\diamond}$.  
Fix $a\in (dc)$ and let $a_{\delta}^{\diamond}$ be the medial vertex along $(c_{\delta}^{\diamond}d_{\delta}^{\diamond})$ nearest to $a$. Let $a_{\delta}\in V(\Omega_{\delta})$ be the primal vertex in $\Omega_{\delta}$ that is nearest to $a_{\delta}^{\diamond}$. Let $a_{\delta}^*$ (resp. $b_{\delta}^*$) be the dual vertex along $(d_{\delta}^*c_{\delta}^*)$ that is nearest to $a_{\delta}^{\diamond}$ and is closer to $d_{\delta}^*$ (resp. closer to $c_{\delta}^*$) along $(d_{\delta}^*c_{\delta}^*)$. See Figure~\ref{fig::single_to_triple}. We divide the Peano curve $\eta_{\delta}$ into two parts: denote by $\tilde{\eta}^L_{\delta}$ the part of $\eta_{\delta}$ from $d_{\delta}^{\diamond}$ to $a_{\delta}^{\diamond}$, and denote by $\eta^R_{\delta}$ the part of $\eta_{\delta}$ from $a_{\delta}^{\diamond}$ to $c_{\delta}^{\diamond}$. Denote by $\eta^L_{\delta}$ the time-reversal of $\tilde{\eta}^L_{\delta}$. 
There is a branch in $\LT_{\delta}$ connecting $a_{\delta}$ to $(c_{\delta}d_{\delta})$ and we denote it by $\gamma_{\delta}$. We parameterize $\gamma_{\delta}$ so that it starts from $a_{\delta}$ and terminates when it hits $(c_{\delta}d_{\delta})$.  We have the convergence of the triple $(\eta^L_{\delta}; \gamma_{\delta}; \eta^R_{\delta})$. 

\begin{figure}[ht!]
\begin{center}
\includegraphics[width=0.4\textwidth]{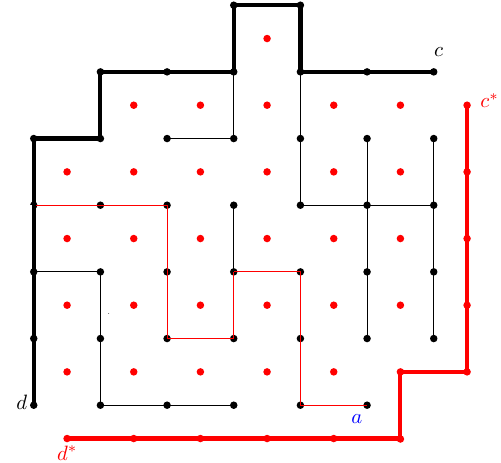}$\quad$
\includegraphics[width=0.4\textwidth]{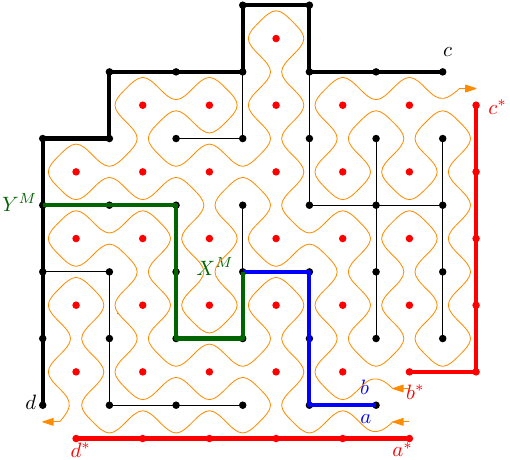}
\end{center}
\caption{\label{fig::single_to_triple} 
In the left panel, the solid edges in black are wired boundary arc $(cd)$, and the solid edges in red are dual-wired boundary arc $(d^*c^*)$. The thin edges are in the UST. The thin edges in red are the branch $\gamma$ in the tree connecting $a$ to $(cd)$. 
Suppose $T$ is the hitting time of $\gamma$ at $(cd)$ and $\tau$ is any stopping time before $T$.
In the right panel, the solid edges in blue are $\gamma[0,\tau]$ and the solid edges in green are $\gamma[\tau, T]$. The two orange curves are $\eta^L$ and $\eta^R$. }
\end{figure}

\begin{lemma}\label{lem::triple_in_single}
Fix a polygon $(\Omega; d, a, c)$ with three marked points such that $\partial\Omega$ is $C^1$ and simple. Suppose that a sequence of medial polygons $(\Omega_{\delta}^{\diamond}; d_{\delta}^{\diamond}, a_{\delta}^{\diamond}, c_{\delta}^{\diamond})$ converges to $(\Omega; d, a, c)$ in the sense of~\eqref{eqn::topology}. Then the triple $(\eta^L_{\delta}; \gamma_{\delta}; \eta^R_{\delta})$ converges weakly to a triple of continuous curves $(\eta^L; \gamma; \eta^R)$ whose law is characterized as follows. Let $\eta$ be an $\SLE_8$ in $\Omega$ from $d$ to $c$ and let $T_a$ be the first time that it swallows $a$. Then, the joint distribution of $(\eta^L; \eta^R)$ is the same as $(\eta(T_a-t), 0\le t\le T_a\,;\, \eta(t), t\ge T_a)$; and $\gamma=\eta^L\cap\eta^R$.  
\end{lemma}
\begin{proof}
First of all, 
recall that the tightness of $\{\eta_{\delta}^{L}\}_{\delta>0}$ and $\{\eta_{\delta}^{R}\}_{\delta>0}$ are given in the proof of Theorem~\ref{thm::ust_Dobrushin}. The tightness of $\{\gamma_{\delta}\}_{\delta>0}$ can be proved in the same way as in Proposition~\ref{prop::tightness_LERW}. Therefore, the triple $\{(\eta_{\delta}^{L};\gamma_{\delta};\eta_{\delta}^{R})\}_{\delta>0}$ is tight. 

Next, we determine the law of subsequential limits. Suppose $(\eta^{L};\gamma;\eta^{R})$ is any subsequential limit. There exists $\{\delta_n\}$ with $\delta_n\to 0$ as $n\to\infty$, such that $\eta^{L}_{\delta_n}\to\eta^{L}$ and $\gamma_{\delta_n}\to\gamma$ and $\eta^{R}_{\delta_n}\to\eta^{R}$ in law as $n\to\infty$. By Theorem~\ref{thm::ust_Dobrushin}, $\{\eta_{\delta_n}\}_n$ converges weakly to $\eta$ as $n\to\infty$. Thus, $(\eta^{L};\eta^{R})$ has the same law as $(\eta(T_a-t), 0\le t\le T_a\,;\, \eta(t), t\ge T_a)$. Since $\SLE_8$ is space filling, we have $\gamma=\eta^{L}\cap\eta^{R}$. This completes the proof.
\end{proof}

From the observation in Lemma~\ref{lem::triple_in_single}, we arrive at the following lemma. 
\begin{lemma}\label{lem::triple_conditional}
Fix a polygon $(\Omega; d, a, c)$ with three marked points such that $\partial\Omega$ is $C^1$ and simple. 
Let $\eta$ be an $\SLE_8$ in $\Omega$ from $d$ to $c$ and let $T_a$ be the first time that it swallows $a$. We define $\gamma:=\partial(\eta[0,T_a])\cap\Omega$ (here we view $\eta[0,T_a]$ as a compact set) and view $\gamma$ as a continuous simple curve starting from $a$ and terminating at some point in $(cd)$.
Denote by $\Omega^L$ and $\Omega^R$ the two connected components of $\Omega\setminus\gamma$ such that $\Omega^L$ has $d$ on the boundary and $\Omega^R$ has $c$ on the boundary.
The joint law of the triple  
\[(\eta(T_a-t), 0\le t\le T_a; \quad \gamma;\quad \eta(t), t\ge T_a)\]
can be characterized as follows: $\gamma$ is $\SLE_2(-1,-1;-1,-1)$ in $\Omega$ from $a$ to $(cd)$ with force points $(d,a^{-};a^{+},c)$; given $\gamma$, the conditional law of $(\eta(T_a-t), 0\le t\le T_a)$ is $\SLE_8$ in $\Omega^L$ from $a^{-}$ to $d$ and the conditional law of $(\eta(t), t\ge T_a)$ is $\SLE_8$ in $\Omega^R$ from $a^{+}$ to $c$, and $(\eta(T_a-t), 0\le t\le T_a)$ and $(\eta(t), t\ge T_a)$ are conditionally independent given $\gamma$.
\end{lemma}
\begin{proof}
 First, we derive the marginal law of $\gamma$. Choose a conformal map $\phi$ from $\Omega$ onto $\HH$ such that $\phi(a)=0$ and $\phi(d)=\infty$. Denote by $(W_t, t\ge 0)$ the driving function of $\phi\left(\gamma\right)$ and by $(g_{t}, t\ge 0)$ the corresponding conformal maps. By the argument in the proof of Lemma~\ref{lem::lerw_observable}, for any $x\in(\phi(c),+\infty)$, the process 
\[\left(g_t'(x)F(g_t(x);g_t(0^{-}),W_t,g_t(0^{+}),g_t(\phi(c))),\, t>0\right)\] 
is a martingale up to the first time that $\gamma$ hits $(cd)$ where $F$ is defined in~\eqref{eqn::Poisson_partialn_def}. For $\eps>0$, define $\tau_\epsilon:=\inf\{t:\dist(\phi(\gamma(t)),0)\ge \epsilon\}$. By the argument in the proof of Theorem~\ref{thm::cvg_lerw}, the conditional law of $(\gamma(t),t\ge\tau_{\epsilon})$ given $\gamma[0,\tau_{\epsilon}]$ is $\SLE_{2}(-1,-1;-1,-1)$ in $\Omega\setminus\gamma[0,\tau_{\epsilon}]$ from $\gamma(\tau_{\epsilon})$ to $(cd)$ with force points $(d,a^{-};a^{+},c)$. Let $\epsilon\to 0$, the law of $\gamma$ is $\SLE_2(-1,-1;-1,-1)$ in $\Omega$ from $a$ to $(cd)$ with force points $(d,a^{-};a^{+},c)$.

Second, the conditional law of $(\eta(T_a-t), 0\le t\le T_a; \eta(t), t\ge T_a)$ given $\gamma$ can be proved in the same way as in Theorem~\ref{thm::cvg_triple}, thanks to the observation in Lemma~\ref{lem::triple_in_single} and Figure~\ref{fig::single_to_triple}. 
\end{proof}

\begin{proof}[Proof of Corollary~\ref{cor::sle8duality}]
The conclusion is immediate from Lemma~\ref{lem::triple_conditional}. 
\end{proof}

\begin{proof}[Proof of Proposition~\ref{prop::fromSLEtohSLE}]
Assume the same notation as in Lemma~\ref{lem::triple_conditional}. 
We parameterize $\gamma$ so that it starts from $a$ and terminates when it hits $(cd)$ at time $T$. 
Let $\tau$ be any stopping time of $\gamma$ before $T$. Consider the conditional law of the following triple given $\gamma[0,\tau]$:  
\[(\eta(T_a-t), 0\le t\le T_a; \quad \gamma(t), \tau\le t\le T; \quad \eta(t), t\ge T_a).\] 
From Lemma~\ref{lem::triple_conditional}, the law of $(\gamma(t), \tau\le t\le T)$ is $\SLE_2(-1, -1; -1, -1)$ in $\HH\setminus\gamma[0,\tau]$ from $\gamma(\tau)$ to $(-\infty, x)$ with force points $(d, a^-; a^+, c)$; the conditional law of $(\eta(T_a-t), 0\le t\le T_a)$ given $\gamma[\tau, T]$ is $\SLE_8$ in $\Omega^L$ from $a^-$ to $d$, the conditional law of $(\eta(t),t\ge T_a)$ given $\gamma[\tau, T]$ is $\SLE_8$ in $\Omega^R$ from $a^+$ to $c$, and $(\eta(T_a-t), 0\le t\le T_a)$ and $(\eta(t),t\ge T_a)$ are conditionally independent given $\gamma[\tau, T]$. Comparing with Theorems~\ref{thm::cvg_triple} and~\ref{thm::cvg_lerw}, we see that the triple has the same law as the triple $(\eta^L; \gamma^M; \eta^R)$ in Theorem~\ref{thm::cvg_triple} in the quad 
$(\Omega\setminus\gamma[0,\tau]; d, a^-, a^+, c)$
conditional on $X^M=\gamma(\tau)$. 
In particular, the law of $(\eta(T_a-t), 0\le t\le T_a)$ is $\hSLE_8$ in $\Omega\setminus\gamma[0,\tau]$ from $a^-$ to $d$ conditional that its last hitting point of $\gamma[0,\tau]$ is $\gamma(\tau)$. Combining with reversibility of $\hSLE_8$ proved in Corollary~\ref{cor::ust_hsle8}, we obtain the conclusion.  
\end{proof}


\appendix
\section{Hypergeometric function and elliptic integral}
\label{appendix_hyper_elliptic}
For $A, B, C\in \R$, the hypergeometric function is defined for $|z|<1$ by the power series:
\begin{equation}\label{eqn::hypergeometric_def}
F(z)=\hF(A, B, C; z)=\sum_{n=0}^{\infty}\frac{(A)_n(B)_n}{(C)_n}\frac{z^n}{n!},
\end{equation}
where $(x)_n:=x(x+1)\cdots(x+n-1)$ for $n\ge 1$ and $(x)_n=1$ for $n=0$. The power series is well-defined when $C\not\in\{0,-1,-2,-3, \ldots\}$. The hypergeometric function is a solution of Euler's hypergeometric differential equation: 
\begin{align}\label{eqn::euler_eq}
z(1-z)F''(z)+\left(C-(A+B+1)z\right)F'(z)-ABF(z)=0.
\end{align}

We collect some properties for hypergeometric functions here.  
For $z\in (-1,1)$, we have (see \cite[Eq.~15.2.1 and Eq.~15.3.3]{AbramowitzHandbook})
\begin{align}
\hF(A, B, C; z)&=(1-z)^{C-A-B}\hF(C-A, C-B, C; z),\label{eqn::hyperf_eulertransform}\\
\frac{d}{dz}\hF(A,B,C;z)&=\frac{AB}{C}\hF(A+1,B+1,C+1;z).
\label{eqn::hyperf_derivative}
\end{align}

The series~\eqref{eqn::hypergeometric_def} is absolutely convergent on $z\in [0,1]$ when $C>A+B$ and $C\not\in\{0, -1, -2, ...\}$. In this case, we have (see \cite[Eq.~15.1.20]{AbramowitzHandbook})
\begin{align}\label{eqn::hyperf_one}
\hF(A, B, C; 1)=\frac{\Gamma(C)\Gamma(C-A-B)}{\Gamma(C-A)\Gamma(C-B)},
\end{align}
where $\Gamma$ is Gamma Function. 

We need the following lemma on hypergeometric functions in Section~\ref{subsec::hsle_continuity}. Fix $\kappa\ge 8$ and $\nu>-2$. Recall from~\eqref{eqn::hyper_def}, we denote 
\begin{align*}
F(z):=\hF\left(\frac{2\nu+4}{\kappa}, 1-\frac{4}{\kappa},\frac{2\nu+8}{\kappa};z\right).
\end{align*}

\begin{lemma}\label{lem::hyperfunction_asy}
Fix $\kappa\geq 8$ and $\nu>-2$. Denote by $\Gamma$ the Gamma function.
The function $F$ defined in \eqref{eqn::hyper_def} is increasing on $[0,1)$ with $F(0)=1$.
Moreover, we have the following asymptotic. 
\begin{itemize}
\item When $\kappa>8$ and $\nu>-2$, we have 
\begin{equation}\label{eqn::hyperfunction_asy}
\lim_{z\rightarrow 1-}(1-z)^{1-8/\kappa}F(z)=\hF\left(\frac{4}{\kappa},\frac{2\nu+12}{\kappa}-1, \frac{2\nu+8}{\kappa};1\right)=\frac{\Gamma(\frac{2\nu+8}{\kappa})\Gamma(1-\frac{8}{\kappa})}{\Gamma(\frac{2\nu+4}{\kappa})\Gamma(1-\frac{4}{\kappa})}\in (0,\infty).
\end{equation}
\item When $\kappa=8$ and $\nu> -2$, we have 
\begin{equation}\label{eqn::hyperfunction_asy_8}
\lim_{z\rightarrow 1-}\frac{F(z)}{\log\frac{1}{1-z}}=\frac{1}{\sqrt{\pi}}\frac{(\nu+2)\Gamma(2+\frac{\nu}{4})}{(\nu+4)\Gamma(\frac{3}{2}+\frac{\nu}{4})}\in (0,\infty).
\end{equation}
\end{itemize}
\end{lemma}
\begin{proof}
In this lemma, we set
\begin{align*}
A=\frac{2\nu+4}{\kappa},\quad B=1-\frac{4}{\kappa},\quad C=\frac{2\nu+8}{\kappa}.
\end{align*}
When $\kappa\geq 8$ and $\nu>-2$,  we have $A>0, B>0, C>0$, thus $F(z)$ is increasing on $[0,1)$ by the definition in~\eqref{eqn::hypergeometric_def}. It remains to derive the asymptotic. 

When $\kappa>8$ and $\nu>-2$, by~\eqref{eqn::hyperf_eulertransform}, we have
\[
F(z)
=(1-z)^{\frac{8}{\kappa}-1}\hF\left(\frac{4}{\kappa},\frac{2\nu+12}{\kappa}-1, \frac{2\nu+8}{\kappa};z\right).
\]
By~\eqref{eqn::hyperf_one}, we have
\[
\hF\left(\frac{4}{\kappa},\frac{2\nu+12}{\kappa}-1, \frac{2\nu+8}{\kappa};1\right)=\frac{\Gamma(\frac{2\nu+8}{\kappa})\Gamma(1-\frac{8}{\kappa})}{\Gamma(\frac{2\nu+4}{\kappa})\Gamma(1-\frac{4}{\kappa})}\in (0,\infty).
\]This gives~\eqref{eqn::hyperfunction_asy}. 

When $\kappa=8$ and $\nu>-2$, we have
\begin{align*}
\lim_{z\rightarrow 1-}\frac{F(z)}{\log\frac{1}{1-z}}
&=\lim_{z\rightarrow 1-}(1-z)F'(z)\tag{by L'Hospital rule}\\
&=\lim_{z\rightarrow 1-}\frac{AB}{C}(1-z)\hF\left(A+1, B+1, C+1; z\right)\tag{by~\eqref{eqn::hyperf_derivative}}\\
&=\lim_{z\rightarrow 1-}\frac{AB}{C} \hF\left(C-A, C-B, C+1; z\right)\tag{by~\eqref{eqn::hyperf_eulertransform}}\\
&=\frac{AB}{C}\frac{\Gamma(C+1)\Gamma(A+B+1-C)}{\Gamma(1+A)\Gamma(1+B)},\tag{by~\eqref{eqn::hyperf_one}}
\end{align*}
as desired in~\eqref{eqn::hyperfunction_asy_8}.
\end{proof}

\medbreak
Next, we introduce elliptic integral. 
Denote by $\LK$ the elliptic integral of the first kind (see~\cite[Eq.~17.2.6]{AbramowitzHandbook}): for $\varphi\in \overline{\HH}$ and $x\in (0,1)$, 
\begin{equation}\label{eqn::elliptic_def}
\LK(\varphi, x):=\int_0^{\varphi} \frac{d\theta}{\sqrt{1-x\sin^2\theta}},\quad \LK(x):=\int_0^{\pi/2}\frac{d\theta}{\sqrt{1-x\sin^2\theta}}. 
\end{equation}
There is a relation between complete elliptic integral and hypergeometric function (see~\cite[Eq.~17.3.9]{AbramowitzHandbook}): 
\begin{equation}\label{eqn::elliptic_hyper}
\LK(x)=\frac{\pi}{2}\hF\left(\frac{1}{2}, \frac{1}{2}, 1; x\right), \quad \forall x\in (0,1). 
\end{equation}

Let us calculate the derivatives of $\LK$: 
\begin{equation}\label{eqn::elliptic_derivatives}
\partial_{\varphi}\LK(\varphi, x)=\frac{1}{\sqrt{1-x\sin^2\varphi}}, \quad \partial_{\varphi}^2\LK(\varphi, x)=\frac{x\sin\varphi\cos\varphi}{\sqrt{1-x\sin^2\varphi}^3}, \quad \partial_{\varphi}\partial_x\LK(\varphi, x)=\frac{\sin^2\varphi}{2\sqrt{1-x\sin^2\varphi}^3}. 
\end{equation}
The derivatives $\partial_x\LK$ and $\partial_x^2\LK$ involve the elliptic integral of the second kind, but luckily, we do not need them. 
\medbreak

\section{Convergence of discrete harmonic functions}
Discrete harmonic functions have been studied in many cases in proof of convergence of discrete observables to continuous observables, see for instance~\cite{LawlerSchrammWernerLERWUST},~\cite{SmirnovPercolationConformalInvariance},~\cite{SchrammSheffieldDiscreteGFF} and~\cite{CDCHKSConvergenceIsingSLE}. Some standard methods to study discrete harmonic functions and discrete holomorphic functions have been established, see for instance~\cite{ChelkakRobustComplexAnalysis} and~\cite{ChelkakSmirnovDiscreteComplexAnalysis}. In this appendix, we collect some basic facts about harmonic functions with mixed boundary conditions.
\begin{lemma}\label{lem::harmonic_unique}
Fix a quad $(\Omega; a, b, c, d)$ and fix a conformal map $\xi$ from $\Omega$ onto $\U$ and extend its definition continuously to the boundary. There is a unique bounded harmonic function $u$ on $\Omega$ such that $u\circ\xi^{-1}$ satisfies the following boundary data: 
\begin{align}\label{eqn::harmonic_boundarydata}
\begin{cases}
u\circ\xi^{-1}=1, \quad &\text{on }(\xi(a)\xi(b)); \\
 u\circ\xi^{-1}=0, \quad &\text{on }(\xi(c)\xi(d)); \\
\partial_{n}u\circ\xi^{-1}=0,\quad &\text{on }(\xi(b)\xi(c))\cup (\xi(d)\xi(a)); 
\end{cases}
\end{align}
where $n$ is the outer normal vector.
 Moreover, the harmonic function $u$ can only obtain its maximum on $(ab)$ and obtain its minimum on $(cd)$.
\end{lemma}

\begin{proof}
The existence of such harmonic function is clear. We only need to show the uniqueness. 
Suppose there are two bounded harmonic functions $u_1$ and $u_2$ with the boundary data~\eqref{eqn::harmonic_boundarydata}. Define $\tilde{u}=(u_1-u_2)\circ\xi^{-1}$. Then $\tilde{u}$ is a bounded harmonic function with the following boundary data: $\tilde{u}=0$ on $(\xi(a)\xi(b))\cup(\xi(c)\xi(d))$ and $\partial_n \tilde{u}=0$ on $(\xi(b)\xi(c))\cup (\xi(d)\xi(a))$. It suffices to show $\tilde{u}=0$. 

First, we extend $\tilde{u}$ to $\C\setminus\left((\xi(a)\xi(b))\cup(\xi(c)\xi(d))\right)$ harmonically as follows. Choose $\tilde{v}$ to be a harmonic conjugate of $\tilde{u}$. Since $\partial_{n}\tilde{u}=0$ on $(\xi(b)\xi(c))\cup (\xi(d)\xi(a))$, we see that $\tilde{v}$ is constant along $(\xi(b)\xi(c))$ and is constant along $(\xi(d)\xi(a))$. We may set $\tilde{v}=0$ on $(\xi(b)\xi(c))$. Define $g=\ii \tilde{u}-\tilde{v}$ and this is a holomorphic function on $\U$. We define $g$ on $\C\setminus \overline{\U}$ by setting $g(z)=\overline{g(1/\overline{z})}$. 
Since $\partial_{n}\tilde{u}=0$ on $(\xi(b)\xi(c))\cup (\xi(d)\xi(a))$, by Schwarz reflection principle, the function $g$ can be extended to a holomorphic function on $\C\setminus\left((\xi(a)\xi(b))\cup(\xi(c)\xi(d))\right)$ which we still denote by $g$. This implies that $\tilde{u}$ can be extended to $\C\setminus\left((\xi(a)\xi(b))\cup(\xi(c)\xi(d))\right)$ harmonically and we still denote its extension by $\tilde{u}$.

Second, we show that $\tilde{u}$ is continuous at $\xi(a),\xi(b),\xi(c)$ and $\xi(d)$. It suffices to show $\lim_{z\to\xi(a)}\tilde{u}(z)=0$ and the limit at the other three points can be derived similarly. Suppose $\left|\tilde{u}\circ\xi^{-1}\right|\le M$ for some $M>0$. Fix two small constants $r>\epsilon>0$. Define $\tilde{u}_{M}$ to be the harmonic function on $B(\xi(a),r)\setminus(\xi(a)\xi(b))$ with the following boundary data: $\tilde{u}_M=0$ on $B(\xi(a),r)\cap(\xi(a)\xi(b))$ and $\tilde{u}_M=M$ on $\partial B(\xi(a),r)$. 
From maximum principle, we have $|\tilde{u}(z)|\le |\tilde{u}_M(z)|$ for all $z\in B(\xi(a),r)\setminus(\xi(a)\xi(b))$. Combining with Beurling estimate, for every $z\in B(\xi(a),\epsilon)$, we have
\[\left|\tilde{u}(z)\right|\le\left|\tilde{u}_M(z)\right|\le CM\sqrt{\epsilon/r},\]
for some universal constant $C>0$. This gives $\lim_{z\to\xi(a)}\tilde{u}(z)=0$ as desired. 

From the first step, $\tilde{u}$ is a bounded harmonic function on $\C\setminus\left((\xi(a)\xi(b))\cup(\xi(c)\xi(d))\right)$, by maximum principle, it obtains its maximum and minimum on $[\xi(a)\xi(b)]\cup [\xi(c)\xi(d)]$. 
In particular, $\tilde{u}$ obtains its maximum and minimum in $\U$ on $[\xi(a)\xi(b)]\cup [\xi(c)\xi(d)]$. 
Recall that $\tilde{u}=0$ on $(\xi(a)\xi(b))\cup (\xi(c)\xi(d))$ and it is continuous at $\xi(a), \xi(b), \xi(c), \xi(d)$ as proved in the second step. Thus the maximum and the minimum of $\tilde{u}$ are both zero on $[\xi(a)\xi(b)]\cup [\xi(c)\xi(d)]$. This gives $\tilde{u}=0$ as desired. 

The last statement can be proved by similar reflection method and the maximum and minimum principle. This completes the proof.
\end{proof}

\begin{lemma}\label{lem::harmonic_cvg}
Fix a quad $(\Omega; a, b, c, d)$ and fix a conformal map $\xi$ from $\Omega$ onto $\U$ and extend $\xi$ continuously to the boundary. 
Suppose a sequence of discrete domains $(\Omega_\delta; a_\delta, b_\delta, c_\delta, d_\delta)$ converges to $(\Omega; a, b, c, d)$  in the Carath\'{e}odory sense as $\delta\rightarrow 0$. 
For $z\in V(\Omega_{\delta})$, let $u_{\delta}(z)$ be the probability that a simple random walk in $\Omega_{\delta}$ starting from $z$ hits $(a_{\delta}b_{\delta})$ before $(c_{\delta}d_{\delta})$.
Let $u(z)$ be the bounded harmonic function in Lemma~\ref{lem::harmonic_unique}. 
Then $u_{\delta}$ converges to $u$ locally uniformly as $\delta\to 0$.  
\end{lemma}
\begin{proof}
This is a consequence of~\cite[Proposition~4.2]{LawlerSchrammWernerLERWUST}. We summarize the proof here for concreteness. 
For any function $g_{\delta}$ on $V(\Omega_{\delta})$ or $V(\Omega^{*}_{\delta})$, define its discrete derivatives  as
\[\partial_{x}^{\delta}g_{\delta}:=\delta^{-1}(g_{\delta}(v+\delta)-g_{\delta}(v)),\quad \partial_{y}^{\delta}g_{\delta}:=\delta^{-1}(g_{\delta}(v+i\delta)-g_{\delta}(v)).\]
We extend $u_{\delta}$ and all its derivatives to functions on $\Omega_{\delta}$ by linear interpolation.

First, we show that any subsequential limit of $u_{\delta}$ is harmonic on $\Omega$. 
Since $u_{\delta}$ is discrete harmonic and $0\le u_{\delta}\le 1$, by~\cite[Lemma 5.2]{LawlerSchrammWernerLERWUST}, for any compact set $S\subset\Omega$ and $k\in\N$, there exists a constant $C>0$ which depends on $S$ and $k$ such that
\[|\partial_{a_{1}}^{\delta}\partial_{a_{2}}^{\delta}\cdots\partial_{a_{k}}^{\delta}u_{\delta}(v)|\le C, \quad \text{for any }\partial_{a_{1}}^{\delta},\ldots,\partial_{a_{k}}^{\delta}\in\{\partial_{x}^{\delta},\partial_{y}^{\delta}\}, \text{ and }v\in S\cap\Omega_{\delta}. \]
By Arzela-Ascoli theorem, for any sequence $\delta_{n}\to 0$, there exists a subsequence, still denoted by $\{\delta_{n}\}$, and continuous functions $u,u_{x},u_{y},u_{xx},u_{yy}$ such that
\[u_{\delta}\to u,\quad \partial_{x}^{\delta_{n}}u_{\delta_{n}}\to u_{x},\quad \partial_{y}^{\delta_{n}}u_{\delta_{n}}\to u_{y},\quad \partial_{x}^{\delta_{n}}\partial_{x}^{\delta_{n}}u_{\delta_{n}}\to u_{x},\quad \partial_{y}^{\delta_{n}}\partial_{y}^{\delta_{n}}u_{\delta_{n}}\to u_{y},\quad \text{locally uniformly.}\]
This implies 
\[u_{x}=\partial_{x}u,\quad u_{y}=\partial_{y}u,\quad u_{xx}=\partial_{x}^2u,\quad u_{yy}=\partial_{y}^2u.\]
Since $u_{\delta_{n}}$ is discrete harmonic, the function $u$ is harmonic. This shows that any subsequential limit of $u_{\delta}$ is harmonic. 

Next, we show that any subsequential limit $u$ has the boundary data given in the statement. Note that such boundary data and harmonicity and boundedness uniquely determines the subsequential limit, hence gives the convergence of sequence $u_{\delta}$. 
From the definition of $u_{\delta}$ and the discrete Beurling estimate, it is clear that $u=1$ on $(ab)$ and $u=0$ on $(cd)$. To derive the boundary data along $(bc)\cup(ad)$, we need to introduce  the discrete harmonic conjugate function $v^{*}_{\delta}$ of $u_{\delta}$. 

Define $v^{*}_{\delta}=0$ on $(b_{\delta}^{*}c_{\delta}^{*})$. For every oriented edge $e^{*}_{\delta}=\{x^{*}_{\delta},y^{*}_{\delta}\}\in E(\Omega_{\delta}^{*})$, there is a unique oriented edge $e=\{x_{\delta},y_{\delta}\}\in E(\Omega_{\delta})$ which crosses $e^{*}_{\delta}$ from its right-side. Define $v^{*}_{\delta}(y_{\delta}^{*})-v^{*}_{\delta}(x_{\delta}^{*}):=u_{\delta}(y_{\delta})-u_{\delta}(x_{\delta})$. Since $u_{\delta}$ is discrete harmonic, this is well-defined and $v^{*}_{\delta}$ is constant on $(d_{\delta}^{*}a_{\delta}^{*})$. Moreover, $v_{\delta}^{*}$ takes its maximum, denoted by $L_{\delta}$, on $(d_{\delta}^{*}a_{\delta}^{*})$. 

We claim that $\{L_{\delta}\}_{\delta>0}$ is uniformly bounded. Assume this is true, by the same argument as above, for any sequence $\delta_n\to 0$, there exists a subsequence, still denoted by $\delta_n$, and a constant $\tilde{K}$ and a harmonic function $v$ such that $v_{\delta_{n}}^*\to v$ and the other related derivatives also converge locally uniformly and that $L_{n}\to \tilde K$ as $n\to\infty$. From the construction and the discrete Beurling estimate, we have $v=\tilde K$ on $(da)$ and $v=0$ on $(bc)$. By the definition of $v_{\delta}^{*}$, the function $f_{\delta}:=u_{\delta}+\ii v_{\delta}^{*}$ is discrete holomorphic. Then, the convergence of discrete derivatives implies that $f:=u+\ii v$ is holomorphic on $\Omega$. By Schwartz reflection principle, we can extend $f\circ\xi^{-1}$ to $\partial\U\setminus\{\xi(a), \xi(b), \xi(c), \xi(d)\}$ analytically. By Cauchy-Riemann equation, we have 
\[\partial_{n}(u\circ \xi^{-1})=\partial_{t}(v\circ \xi^{-1})=0\]
on $(\xi(b)\xi(c))$ and $(\xi(d)\xi(a))$ where $t$ is the tangential vector. This gives the required boundary data of $u$ along $(bc)\cup(da)$. 

Finally, it remains to show that $\{L_{\delta}\}_{\delta>0}$ is uniformly bounded. If this is not the case, there exists a sequence $\delta_{n}\to 0$ such that $L_{\delta_n}\to\infty$ as $n\to\infty$. By the same argument as above, the sequence $\frac{1}{L_{n}}f_{(\Omega_{\delta_{n}}; a_{\delta_{n}}, b_{\delta_{n}}, c_{\delta_{n}}, d_{\delta_{n}})}$ converges to a holomorphic function $h$ locally uniformly. In such case, we have $\Re h=0$ on $\Omega$, thus $h$ is constant. But $\Im h=1$ on $(da)$ and $\Im h=0$ on $(bc)$, this gives a contradiction. Thus, $\{L_{\delta}\}_{\delta>0}$ is uniformly bounded and we complete the proof.
\end{proof}

\section{Tightness}\label{app::tightness}
The goal of this section is to show the tightness of LERW branches (Proposition~\ref{prop::tightness_LERW}) and the Peano curves (Proposition~\ref{prop::tightness}). 
We first recall the tightness of LERW branches in the case of Dobrushin boundary conditions from~\cite{{LawlerSchrammWernerLERWUST}} in Lemma~\ref{lem::tightness_Dobru}. 
We will give a similar conclusion for LERW branches in quad with alternating boundary conditions in Section~\ref{app::C1} and then complete the proof of Proposition~\ref{prop::tightness_LERW} and Proposition~\ref{prop::tightness} in Section~\ref{app::C2}.

The setup for the Dobrushin boundary condition is as follows. Fix two points $a, b$. Suppose $(ab)$ is a curve and $(ba)$ is a simple and $C^1$ curve and $\Omega$ is the domain surrounded by $[ab]\cup[ba]$. 
We emphasize that we do not require extra regularity of $(ab)$ and we allow that $(ab)$ intersects $(ba)$. 
Fix a sequence of discrete Dobrushin domains (defined at the beginning of Section~\ref{subsec::ust_dobrushin}) $\{(\Omega_\delta;a_\delta,b_\delta)\}_{\delta>0}$ such that there exists a constant $C>0$,
\[d((b_\delta a_\delta),(ba))\le C\delta,\quad d((a_\delta b_\delta),(ab))\to 0,\quad\text{as }\delta\to 0,\]
where $d$ is the metric~\eqref{eqn::curves_metric}. 
Fix $o\in \Omega$ and denote by $o_\delta$ the discrete approximation of $o$ on $\Omega_\delta$ and denote by $\tilde\gamma^o_\delta$ the loop-erased random walk from $o_\delta$ to $(a_\delta b_\delta)$ on $\Omega_\delta$.
\begin{lemma}\label{lem::tightness_Dobru}
The sequence $\{\tilde\gamma_{\delta}^{o}\}_{\delta>0}$ is tight. Moreover, any subsequential limit is a simple curve in $\overline \Omega$ which intersects $[ab]$ only at its end. 
\end{lemma}
\begin{proof}
The proof is similar to the proof of~\cite[Lemma 3.12]{LawlerSchrammWernerLERWUST} and we will  prove a similar result  in Lemma~\ref{lem::tightness_LERW1}.
Since the proof of Lemma~\ref{lem::tightness_Dobru} is almost the same as the proof of Lemma~\ref{lem::tightness_LERW1} which is the prerequisite of all of our main results, we choose to provide a detailed proof there.
\end{proof}

\subsection{Tightness of loop-erased random walk}
\label{app::C1}

In this section, the goal is to show the tightness of LERW branches in quad. The setup is as follows. Given a quad $(\Omega;a,b,c,d)$ such that $(bc)$ and $(da)$ are simple and $C^1$. 
Suppose there exists a simply connected domain $\tilde\Omega$ such that $\partial\tilde\Omega$ is $C^1$ and simple with $\Omega\subset\tilde\Omega$ and $\partial\Omega\cap\partial\tilde\Omega=[bc]\cup[da]$. Fix a sequence of discrete quads $\{(\Omega_{\delta}; a_{\delta}, b_{\delta}, c_{\delta}, d_{\delta})\}_{\delta>0}$ satisfying that 
there exists a constant $C>0$ such that
\[d((b_\delta c_\delta),(bc))\le C\delta,\quad d((d_\delta a_\delta),(da))\le C\delta,\quad d((a_\delta b_\delta),(ab))\to 0,\quad d((c_\delta d_\delta),(cd))\to 0,\quad\text{as }\delta\to 0,\]
where $d$ is the metric~\eqref{eqn::curves_metric}.
We consider the $\ust$ on $\Omega_\delta$ with alternating boundary conditions: $(a_{\delta}b_{\delta})$ is wired and $(c_{\delta}d_{\delta})$ is wired. 
Note that we require the simple and $C^1$ regularity only on $(bc)$ and $(da)$ and the stonger convergence of $(b_\delta c_\delta)$ to $(bc)$ and $(d_\delta a_\delta)$ to $(da)$. This is mainly because we generate the $\ust$ by Wilson's algorithm and the random walk ends at $(a_\delta b_\delta)\cup(c_\delta d_\delta)$ and continues when it hits $(b_\delta c_\delta)$ and $(d_\delta a_\delta)$. To get the Beurling estimate near the boundary $(b_\delta c_\delta)$ and $(d_\delta a_\delta)$, more smoothness conditions are needed. For instance, the Beurling estimate fails to hold if there is a non-simple part in the boundary $(b_\delta c_\delta)$ and $(d_\delta a_\delta)$ at the macroscopic level.

Recall that $\gamma_\delta^M$ is the branch of $\ust$ connecting $(a_\delta b_\delta)$ and $(c_\delta d_\delta)$.
Lemma~\ref{lem::tightness_LERW1} below gives the tightness of $\{\gamma_\delta^M\}_{\delta>0}$, which is a weaker version of Proposition~\ref{prop::tightness_LERW}. Lemma~\ref{lem::simple} shows the tightness of branches between any two vertices.
\begin{lemma}\label{lem::tightness_LERW1}
The sequence $\{\gamma_{\delta}^{M}\}_{\delta>0}$ is tight. Moreover, any subsequential limit is a simple curve in $\overline\Omega$ which intersects $[ab]\cup[cd]$ only at its two ends. 
\end{lemma}
To prove Lemma~\ref{lem::tightness_LERW1}, we first establish the Beurling estimate in $\Omega_\delta$ with reflection boundary  arcs in Lemma~\ref{lem::Beur}. 
Such estimates are crucial in the proof of the tightness. 
We then prove a technical result in Lemma~\ref{lem::tech}. With these tools, we are able to complete the proof of Lemma~\ref{lem::tightness_LERW1}. 
\begin{lemma}\label{lem::Beur}
There exist constants $C=C(\Omega)>0$ and $\delta_0=\delta_0(\Omega)>0$, such that the following holds: for any $\delta\le\delta_1\le\delta_0$ and  connected subgraph $A$ of $\Omega_\delta$ with $\diam(A)>\delta_1$, if  $z_\delta\in\Omega_\delta$ satisfies  $\dist(z_\delta, A\cup(a_\delta b_\delta)\cup(c_\delta d_\delta))\le\delta_1$, then
\[
\PP\left[\LR\text{ hits }\partial B(z_\delta,C\delta_1)\text{ before }A\cup(a_\delta b_\delta)\cup(c_\delta d_\delta)\right]\le 1/2,
\]
where $\LR$ is the simple random walk starting from $z_\delta$, reflecting at $(b_\delta c_\delta)\cup(d_\delta a_\delta)$, and stopping when it hits $(a_\delta b_\delta)\cup(c_\delta d_\delta)$. Similarly, for any discrete curve $\gamma_\delta$  on $\Omega_\delta$, if  $z_\delta\in\Omega_\delta$ satisfies $\dist(z_\delta, A\cup\gamma_{\delta}\cup(a_\delta b_\delta)\cup(c_\delta d_\delta))\le\delta_1$, then
\[
\PP\left[\LR\text{ hits }\partial B(z_\delta,C\delta_1)\text{ before }A\cup\gamma_{\delta}\cup(a_\delta b_\delta)\cup(c_\delta d_\delta)\right]\le 1/2,
\]
where $\LR$ is the same random walk as above.
\end{lemma}
\begin{proof}
{We only prove the first statement since the second is similar.} By the assumptions, there exists a simply connected domain $\tilde\Omega$ such that $\partial\tilde\Omega$ is $C^1$ with $\Omega\subset\tilde\Omega$ and $\partial\Omega\cap\partial\tilde\Omega=[bc]\cup[da]$.
Denote by $\tilde\Omega_\delta$ the discrete approximation of $\tilde{\Omega}$ on $\delta\Z^2$ 
such that $\partial\tilde\Omega_\delta\cap\partial\Omega_\delta=[b_\delta c_\delta]\cup[d_\delta a_\delta]$ 
and that there exists $C>0$ such that,
\[d(\partial\tilde\Omega_\delta,\partial\tilde\Omega)\le C\delta,\]
where $d$ is the metric~\eqref{eqn::curves_metric}.
Denote by $\tilde\LR$ the simple random walk on $\tilde\Omega_\delta$ starting from $z_\delta$ and reflecting at $\partial\tilde\Omega_{\delta}$. Define 
\[P_1:=\PP\left[\tilde\LR\text{ hits }\partial B(z_\delta,C\delta_1)\text{ before }A\right]\]
and
\[P_2:=\PP\left[\tilde\LR\text{ hits }\partial B(z_\delta,C\delta_1)\text{ before }(a_\delta b_\delta)\right],\quad P_3:=\PP\left[\tilde\LR\text{ hits }\partial B(z_\delta,C\delta_1)\text{ before }(c_\delta d_\delta)\right].\]
Then, we have
\[\PP\left[\LR\text{ hits }\partial B(z_\delta,C\delta_1)\text{ before }A\cup(a_\delta b_\delta)\cup(c_\delta d_\delta)\right]\le \min\{P_1,P_2,P_3\}.\]
The conclusion follows from \cite[Lemma~11.2]{SchrammScalinglimitsLERWUST}. The proof there relies crucially on the simple and $C^1$ regularity assumption and the stronger convergence. In the  proof there, the author used estimate of the discrete Dirichlet energy to get estimates on discrete harmonic functions. To this end, the author counted the number of certain  discrete paths. The $C^1$ and simple regularity condition and the stronger convergence make it possible to have a good control of the number of  the discrete paths when $A$ is near $(b_\delta c_\delta)$ and $(d_\delta a_\delta)$.
\end{proof}


\begin{lemma}\label{lem::tech}
For  $z_0\in\Omega$, $\alpha>0$ and $0<\beta\le 1/4\dist((ab),(cd))$, let $\LA_{\delta}(z_0,\beta,\alpha;\gamma_\delta^M)$ denote the event that there are two points $o_1,o_2\in\gamma_\delta^M$ with
$o_1$, $o_2\in B(z_0,\beta/4)$ and $|o_1-o_2|\le\alpha$, such that the subarc of $\gamma_\delta^M$ between $o_1$ and $o_2$ is not contained in $B(z_0,\beta)$. Then for any $\epsilon>0$, there exists $\alpha=\alpha(\beta,\epsilon)$ such that
\[\PP[\LA_{\delta}(z_{0},\beta,\alpha;\gamma_\delta^M)]\le\eps,\quad \forall \delta>0.\]
\end{lemma}
\begin{proof}
The proof is similar to the proof of~\cite[Lemma 3.4]{SchrammScalinglimitsLERWUST}. By the choice of $\beta$, we may assume that $\dist(B(z_0,\beta),(a_\delta b_\delta))>1/4\dist((ab),(cd))$. The case for $\dist(B(z_0,\beta),(c_\delta d_\delta))>1/4\dist((ab),(cd))$ can be dealt similarly by the reversibility of loop-erased random walk and simple random walk. Let $\LR_\delta$ be the simple random walk from $(a_\delta b_\delta)$ to $(c_\delta d_\delta)$ which generates the loop-erased random walk $\gamma_\delta^M$. Let $t_0=0$. For $j\ge 1$, define inductively \[
s_j:=\inf\left\{t\ge t_{j-1}:\LR_\delta(t)\in B\left(z_0,\beta/4\right)\right\}\quad \text{and}\quad t_j:=\inf\{ t\ge s_{j}:\LR_\delta(t)\notin B(z_0,\beta)\}.
\] 
Let $\tau$ be the hitting time of $\LR_\delta$ at $(c_\delta d_\delta)$. For every positive integer $t\in\N$, define $\text{LE}(\LR_\delta[0,t])$ to be the loop erasure of $\LR_\delta[0,t]$. Define $\LA_{\delta}(z_0,\beta,\alpha;\text{LE}(\LR_\delta[0,t_j]))$ similarly as $\LA_{\delta}(z_0,\beta,\alpha;\gamma_\delta^M)$. For $j\geq 0$, define the event $T_j:=\{t_j\le\tau\}$.  For simplicity, we denote by $Y_j$ the event  $\LA_{\delta}(z_0,\beta,\alpha;\text{LE}(\LR_\delta[0,t_j]))$. Note that $Y_1=\emptyset$ and 
\[\LA_{\delta}(z_0,\beta,\alpha;\gamma_\delta^M)\subset\{\cup_{j=2}^{\infty}(Y_j\cap T_j)\}\subset \left\{T_{m+1}\cup\bigcup_{j=2}^{m}Y_j\right\},\quad\text{ for every }m>1.\]
Thus, we only need to control $\PP[\cup_{j=1}^{m}Y_j]$ and $\PP[T_{m+1}]$.

First, for every $j\ge 1$, denote by $Q_j$ the set consisting of connected components of $\LR_\delta[0,s_{j+1}]\cap B(z_0,\beta)$ that intersect $\partial B(z_0,\beta/4)$ and $\partial B(z_0,\beta)$ and do not contain $\LR_\delta(s_{j+1})$. Note that the cardinality $|Q_j|$ of $Q_j$ is at most $j$. Given $Y^{c}_{j}$, on the event $Y_{j+1}$, the random walk $\LR_\delta$ first comes into the $\alpha$-neighbour of $Q_j$ in $B\left(z_0,\alpha\right)$, then it exits $B(z_0,\beta)$ without hitting any component in $Q_j$.

Combining with Lemma~\ref{lem::Beur}, there exist constants $C>0$ and $k>0$, such that
\[\PP[Y_{j+1}\cond Y^c_j]\le C|Q_j|\left(\frac{\alpha}{\beta}\right)^{k}\le Cj\left(\frac{\alpha}{\beta}\right)^{k}.\]
This implies that 
\begin{equation}\label{eqn::first}
\PP[\cup_{j=2}^{m}Y_j]\le\sum_{j=1}^{m-1}\PP[Y_{j+1}\cond Y^c_j]\le Cm^2\left(\frac{\alpha}{\beta}\right)^{k}.
\end{equation}

Second, for every $z_\delta\in\Omega_\delta$, define
\[h_\delta(z_\delta):=\PP\left[\text{the simple random walk starting from }z_\delta\text{ hits }(c_\delta d_\delta)\text{ before }B\left(z_0,\beta/4\right)\cup (a_\delta b_\delta)\right].\]
Recall that the simple random walk will continue when it hits $(b_\delta c_\delta)\cup(d_\delta a_\delta)$ and thus $h_\delta$ is harmonic on $\Omega_\delta\setminus\{(a_\delta b_\delta)\cup(c_\delta d_\delta)\cup B(z_0,\beta/4)\}$.
Let $\partial B_\delta(z_0,\beta)$ be the collection of vertices on $\Omega_\delta$ closest to $\partial B(z_0,\beta)$. We claim that there exists $0<q<1$, such that for every $\delta>0$,
\begin{equation}\label{eqn::claim}
\min_{z_\delta\in\partial B_\delta(z_0,\beta)}h_\delta(z_\delta)\ge q
\end{equation}
Assuming this is true. Combining with the Markov property of random walk, we have
\begin{equation}\label{eqn::second}
\PP[T_{m+1}]\le (1-q)^{m}.
\end{equation}
Combining~\eqref{eqn::first} and~\eqref{eqn::second}, for every $\eps>0$, by first choosing $m$ large enough and then choosing appropriate $\alpha=\alpha(\beta,\epsilon)$, we have
\[\PP[\LA_{\delta}(z_0,\beta,\alpha;\gamma_\delta^M)]\le\eps.\]
This completes the proof.

We are left to prove~\eqref{eqn::claim}. Assuming~\eqref{eqn::claim} does  not hold. Then, there exist $\delta_n\to 0$ and $z_{\delta_n}\in \partial B_{\delta_n}(z_0,\beta)$, such that 
\[\lim_{n\to\infty}h_{\delta_n}(z_{\delta_n})=0.\]
By the minimal principle for discrete harmonic functions, we can find a path $r_{\delta_n}$ connecting $z_{\delta_n}$ to $B\left(z_0,\beta/4\right)$ such that 
\[\max_{v\in r_{\delta_n}}h_{\delta_n}(v)\le h_{\delta_n}(z_{\delta_n}).\]
By extracting a subsequence, we may assume $r_{\delta_n}\subset B(z_0,\beta)$ and $r_{\delta_n}$ converges to a compact set $K$ in Hausdorff metric and $h_{\delta_n}$ converges to a harmonic function $h$ locally uniformly. By Beurling estimate, we have $h$ equals $1$ on $(cd)$ and equals $0$ on $(ab)\cup \partial B\left(z_0,\beta/4\right)$.
Since $\diam(r_{\delta_n})\ge \beta/4$, we have $\diam(K)\ge \beta/4$. 
Choose a sequence of points $\{w_m\}\subset\Omega$ such that $w_m\to w\in K\cap B^{c}\left(z_0,\beta/4\right)$. Denote by $w_m^{\delta_n}$ the discrete approximation of $w_m$ on $\Omega_{\delta_n}$. By Lemma~\ref{lem::Beur}, there exist constants $C>0$ and $k>0$, such that 
\[h_{\delta_n}(w_m^{\delta_n})\le h_{\delta_n}(z_{\delta_n})+C\left(\frac{\dist(w_m^{\delta_n},r_{\delta_n})}{\beta}\right)^k.\]
By letting $n\to\infty$, we have
\[h(w_m)\le C\left(\frac{\dist(w_m,K)}{\beta}\right)^k.\]
By letting $m\to\infty$, we have 
\begin{equation}\label{eqn::contra1}
h(w)=0.
\end{equation}
There are two cases:
\begin{itemize}
\item
If $K\cap B^c\left(z_0,\beta/4\right)\cap\Omega\neq\emptyset$, we can choose $w\in K\cap B^c\left(z_0,\beta/4\right)\cap\Omega$. Then~\eqref{eqn::contra1} contradicts the minimum principle since we have $h$ equals $1$ on $(cd)$.
\item
If $K\cap  B^c\left(z_0,\beta/4\right)\subset\partial\Omega$, in this case, $\Omega_\delta\setminus\{(a_\delta b_\delta\cup  B(z_0,\beta/4)\}$ is simply connected. Thus, we can construct the harmonic conjugate of  $h_\delta$ as in Lemma~\ref{lem::harmonic_cvg}. Then $ h$ only equals $0$ on $(ab)\cup \partial B(z_0,\beta/4)$ by the same argument as in Lemma~\ref{lem::harmonic_unique}.
Since $B(z_0,\beta)\cap(ab)=\emptyset$, we have that $w\in (ba)$. This is a contradiction to~\eqref{eqn::contra1}. 
\end{itemize}
We complete the proof of~\eqref{eqn::claim}.
\end{proof}
\begin{proof}[Proof of Lemma~\ref{lem::tightness_LERW1}]
The proof is similar to the proof of~\cite[Lemma 3.12]{LawlerSchrammWernerLERWUST}, which is essentially contained in~\cite[Theorem 1.1]{SchrammScalinglimitsLERWUST}. 
Suppose $\Upsilon:(0,\infty)\to (0,1]$ is an increasing function. For every simply connected domain $D$, denote by $\chi_{\Upsilon}(D)$ the space of simple curves $r :[0,1]\to\overline D$, such that for every $0\le s_1<s_2\le 1$, 
\[\dist(r[0,s_1],r[s_2,1])\ge\Upsilon(\diam(r[s_1,s_2])).\]
We claim that for $\epsilon>0$, there exists $\Upsilon$ such that
\begin{equation}\label{eqn::tight}
\PP[\gamma_{\delta}^{M}\in\chi_{\Upsilon}(\Omega)]\ge 1-\epsilon,\quad\text{for all }\delta. 
\end{equation}
Roughly speaking, this estimate says that $\gamma_{\delta}^M$ does not create ``almost bubble" with high probability. 
We choose $R>0$ such that $\Omega_{\delta}\subset B(0,R)$ for all $\delta$. Then, we have 
\[\PP[\gamma_{\delta}^{M}\in\chi_{\Upsilon}(B(0,R))]\ge1-\epsilon.\] 
By the same argument as in~\cite[Lemma 3.10]{LawlerSchrammWernerLERWUST}, the set $\chi_{\Upsilon}(B(0,R))$ is a compact set of simple curves. This completes the proof of tightness.

Second, we prove that any subsequential limit $\gamma$ intersects $[ab]\cup[cd]$ only at its two ends. We only prove that $\gamma$ intersects $[cd]$ at its one end. The proof for $[ab]$ is similar by the reversibility of loop-erased random walk. Denote by $\LR_\delta$ the random walk which generates the loop-erased random walk $\gamma_\delta^M$. For $m\ge 1$, define 
\[\tau^m_\delta:=\inf\{t:\dist(\gamma_\delta^M(t),(c_\delta d_\delta))=2^{-2m}\}\]
and let $\tau_\delta$ be the hitting time of $\gamma^M_\delta$ at $(c_\delta d_\delta)$. Similarly, define \[\tilde\tau^m_\delta:=\inf\{t:\dist(\LR_\delta(t),(c_\delta d_\delta))=2^{-2m}\}\]
and let $\tilde\tau_\delta$ be the hitting time of $\LR_\delta$ at $(c_\delta d_\delta)$. By Lemma~\ref{lem::Beur}, there exist constants $C>0$ and $c>0$, such that for every $m>0$, we have
\[\PP[\diam(\LR_\delta[\tilde\tau_\delta^m,\tilde\tau_\delta])\ge 2^{-m}]\le C 2^{-cm}.\]
Since $\gamma_\delta^M[\tau^m_\delta,\tau_\delta]\subset\LR[\tilde\tau^m_\delta,\tilde\tau_\delta]$, we have
\[\PP[\diam(\gamma^M_\delta[\tau_\delta^m,\tau_\delta])\ge 2^{-m}]\le C 2^{-cm}.\]
Now, suppose $\gamma_{\delta_n}^M$ converges to $\gamma$ in law and we may couple them together such that $\gamma_{\delta_n}^M$ converges to $\gamma$ almost surely. We define $\tau^m$ and $\tau$ for $\gamma$ similarly as before. Then, we have 
\[\{\diam(\gamma[\tau^{m},\tau])\ge 2^{-m}\}\subset\cup_{i=1}^{\infty}\cap_{j=i}^{\infty}\{\diam(\gamma_{\delta_j}^M[\tau_{\delta_j}^{m},\tau_{\delta_j}])\ge 2^{-m}\}.\]
This implies that
\[\PP[\diam(\gamma[\tau^{m},\tau])\ge 2^{-m}]\le C2^{-cm}.\]
Thus, we have
\[\sum_{m=1}^{\infty}\PP[\diam(\gamma[\tau^{m},\tau])\ge 2^{-m}]<+\infty.\]
By Borel-Cantelli lemma, we have that $\gamma$ hits $[cd]$ only at its end almost surely.

It remains to prove~\eqref{eqn::tight}. It suffices to show it for any sequence $\delta_n\to 0$. For $\alpha, \beta>0$, denote by $\LA_{\delta_n}(\beta,\alpha)$ the event that there exist $0\le s_{1}<s_{2}\le 1$ such that $\dist(\gamma_{\delta_n}^{M}[0,s_1],\gamma_{\delta_n}^{M}[s_2,1])\le\alpha$ but $\diam(\gamma_{\delta_n}^{M}[s_1,s_2])\ge\beta$. Note that we can choose $\{z_1,\ldots,z_k\}\subset\Omega$, where $k=k(\beta)$, such that every ball with center in $\Omega$ and radius $2\alpha$ is contained in one of the $k$ balls $B(z_j,\frac{\beta}{2})$ for $j=1,\ldots,k$. Recall the definition of $\LA_{\delta}(z,\beta,\alpha;\gamma_\delta^M)$ in Lemma~\ref{lem::tech}. Note that $\LA_{\delta_n}(\beta,\alpha)\subset\cup_{j=1}^{k}\LA_{\delta_n}(z_j,\beta,\alpha;\gamma_\delta^M)$. 
Thus, by Lemma~\ref{lem::tech}, for every $m>0$, we can choose $\alpha_{m}$ such that
\[\PP[\LA_{\delta_n}(2^{-m-1}\dist((ab),(cd)),\alpha_{m})]\le \frac{\epsilon}{2^{m}},\quad\text{for all }n. \]
We choose $\Upsilon$ such that $\Upsilon(t)<\alpha_{m}$ for every $t\le 2^{-m-1}\dist((ab),(cd))$. Then, we have
\[\PP[\gamma_{\delta_n}^{M}\in\chi_{\Upsilon}(\Omega_{\delta_n})]\ge \PP[\left(\cup_{m=1}^{\infty}\LA_{\delta_n}(2^{-m-1}\dist((ab),(cd)),\alpha_m)\right)^{c}]\ge 1-\epsilon.\]
This gives~\eqref{eqn::tight} and completes the proof.
\end{proof}

Fix $o_1,o_2\in\overline\Omega$ and suppose $o_1^\delta,o_2^\delta$ are the approximations of $o_1$ and $o_2$ (we allow $o_1$ and $o_2$ to be $(ab)$ and $(cd)$ and allow $o_1^\delta$ and $o_2^\delta$ to be $(a_\delta b_\delta)$ and $(c_\delta d_\delta)$). Denote by $\gamma_{\delta}^{o_1,o_2}$ the branch of the $\ust$ between $o_1^\delta$ and $o_2^\delta$. Recall that the wired boundary arcs $(a_\delta b_\delta)$ and $(c_\delta d_\delta)$ belong to the $\ust$. Thus, if $\gamma_{\delta}^{o_1,o_2}$ hits boundary arcs $(a_{\delta}b_{\delta})$ or $(c_{\delta}d_{\delta})$, it always contains subarcs of boundary arcs so that it is a continuous simple curve. 
\begin{lemma}\label{lem::simple} 
Fix $o_1,o_2\in\overline\Omega$ and suppose $o_1^\delta,o_2^\delta$ are the approximations of $o_1$ and $o_2$. Then $\left\{\gamma_{\delta}^{o_1,o_2}\right\}_{\delta>0}$  is tight. Moreover, any subsequential limit is a simple curve in $\overline\Omega$.
\end{lemma}
\begin{proof}
It is equivalent to show that the law of $\left\{\gamma_{\delta}^{o_1,o_2}\right\}_{\delta>0}$ is relatively compact.  We fix a subsequence $\left\{\gamma_{\delta_n}^{o_1,o_2}\right\}_{n\ge 0}$. From Lemma~\ref{lem::tightness_LERW1}, the sequence $\{\gamma_{\delta_n}^M\}_{n\ge 0}$ is tight. By extracting a subsequence, we may assume that $\left\{\gamma_{\delta_n}^{M}\right\}_{n\ge 0}$ converges to a continuous simple curve $\gamma^{M}$ in law and therefore we can couple them together such that $\gamma_{\delta_n}^M$ converges to $\gamma^{M}$ almost surely. 
We only need to consider two cases.
\begin{itemize}
\item
At least one of $o_1$ and $o_2$ equals $(ab)$ or $(cd)$. 
\item
Both $o_1$ and $o_2$ belong to the interior of $\Omega$.
\end{itemize}

For the first case, we may assume $o_2$ equals $(ab)$. We first run the loop-erased random walk $\gamma_{\delta_n}^M$ and then run the loop-erased random walk from $o_1^{\delta_n}$ to $(a_{\delta_n} b_{\delta_n})\cup\gamma_{\delta_n}^M\cup(c_{\delta_n} d_{\delta_n})$ which we denote by $\tilde\gamma^{o_1}_{\delta_n}$.
From Lemma~\ref{lem::tightness_Dobru}, the sequence $\left\{\tilde\gamma^{o_1}_{\delta_n}\right\}_{n\ge 0}$ is tight and any subsequential limit $\tilde{\gamma}^{o_1}$ of ${\left\{\tilde\gamma^{o_1}_{\delta_n}\right\}_{n\ge 0}}$ hits $[ab]\cup\gamma^M\cup[cd]$ at $\tilde{\gamma}^{o_1}$'s end. We may couple them together such that $\tilde\gamma^{o_1}_{\delta_n}$ converges to $\tilde\gamma^{o_1}$ almost surely. Note that $\gamma_{\delta_n}^{o_1,(ab)}$ may consist of subsegments of $\gamma^{M}_{\delta_n}$ and $\tilde\gamma^{o_1}_{\delta_n}$ and $(c_{\delta_n}d_{\delta_n})$. From the fact that $\gamma^{M}_{\delta_n}$ converges to $\gamma^M$ and $\tilde\gamma^{o_1}_{\delta_n}$ converges to $\tilde\gamma^{o_1}$ almost surely and the fact that $\tilde\gamma^{o_1}$ intersects $[ab]\cup\gamma^M\cup[cd]$ at $\tilde{\gamma}^{o_1}$'s end almost surely, we have that $\gamma_{\delta_n}^{o_1,(ab)}$ converges to a continuous simple curve almost surely. This completes the proof of the first case.

For the second case, We first run $\gamma_{\delta_n}^M$ and $\tilde\gamma_{\delta_n}^{o_1}$ in the same way as above and then run the loop-erased random walk from $o_2^{\delta_n}$ to $(a_{\delta_n} b_{\delta_n})\cup\gamma_{\delta_n}^M\cup\tilde\gamma_{\delta_n}^{o_1}\cup(c_{\delta_n} d_{\delta_n})$ which we  denote by $\tilde\gamma^{o_2}_{\delta_n}$. The remaining proof is similar to the first case (by replacing $\gamma_{\delta_n}^M$ by $\gamma_{\delta_n}^M\cup\tilde\gamma_{\delta_n}^{o_1}$ and replacing $\tilde\gamma_{\delta_n}^{o_1}$ by $\tilde\gamma_{\delta_n}^{o_2}$). This completes the proof.
\end{proof}

\subsection{Proof of Propositions~\ref{prop::tightness} and~\ref{prop::tightness_LERW}}\label{app::C2}
In this section, we assume the following setup. Suppose $(\Omega;a,b,c,d)$ is a quad such that $\partial\Omega$ is $C^1$ and simple and suppose that a sequence of discrete quads $\{(\Omega_{\delta}; a_{\delta}, b_{\delta}, c_{\delta}, d_{\delta})\}_{\delta>0}$ converges to $(\Omega; a, b, c, d)$ as in the sense of~\eqref{eqn::topology}. Note that here we require that the boundary $\partial\Omega$ is $C^1$ and simple because we need to apply Lemma~\ref{lem::Beur} both on primal lattice and dual lattice. We consider the $\ust$ on $\Omega_\delta$ with alternating boundary conditions. We will complete the proof of Proposition~\ref{prop::tightness} and Proposition~\ref{prop::tightness_LERW}. 
To this end, we first show Lemmas~\ref{lem::tightness_trunk} and~\ref{lem::subseqlimit_avoids_c}.  Recall that $\mathbf{trunk}_{\delta}\left(\eps\right)$, $\mathbf{trunk}_{0}\left(\frac{1}{n}\right)$ and $\mathbf{trunk}$ for the UST and $\mathbf{trunk}_{\delta}^*\left(\eps\right)$, $\mathbf{trunk}^{*}_{0}\left(\frac{1}{n}\right)$ and $\mathbf{trunk}^{*}$ for dual forest are defined in Lemma~\ref{lem::tightness_trunk}.
\begin{proof}[Proof of Lemma~\ref{lem::tightness_trunk}]
The proof of this lemma is almost the same as the proof of~\cite[Theorem~11.1]{SchrammScalinglimitsLERWUST}. We summarize the strategy of the proof of~\cite[Theorem~11.1]{SchrammScalinglimitsLERWUST} briefly. There are two main steps.
\begin{enumerate}
\item \label{step1}
Showing that with large probability, the $\mathbf{trunk}_{\delta}\left(\eps\right)$ is contained in the minimal subtree which contains a collection of vertices $V_\delta(\eps)$ in $\Omega_\delta$ that forms a kind of dense set in $\Omega$ and showing a similar result about $\mathbf{trunk}^{*}_{\delta}\left(\eps\right)$. Here we require that $V_\delta$ satisfies the following condition: for a small constant $\delta_0$ which depends on $\eps$, for every $z_\delta\in\Omega_\delta$, there is a vertex  $v_\delta\in V_\delta(\eps)$ such that $|v_\delta-z_\delta|<\delta_0$.
Note that~\eqref{eqn::trunktrunkstar_C1regularity} is equivalent to the following equation:
\[\lim_{\tilde\eps\to 0}\lim_{\eps\to 0}\varlimsup_{\delta\to 0}\PP[\text{the Hausdorff distance between } \mathbf{trunk}_{\delta}(\eps)\cap\mathbf{trunk}^{*}_{\delta}(\eps)\text{ is less than }\tilde\eps]=0.\]
Thus, the results in this step reduce the estimate of the distance between $\mathbf{trunk}_\delta(\eps)$ and $\mathbf{trunk}^*_{\delta}(\eps)$ to an estimate of the distance between a subtree of the $\ust$ and a subforest of the dual forest.
\item \label{step2}
Reducing the estimate about a subtree of the $\ust$ and a subforest of the corresponding dual forest in Step~\ref{step1} to estimates about random walks on $\Omega_\delta$ and $\Omega_\delta^{*}$ and completing these estimates.
\end{enumerate}
Now, we show how to carry out these steps in our setting.
\paragraph*{Step~\ref{step1}.}The proof, as being pointed out in the proof of~\cite[Theorem~11.1(i)]{SchrammScalinglimitsLERWUST}, is the same as the proof of~\cite[Theorem~10.2]{SchrammScalinglimitsLERWUST}, which only uses the discrete Beurling estimate. On the primal graph, we first generate the branch $\gamma_\delta^M$ between  $(a_\delta b_\delta)$ and $(c_\delta d_\delta)$. Then, the Beurling estimate in $\Omega_\delta\setminus{\gamma_\delta^M}$ is from the second item of  Lemma~\ref{lem::Beur}. Recall that we denote by $\Omega_\delta^L$ and $\Omega_\delta^R$ the two connected components of $\Omega_\delta\setminus\gamma_\delta^M$. The boundary conditions of them are all Dobrushin boundary conditions. Denote by $\mathbf{trunk}^L_{\delta}\left(\eps\right)$ the $\epsilon-$trunk of the $\ust$ in $\Omega_\delta^L$ and denote by $\mathbf{trunk}^R_{\delta}\left(\eps\right)$ similarly. Thus, the proof of~\cite[Theorem~10.2]{SchrammScalinglimitsLERWUST} works for $\mathbf{trunk}^L_{\delta}\left(\eps\right)$ and $\mathbf{trunk}^R_{\delta}\left(\eps\right)$ as well and implies the following: For every $\epsilon>0$, there exists $\hat\delta>0$ with the following property. For $i\in\{L,R\}$,  let  $V_\delta^i$  be a vertex set such that for every $w\in\Omega_\delta^i$, there exists a point $z\in V_\delta^i$  satisfying $\dist(w,z)\le\hat\delta$, then,
\[
\PP\left[\mathbf{trunk}^i_{\delta}\left(\eps\right)\subset T^{i}_{\delta}\right]\ge 1-\epsilon,
\]
where $T^{i}_{\delta}$ the minimal subtree of the $\ust$ which contains $V_\delta^i$. 
Denote by $V_\delta=V_\delta^L\cup V_\delta^R\cup\{a_\delta,b_\delta,c_\delta,d_\delta\}$. Recall that we always require the $\epsilon-$trunk contains the wired boundary arc. By Wilson's algorithm, we can couple the $\ust$s in $\Omega_\delta^L$ and $\Omega_\delta^R$ and $\Omega_\delta$ together, such that
\[\mathbf{trunk}_{\delta}(\eps)\subset\mathbf{trunk}^L_{\delta}(\eps)\cup\mathbf{trunk}^R_{\delta}(\eps).\]
Denoted by $T_\delta^V$ the the minimal subtree of the $\ust$ on $\Omega_\delta$ which contains $V_\delta$. Then, we have
\[
\PP\left[\mathbf{trunk}_{\delta}\left(\eps\right)\subset T^{V}_{\delta}\right]\ge 1-2\epsilon.
\]
On the dual graph $(\Omega_\delta^*;a_\delta^*,b_\delta^*,c_\delta^*,d_\delta^*)$, the dual forest can be obtained as follows. We denote by $\tilde\Omega_\delta^*$ the graph obtained from $\Omega_\delta^*$ by wiring $(b^*_\delta c_\delta^*)$ and $(d_\delta^* a_\delta^*)$ as one vertex and keeping the incident relation with other vertices. We generate the $\ust$ on $\Omega_\delta^*$ and then view its edges as edges on $\Omega_\delta^*$. Note that in this case, there are two connected components for the $\ust$ if we view it on $\Omega_\delta^*$. This is exactly the dual forest. The discrete Beurling estimate for the dual forest  is obtained from Lemma~\ref{lem::Beur} similarly by considering the random walk on $\tilde\Omega_\delta^*$ and hence the conclusion about $\mathbf{trunk}^{*}_{\delta}\left(\eps\right)$ is similar.
\paragraph*{Step~\ref{step2}.}
As explained in the first paragraph of the proof of~\cite[Theorem~11.1]{SchrammScalinglimitsLERWUST}, the remaining proof is the same as the proof of~\cite[Theorem~10.7]{SchrammScalinglimitsLERWUST}. As explained in the third paragraph of the proof of ~\cite[Theorem~10.7]{SchrammScalinglimitsLERWUST}, we only need to prove the following estimate on the primal graph $\Omega_\delta$ and the same estimate on $\Omega_\delta^*$: For every $a_1,a_2\in\Omega$, denote by $a_1^{\delta},a_2^{\delta}$ the corresponding approximations on $\Omega_\delta$ and by $\beta_\delta$ the branch in the $\ust$ between $a_1^{\delta}$ and $a_2^\delta$, for every $v_\delta\in\Omega_\delta\setminus((a_\delta b_\delta)\cup(c_\delta d_\delta))$, denote by $J(v_{\delta}, \beta_{\delta})$ the branch in the $\ust$ from $v_\delta$ to $\beta_\delta$. Here, if $\beta_\delta$ intersects $(a_\delta b_\delta)$ or $(c_\delta d_\delta)$, we require that $(a_\delta b_\delta)\subset\beta_\delta$ or $(c_\delta d_\delta)\subset\beta_\delta$. This is natural by our definition of $\ust$ with $(a_\delta b_\delta)$ wired and $(c_\delta d_\delta)$ wired. Then, for every $s>0$ and $t>0$, we have
\begin{align*}
\lim_{h\to 0}\varlimsup_{\delta\to 0}\PP\big[&\exists v_\delta\in B(a_1^\delta,s)^{c} \cap B(a_2^\delta,s)^{c} \cap\Omega_\delta\setminus((a_\delta b_\delta)\cup(c_\delta d_\delta))\text{ such that }\\
&\dist(v_\delta,\beta_\delta)<h\text{ and }\diam(J(v_{\delta}, \beta_{\delta}))>t \big]=0.
\end{align*}
The proof of this equation is similar to the proof of~\cite[Theorem~10.7]{SchrammScalinglimitsLERWUST}. In~\cite{SchrammScalinglimitsLERWUST}, the main inputs of the proof of~\cite[Theorem~10.7]{SchrammScalinglimitsLERWUST} are~\cite[Lemma~10.8, Corollary~10.6]{SchrammScalinglimitsLERWUST} whose proof requires:
\begin{itemize}
\item[(a)] discrete Beurling estimate;
\item[(b)] estimates of discrete harmonic function, see~\cite[Lemma~10.9]{SchrammScalinglimitsLERWUST};
\item[(c)] the fact that any subsequential limit of the branches of the UST in Hausdorff metric are simple curves. 
\end{itemize}
In our setup, the discrete Beurling estimate (a) is true due to Lemma~\ref{lem::Beur}; the estimates of discrete harmonic function (b) comes from~\cite[Lemma~10.9]{SchrammScalinglimitsLERWUST}  directly (in particular, we do not need the boundary regularity in (b));  and the fact (c) that any subsequential limit of the branches of the UST is  are simple curves also holds due to Lemma~\ref{lem::simple} and its analogue for the dual forest. Therefore, the proof works in our setup as well. 
This completes the proof.
\end{proof}



Next, we will prove Lemma~\ref{lem::subseqlimit_avoids_c}. Recall that $\eta_\delta^L$ is the Peano curve from $a_\delta^{\diamond}$ to $d_\delta^{\diamond}$ and that $\gamma_\delta^M$ is the branch of the $\ust$ connecting $(a_\delta b_\delta)$ to $(c_\delta d_\delta)$. To estimate the probability in~\eqref{eqn::subseqlimit_avoids_c}, we will use Wilson's algorithm which relates $\gamma_{\delta}^M$ to loop-erased random walk, as described in the proof of Lemma~\ref{lem::YconditionalXdiscrete}.
%
\begin{proof}[Proof of Lemma~\ref{lem::subseqlimit_avoids_c}]
The proof is almost the same as the proof of~\cite[Lemma~10.9]{SchrammScalinglimitsLERWUST}.
We denote by $\tilde{\Omega}_{\delta}$ the graph obtained from $\Omega_\delta$ by viewing $(a_\delta b_\delta)$ as one vertex and keeping the incident relation from $(a_\delta b_\delta)$ to other vertices. Note that here we do not view $(c_\delta d_\delta)$ as one vertex. We denote by $\tilde{\LR}$ the simple random walk on $\tilde{\Omega}_{\delta}$ which ends at $(c_\delta d_\delta)$. Denoye by $\PP^{z_\delta}$ the law of the simple random walk starting from $z_\delta$. Define
\[h_{\delta,\eps}(z_\delta):=\PP^{z_\delta}[\tilde\LR\text{ hits }B(c_\delta,\eps)\text{ before }(c_\delta d_\delta)].\]
Note that the right boundary of $\eta_\delta^L$ equals to $\gamma_\delta^M$. By Wilson's algorithm, $\gamma_\delta^M$ can be generated by the loop-erasure of $\tilde\LR$. Thus, it suffices to show $ h_{\delta,\eps}((a_\delta b_\delta))\to 0$ when letting $\delta\to 0$ and then letting $\eps\to 0$.

Fix $r>\eps$ and we denote by $\overline\LR$ the simple random walk on $\tilde{\Omega}_{\delta}$ which ends at $(c_\delta d_\delta)\cap B^c(c_\delta,r)$. We define 
\[\overline h_{\delta,\eps}(z_\delta):=\PP[\overline\LR\text{ hits }B(c_\delta,\eps)\text{ before }(c_\delta d_\delta)].\]
Then, we have
\[h_{\delta,\eps}\le \overline h_{\delta,\eps}\quad\text{on }\tilde\Omega_\delta.\]
Thus, it suffices to show that $\overline h_{\delta,\eps}((a_\delta b_\delta))\to 0$ when letting $\delta\to 0$ and then letting $\eps\to 0$.

The remainning argument is same as the argument in~\cite[Lemma~10.9]{SchrammScalinglimitsLERWUST}. Since $\overline h_{\delta,\eps}$ is a discrete harmonic function on $\tilde\Omega\setminus\{B(c_\delta,\eps)\cup((c_\delta d_\delta)\cap B^c(c_\delta,r))\}$ and equals $0$ on $(c_\delta d_\delta)\cap B^c(c_\delta,r)$ and equals $1$ on $B(c_\delta,\eps)$. By maximum principle, there is a path from $(a_\delta b_\delta)$ to $B(c_\delta,\eps)$, which we denote by $l_\delta$, such that for every $z_\delta\in l_\delta$, we have
\begin{equation*}\overline h_{\delta,\eps}(z_\delta)\ge \overline h_{\delta,\eps}((a_\delta b_\delta)).
\end{equation*}
Next, we will bound the discrete Dirichlet energy of $\overline h_{\delta,\eps}$. Note that there exists a constant $C>0$, which only depends on $\Omega$, such that there are at least $\frac{C}{\delta}$ paths connecting $(c_\delta d_\delta)\cap B^c(c_\delta,r)$ to $l_\delta$ and each path has length less than $\frac{C}{\delta}$. This implies that by Cauchy-Schwarz inequality, the discrete Dirichlet energy of $\overline h_{\delta,\eps}$ satisfies the following bound
\begin{equation}\label{eqn::aux111}
\sum_{\substack{x_\delta,y_\delta\in \tilde\Omega_\delta\\ x_\delta\sim y_\delta}}|\overline h_{\delta,\eps}(x_\delta)-\overline h_{\delta,\eps}(y_\delta)|^2\ge C \overline h_{\delta,\eps}((a_\delta b_\delta)).
\end{equation}
Note that $\overline h_{\delta,\eps}$ is the discrete function with minimal discrete Dirichlet energy over the set
\[\{r_\delta:\tilde\Omega\to\R:r_\delta\text{ equals }0\text{ on }(c_\delta d_\delta)\cap B^c(c_\delta,r)\text{ and equals }1\text{ on }B(c_\delta,\eps)\}.\]
We define a special discrete function $r_\delta$ as follows: $r_\delta$ equals $1$ on $B(c_\delta,\eps)$ and equals $0$ on $B^c(c_\delta,r)$. For $z_\delta\in B(c_\delta,r)\setminus B(c_\delta,\eps)$, we define $r_\delta(z_\delta):=\frac{\log(r/|z_\delta-c_\delta|)}{\log(r/\eps)}$. As in~\cite[Lemma~10.9]{SchrammScalinglimitsLERWUST}, the discrete Dirichlet energy of $r_\delta$ satisfies the following bound: there exists a constant $C>0$,
\[\sum_{\substack{x_\delta,y_\delta\in \tilde\Omega_\delta\\ x_\delta\sim y_\delta}}|r_\delta(x_\delta)-r_\delta(y_\delta)|^2\le C \frac{1}{\log|r/\eps|}.\]
Thus, combining with~\eqref{eqn::aux111}, we have
\[\overline h_{\delta,\eps}((a_\delta b_\delta))\le C\frac{1}{\log|r/\eps|}.\]
This completes the proof.
\end{proof}
\begin{figure}[ht!]
\begin{center}
\includegraphics[width=0.4\textwidth]{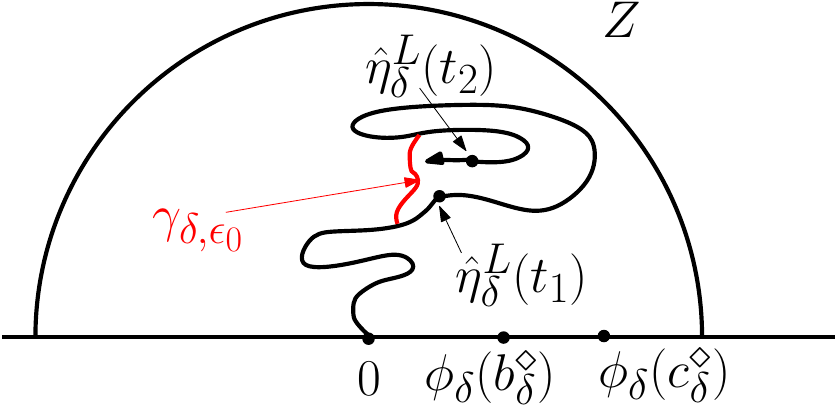}
\includegraphics[width=0.4\textwidth]{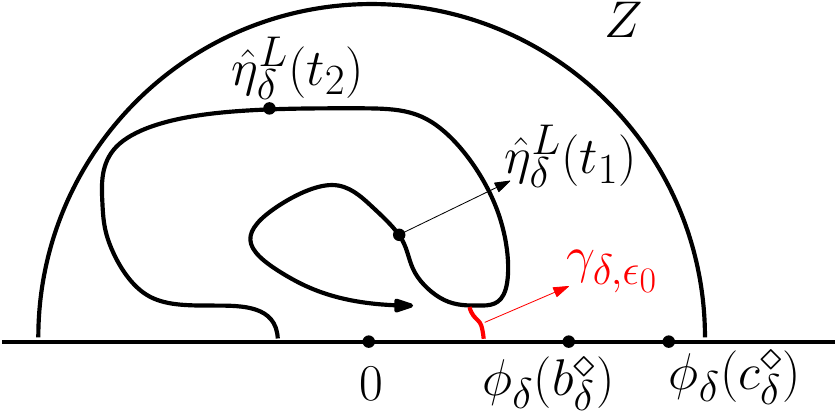}
\includegraphics[width=0.4\textwidth]{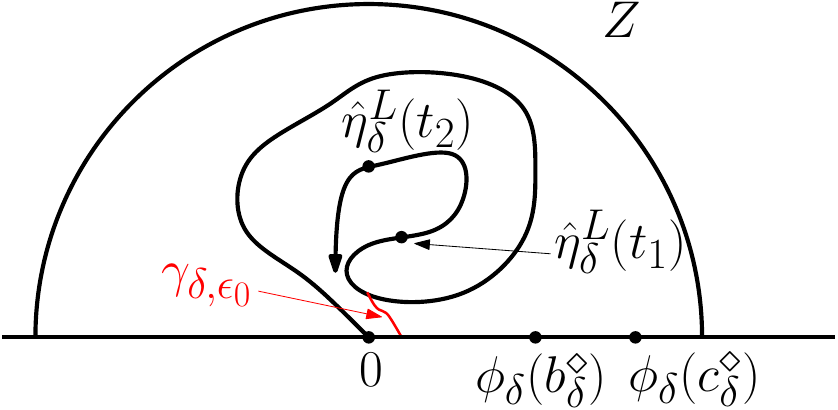}
\includegraphics[width=0.4\textwidth]{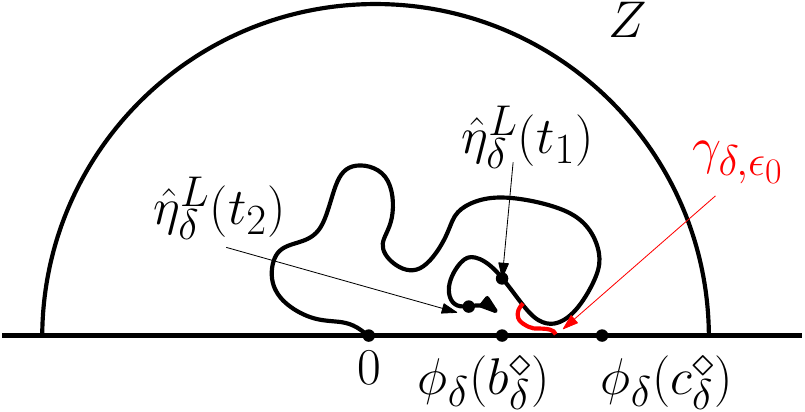}
\end{center}
\caption{\label{fig::fourevents} 
The left-up figure corresponds to the event $\{\LA^{\delta,\epsilon}_{\epsilon_{0}}\setminus \chi^{\delta}_{1,\epsilon_{0}}(s_{1}(\epsilon))\}$. On this event, the ``almost bubble" is in the interior of $\HH$.
The right-up figure corresponds to the event $\{\LA^{\delta,\epsilon}_{\epsilon_{0}}\cap\chi^{\delta}_{1,\epsilon_{0}}(s_{1}(\epsilon))\setminus(\chi^{\delta}_{0,\epsilon_{0}}(s_{0}(\epsilon))\cup
\chi^{\delta}_{b,c,\epsilon_{0}}(s_{0}(\epsilon)))\}$. On this event, the ``almost bubble" is near $\R\setminus\{0,b,c\}$.
The left-bottom figure corresponds to the event $\{\LA^{\delta,\epsilon}_{\epsilon_{0}}\cap\chi^{\delta}_{0,\epsilon_{0}}(s_{0}(\epsilon))\}$. On this event, the ``almost bubble" is near $0$. 
The right-bottom figure corresponds to the event $\{\LA^{\delta,\epsilon}_{\epsilon_{0}}\cap\chi^{\delta}_{b,c,\epsilon_{0}}(s_{0}(\epsilon))\}$. On this event, the ``almost bubble" is near $\{b,c\}$. }
\end{figure}

\begin{proof}[Proof of Proposition~\ref{prop::tightness}]
The proof is similar to~\cite[Proposition~4.5, Lemma~4.6]{LawlerSchrammWernerLERWUST}. We summarize the proof below and adjust it to our setting. 
By the convergence of domains, we can choose conformal maps $\phi_{\delta}:\Omega_\delta^\diamond\to\HH$ with $\phi_{\delta}(a_\delta^\diamond)=0$ and $\phi_{\delta}(d_\delta^\diamond)=\infty$ such that $\phi^{-1}_{\delta}$ converges uniformly on $\overline\HH$. Define $\hat\eta^L_{\delta}:=\phi_{\delta}(\eta^L_{\delta})$ and we parameterize $\hat\eta^L_{\delta}$ by the half-plane capacity. 
To prove the tightness, there are two parts. 
\begin{itemize}
\item For every $t>0, \eps>0$, there exists $\eps_0>0$ such that 
\begin{align}\label{eqn::equicontinuous}
\PP[\sup\{|\hat\eta_{\delta}(t_{2})-\hat\eta_{\delta}(t_{1})|:0\le t_{1}<t_{2}\le t,|t_{2}-t_{1}|\le\epsilon_{0}\}\ge\epsilon]<\epsilon.
\end{align}
\item The transience of curves at $d$: for any $\epsilon>0$, there exist $\epsilon'\le\epsilon$ and $\delta_{0}>0$ such that, for all $\delta\le\delta_{0}$,   
\begin{align}\label{eqn::neard}
\PP[\eta^L_{\delta}\text{ hits } \partial B(d,\epsilon)\text{ after hitting }\partial B(d,\epsilon')]<\epsilon.
\end{align}
\end{itemize}

We first derive~\eqref{eqn::equicontinuous}. For $0<t_{1}<t_{2}<+\infty$, let $(g^{\delta}_{t}, t\ge 0)$ be the corresponding conformal maps of $\hat\eta^L_{\delta}$ and let $Y_{\delta}(t_{1},t_{2}):=\diam(g^{\delta}_{t_{1}}(\hat\eta^L_{\delta}[t_{1},t_{2}]))$. Combining with~\cite[Lemma~2.1]{LawlerSchrammWernerLERWUST}, to prove~\eqref{eqn::equicontinuous}, it suffices to prove the following statement~\cite[Lemma~4.6]{LawlerSchrammWernerLERWUST}: For every $\epsilon>0$, there exist $\epsilon_{0}>0$ and $\delta_{0}>0$ such that, for all $\delta<\delta_{0}$, 
\begin{align}\label{eqn::tightness}
\PP[\sup\{|\hat\eta^L_{\delta}(t_{2})-\hat\eta^L_{\delta}(t_{1})|:0\le t_{1}<t_{2}\le\tau_{\delta},Y_{\delta}(t_{1},t_{2})\le\epsilon_{0}\}\ge\epsilon]<\epsilon
\end{align}
where $\tau_{\delta}:=\inf\{t\ge 0:|\eta^L_{\delta}(t)|=\epsilon^{-1}\}$.
For $\delta, \eps, \eps_0>0$, define  
\[\LA^{\delta,\epsilon}_{\epsilon_{0}}=\{\exists 0\le t_{1}<t_{2}\le\tau_{\delta}\text{ such that }|\hat\eta^L_{\delta}(t_{2})-\hat\eta^L_{\delta}(t_{1})|\ge \epsilon ,\text{ but }Y_{\delta}(t_{1},t_{2})\le\epsilon_{0}\}.\] 
The event $\LA^{\delta,\epsilon}_{\epsilon_{0}}$ implies the existence of ``almost bubble" before time $\tau_\delta$.
To prove~\eqref{eqn::tightness}, it suffices to show that, for a well-chosen positive function $\epsilon_{0}(\eps)$ such that $\epsilon_{0}(\eps)\to 0$ as $\eps\to 0$, we have 
\begin{equation}\label{eqn::tightenss_aux}
\lim_{\epsilon\to 0}\lim_{\delta\to 0}\sup_{\epsilon_{0}\le\epsilon_{0}(\eps)}\PP[\LA^{\delta,\epsilon}_{\epsilon_{0}}]=0.
\end{equation} 

We then show~\eqref{eqn::tightenss_aux}. On the event $\LA^{\delta,\epsilon}_{\epsilon_{0}}$, there exists a simple curve $\gamma_{\delta,\epsilon_0}$ with short length, such that the interior surrounded by $\gamma_{\delta,\epsilon_0}\cup \hat\eta^L_{\delta}[0,t_{1}]$ contains $\hat\eta^L_{\delta}[t_{1},t_{2}]$ as follows. Denote by $Z$ the semicircle $2\epsilon^{-1}\partial\U\cap\HH$. On the one hand, there is a constant $C>0$ such that $\dist(g^{\delta}_{t}(Z),g^{\delta}_{t}(\hat\eta^L_{\delta}[t,\tau_{\delta}]))\ge C$ for all $t\le\tau_{\delta}$. On the other hand,  on the event $\LA^{\delta,\epsilon}_{\epsilon_{0}}$, we have $Y_{\delta}(t_{1},t_{2})\le\epsilon_{0}\to 0$ as $\epsilon_{0}\to 0$. These two facts guarantee that the extremal length of simple arcs in $\HH\setminus g_{t_{1}}(\hat\eta^L_{\delta}[t_{1},t_{2}])$ which separate $g^{\delta}_{t_{1}}(Z)$ from $g^{\delta}_{t_{1}}(\hat\eta^L_{\delta}[t_{1},t_{2}])$ tends to $0$ as $\epsilon_{0}\to 0$. By the conformal invariance of extremal length, there exists a simple curve $\gamma_{\delta,\epsilon_{0}}$ in $\HH\setminus\hat\eta^L_{\delta}[0,t_{1}]$ separating $\hat\eta^L_{\delta}[t_{1},t_{2}]$ and $Z$  such that the length of $\gamma_{\delta,\epsilon_{0}}$ tends to $0$ as $\epsilon_{0}\to 0$.

For $s>0$, we denote
\begin{align*}
&\chi^{\delta}_{0,\epsilon_{0}}(s): =\{\dist(0,\gamma_{\delta,\epsilon_{0}})<s\}, \quad \chi^{\delta}_{1,\epsilon_{0}}(s):=\{\dist(\R,\gamma_{\delta,\epsilon_{0}})<s\}, \\
&\chi^{\delta}_{b,c,\epsilon_{0}}(s): =\{\dist(\{\phi_\delta(b_\delta^\diamond),\phi_\delta(c_\delta^\diamond)\},\gamma_{\delta,\epsilon_{0}})<s\}. 
\end{align*}
We choose two function  $s_{1}(\epsilon)$ and $s_{0}(\epsilon)$ such that $s_{1}(\epsilon)<s_{0}(\epsilon)$ and $\lim_{\epsilon\to 0}\frac{s_{0}(\epsilon)}{\epsilon}=0$. Now, we divide $\LA^{\delta,\epsilon}_{\epsilon_{0}}$ into four events 
\begin{align*}
&\{\LA^{\delta,\epsilon}_{\epsilon_{0}}\setminus \chi^{\delta}_{1,\epsilon_{0}}(s_{1}(\epsilon))\}, 
\quad \{\LA^{\delta,\epsilon}_{\epsilon_{0}}\cap\chi^{\delta}_{1,\epsilon_{0}}(s_{1}(\epsilon))\setminus(\chi^{\delta}_{0,\epsilon_{0}}(s_{0}(\epsilon))\cup
\chi^{\delta}_{b,c,\epsilon_{0}}(s_{0}(\epsilon)))\}, \\
&\{\LA^{\delta,\epsilon}_{\epsilon_{0}}\cap\chi^{\delta}_{0,\epsilon_{0}}(s_{0}(\epsilon))\}, \quad 
\{\LA^{\delta,\epsilon}_{\epsilon_{0}}\cap\chi^{\delta}_{b,c,\epsilon_{0}}(s_{0}(\epsilon))\}.
\end{align*}
These four events are classifed by the location of the ``almost bubble". See Figure~\ref{fig::fourevents}.
For the first two events, we may use the same argument as~\cite[Eq.~(4.9), Eq.~(4.10)]{LawlerSchrammWernerLERWUST}, and we have
\begin{align}\label{eqn::interior}
&\lim_{\epsilon\to 0}\lim_{\delta\to 0}\sup_{\epsilon_{0}\le\epsilon_{0}(\delta)}\PP[\LA^{\delta,\epsilon}_{\epsilon_{0}}\setminus \chi^{\delta}_{1,\epsilon_{0}}(s_{1}(\epsilon))]=0.\\
\label{eqn::boundary}
&\lim_{\epsilon\to 0}\lim_{\delta\to 0}\sup_{\epsilon_{0}\le\epsilon_{0}(\delta)}\PP\left[\LA^{\delta,\epsilon}_{\epsilon_{0}}\cap\chi^{\delta}_{1,\epsilon_{0}}(s_{1}(\epsilon))\setminus(\chi^{\delta}_{0,\epsilon_{0}}(s_{0}(\epsilon))\cup\chi^{\delta}_{b,c,\epsilon_{0}}(s_{0}(\epsilon)))\right]=0.
\end{align}
This part of the proof relies on~\eqref{eqn::trunktrunkstar_C1regularity}. The third event can be estimated by the same argument for~\cite[Eq.~(4.11)]{LawlerSchrammWernerLERWUST}, and we have
\begin{align}\label{eqn::origin}
\lim_{\epsilon\to 0}\lim_{\delta\to 0}\sup_{\epsilon_{0}\le\epsilon_{0}(\delta)}\PP[\LA^{\delta,\epsilon}_{\epsilon_{0}}\cap\chi^{\delta}_{0,\epsilon_{0}}(s_{0}(\epsilon))]=0.
\end{align} 
For the fourth event, from~\eqref{eqn::subseqlimit_avoids_c}, we have
\begin{align}\label{eqn::bc}
\lim_{\epsilon\to 0}\lim_{\delta\to 0}\sup_{\epsilon_{0}\le\epsilon_{0}(\delta)}\PP[\LA^{\delta,\epsilon}_{\epsilon_{0}}\cap\chi^{\delta}_{b,c,\epsilon_{0}}(s_{0}(\epsilon))]=\lim_{\epsilon\to 0}\varlimsup_{\delta\to 0}\PP[\eta_{\delta}^{L}\text{ hits } B(c_{\delta},\epsilon)]=0.
\end{align}
To sum up, for the four events, we have~\eqref{eqn::interior}, \eqref{eqn::boundary}, \eqref{eqn::origin} and~\eqref{eqn::bc}. They imply~\eqref{eqn::tightenss_aux} and complete the proof for~\eqref{eqn::equicontinuous}.

Next, we show~\eqref{eqn::neard}. We follow the same argument as in~\cite[The last paragraph of the proof of Lemma 4.6]{LawlerSchrammWernerLERWUST}.
Fix two constants $r<\epsilon$ and $C>0$. We choose $v\in V(\Omega_{\delta})$ which is adjacent to $\partial B(d,r)$. Suppose $v^{*}\in V(\Omega^{*}_{\delta})$ which is adjacent to $v$. We denote by $\chi_{v}$ the LERW which starts from $v$ and ends at $(c_{\delta}d_{\delta})$. Similarly, we denote by $\chi^{*}_{v^{*}}$ the LERW which starts from $v^{*}$ and ends at $(d^{*}_{\delta}a^{*}_{\delta})$. We define $A_{1}:=\{\diam(\chi_{v})<Cr\}$ and $A_{2}:=\{\dist(d,\chi_{v})>\frac{1}{C}r\}$. We also define $A_{1}^{*}$ and $A_{2}^{*}$ similarly by replacing $\chi_{v}$ with $\chi^{*}_{v^{*}}$. We choose $C$ such that 
\[\PP[A_{1}\cap A_{2}\cap A^{*}_{1}\cap A^{*}_{2}]\ge 1-\epsilon.\]
We choose $r$ such that $Cr<\epsilon$ and we choose $\epsilon'<\frac{1}{C}r$. Note that on the event $A_{1}\cap A_{2}\cap A^{*}_{1}\cap A^{*}_{2}$, we have that $\eta_{\delta}^L\text{ can not hit } \partial B(d,\epsilon)\text{ after hitting }\partial B(d,\epsilon')$. Since $\eta_{\delta}^L$ can only go through the edge $\{v, v^{*}\}$ once. This implies~\eqref{eqn::neard}. 
Together with~\eqref{eqn::equicontinuous}, we complete the proof of tightness. 

Finally, suppose $\eta^L$ is any subsequential limit of $\{\eta_{\delta}^{L}\}_{\delta>0}$, we check that $\PP[\eta^L\cap [bc]=\emptyset]=1$. We have the following two observations.
\begin{itemize}
\item The event $\{\mathbf{trunk}\cap\mathbf{trunk}^*\neq\emptyset\}$ implies $\{\eta^L\cap (bc)\neq\emptyset\}$. Thus $\PP[\eta^L\cap (bc)=\emptyset]=1$. 
\item From~\eqref{eqn::subseqlimit_avoids_c}, we have $\PP[c\not\in\eta^L]=1$. By symmetry, we have $\PP[b\not\in\eta^L]=1$. 
\end{itemize}
In summary, we have $\PP[\eta^L\cap [bc]=\emptyset]=1$ as desired. 
\end{proof}

\begin{remark}
We treat the tightness of the LERW branches (Proposition~\ref{prop::tightness_LERW}) and the Peano curves (Proposition~\ref{prop::tightness}) for UST in topological rectangle in this appendix. Similar tightness argument also holds for the LERW branches and the Peano curves for UST in general topological polygons. These models will be analyzed in subsequent articles~\cite{LiuPeltolaWuUST} and\cite{LiuWuLERW}. 
\end{remark}

\section{Hypergeometric SLE in literature}
\label{appendix_hSLE}
In this section, we compare ``hypergeometric $\SLE$" in the literature~\cite{ZhanReversibilityMore}, ~\cite{QianConformalRestrictionTrichordal} and~\cite{WuHyperSLE}. 
The variant of $\SLE$ whose driving function involves hypergeometric function is first considered in~\cite{ZhanReversibilityMore}. This is a family of curves with two parameters $\kappa$ and $\rho$ and two marked points and the author calls them ``intermediate $\SLE$". In~\cite{ZhanReversibilityMore}, the author considers the case when $\kappa\in(0,4)$ and $\rho\ge \frac{\kappa-4}{2}$ and shows that the time-reversal of $\SLE_\kappa(\rho)$ can be described by intermediate $\SLE_\kappa(\rho)$ with two marked points.

In~\cite{QianConformalRestrictionTrichordal}, the author first uses the terminology ``hypergeometric $\SLE$" and the notation $\hSLE$. The author extends ``intermediate SLE" to more general setup: a family of curves with three parameters $\lambda, \mu$ and $\nu$, which is denoted by $\hSLE_\kappa(\lambda,\mu,\nu)$. The explicit form of driving function will be given below. The parameters of hypergeometric functions are as follows \cite[Eq.(4.2)]{QianConformalRestrictionTrichordal}: 
\begin{equation}\label{eqn::QianConstants}
A:= \lambda + \mu + \nu + 2, \quad B := \mu+ \nu - \lambda, \quad C := 2\nu + 3/2,
\end{equation}
with restrictions \cite[Eq.(4.5)]{QianConformalRestrictionTrichordal}: 
\begin{equation}\label{eqn::QianRestriction}
\lambda>-3/4,\quad \mu\ge -1/4,\quad \nu>-3/4,\quad \mu<\lambda+\nu+3/2.
\end{equation}
The driving function $W$ satisfies the following SDE \cite[Eq.(4.21)]{QianConformalRestrictionTrichordal}: 
\begin{equation}\label{eqn::QianSDE}
\begin{cases}
dW_t = \sqrt{\kappa} dB_t + J( W_t - O_t, W_t - V_t) dt; \\
dO_t = \frac{2dt}{O_t - W_t} dt;\quad dV_t = \frac{2dt}{V_t - W_t} dt.
\end{cases}
\end{equation}
Here $J$ is defined by \cite[Eq.(4.17)]{QianConformalRestrictionTrichordal}:
\begin{equation}\label{eqn::QianJ}
J ( p;q) = J(A,B,C,p;q) = \frac{\kappa}{p} Q \left(A,B,C, \frac{p}{q-p}\right) 
\end{equation}
where (\cite[the equation above Eq.(4.12), the equation above Eq.(4.15)]{QianConformalRestrictionTrichordal}): 
\begin{equation}\label{eqn::QianQG}
Q (A,B,C,z) : = z G'(z)/ G(z), \quad G(z) = z^\nu (1+z)^\mu w(A, B, C; - z),
\end{equation}
and $w(A, B, C;z)$ is an analytic extension of $\hF(A, B, C;z)$. 
In~\cite{QianConformalRestrictionTrichordal}, the author mainly considers the case $\kappa=8/3$. This value is special since it is related to the conformal restriction measure~\cite{LawlerSchrammWernerConformalRestriction}.

In~\cite{WuHyperSLE}, the author uses the same terminology ``hypergeometric $\SLE$" and the same notation $\hSLE$.  In this paper, she extends the definition of the ``intermediate $\SLE$" in another direction, compared with~\cite{QianConformalRestrictionTrichordal}. This is still a family of curves with two parameters $\kappa$ and $\rho$ and two marked points, which is called $\hSLE_\kappa(\rho-2)$ (to avoid the ambiguity, we replace $\nu$ there by $\rho-2$). The author in~\cite{WuHyperSLE} extends it to the case when $\kappa\in (0,8)$ and $\rho\in\R$: 
\begin{itemize}
\item When $\kappa\in(0,8)$ and $\rho> -2\vee (\kappa/2-4)$, the process is defined in~\cite[Eq.(3.2), Eq.(3.4)-(3.7)]{WuHyperSLE}. In this case, the process is defined in the same way as ~\eqref{eqn::hyper_def}-\eqref{eqn::hypersle_sde} in this article. In particular, when $\kappa\in (0,4)$ and $\rho>-2$, it is the same as the intermediate $\SLE_{\kappa}(\rho)$~\cite{ZhanReversibilityMore}.
\item When $\kappa\in (0,8)$ and $\rho\le -2\vee (\kappa/2-4)$, the process is defined in~\cite[Eq.(3.3), Eq.(3.4)-(3.7)]{WuHyperSLE}. 
\end{itemize}
In our Section~\ref{sec::hypersle}, we further extend this definition to the case when $\kappa\ge 8$ and $\rho>\kappa/2-4$. 
In~\cite{WuHyperSLE}, the author relates $\hSLE_\kappa(\rho-2)$ to the scaling limit of interfaces of discrete models in quad, see~\cite[Proposition~1.6]{WuHyperSLE} for Ising model and~\cite{KemppainenSmirnovFKIsingHyperSLE} for FK-Ising model. In our Section~\ref{sec::ust}, we relates $\hSLE_8$ to the scaling limit of the Peano curve in UST in quad. 

Now, let us compare the two notions of ``hypergeometric $\SLE$" in~\cite{QianConformalRestrictionTrichordal} and~\cite{WuHyperSLE}. The “hypergeometric SLE” defined in~\cite{WuHyperSLE} is a subfamily of the “hypergeometric SLE” defined in~\cite{QianConformalRestrictionTrichordal}) (as cited
above~\eqref{eqn::QianConstants}-\eqref{eqn::QianQG}) only when $\kappa=8/3$ and $\rho>-2$: the process $\hSLE_{8/3}(\rho-2)$ defined in~\eqref{eqn::hyper_def}-\eqref{eqn::hypersle_sde} is the same as $\hSLE_{8/3}(-2, 3\rho/8, 3\rho/8)$ defined in~\eqref{eqn::QianConstants}-\eqref{eqn::QianQG} (the restrictions~\eqref{eqn::QianRestriction} needs to be relaxed to include the case when $\lambda=-2$). For other $\kappa$’s, they are two different families of processes. 

In order to get a more general relation, one has to change the constants in~\eqref{eqn::QianConstants} to the following ones:
\begin{equation}\label{eqn::correction}
A=\lambda+\mu+\nu+8/\kappa-1, \quad B=\mu+\nu-\lambda, \quad C=2\nu+4/\kappa,
\end{equation}
and relaxes the restrictions in~\eqref{eqn::QianRestriction}.
\begin{lemma}
When $\kappa>0$ and $\rho>(-2)\vee (\kappa/2-4)$, the process $\hSLE_\kappa(\rho-2)$ defined in~\eqref{eqn::hyper_def}-\eqref{eqn::hypersle_sde} is a special case of the process $\hSLE_{\kappa}(\lambda, \mu, \nu)$ defined in~\eqref{eqn::QianSDE}-\eqref{eqn::correction} for
\begin{equation}\label{eqn::generalpara}
\mu=\nu=\frac{\rho}{\kappa}\quad\text{and}\quad \lambda=1-\frac{8}{\kappa}.
\end{equation}
\end{lemma} 
\begin{proof}
In~\cite{QianConformalRestrictionTrichordal}, the author considers the case $V_0<O_0=0\le W_0$. In this case, $z=p/(q-p)<0$, we use the following relation~\cite[Eq.15.3.4]{AbramowitzHandbook}
\[
w(A, B, C; -z) = (1+z)^{-A} w\left( A, C- B, C ; \frac{z}{z+1}\right)
\]
and write $\tilde w(\cdot) = w( A, C-B, C;\cdot)$. By the choice of $A, B, C$ in~\eqref{eqn::correction}, we have 
\[
\tilde w(\cdot) =\hF(\lambda + \mu + \nu +8/\kappa-1; \nu + \lambda +4/\kappa - \mu; 2\nu+4/\kappa; \cdot).
\]
Setting $p = W_t - O_t$ and $q= W_t- V_t$, 
we find
\[
J (W_t - O_t,W_t- V_t) = \frac{\kappa \nu }{W_t - O_t}  + \frac{\kappa( \mu -A)}{W_t - V_t}  - \frac{\kappa (O_t- V_t)}{(W_t - V_t)^2} \frac{\tilde w' ( (W_t - O_t)/(W_t - V_t))}{\tilde w( (W_t - O_t)/ (W_t - V_t))}. 
\]
To compare with~\eqref{eqn::hsle_sde}, we reflect the $\hSLE_\kappa(\mu,\nu,\kappa)$ with respect to $\ii\R$,  the drift term of the resulting driving function becomes
\begin{equation*}
d L_t:=\frac{\kappa \nu }{W_t - O_t}dt  -\frac{\kappa(A-\mu)}{W_t - V_t}dt-\frac{\kappa (V_t- O_t)}{(W_t - V_t)^2} \frac{\tilde w' ( (W_t - O_t)/(W_t - V_t)))}{\tilde w( (W_t - O_t)/ (W_t - V_t))}dt. 
\end{equation*}
Compared with~\eqref{eqn::hsle_sde}, we see that $\hSLE_\kappa(\rho-2)$ defined in~\eqref{eqn::hsle_sde} is a special case of the process defined in~\eqref{eqn::QianSDE}-\eqref{eqn::correction} with parameters~\eqref{eqn::generalpara}.
\end{proof}


\begin{thebibliography}{CDCH{\etalchar{+}}14}

\bibitem[Ahl78]{AhlforsComplexAnalysis}
Lars~V. Ahlfors.
\newblock Complex analysis.
\newblock {\em McGraw-Hill Book Co., New York, third edition}, 1978.

\bibitem[AS92]{AbramowitzHandbook}
Milton Abramowitz and Irene~A. Stegun, editors.
\newblock Handbook of mathematical functions with formulas, graphs, and
  mathematical tables.
\newblock {\em Dover Publications, Inc., New York}, 1992.

\bibitem[BPW21]{BeffaraPeltolaWuUniqueness}
Vincent Beffara, Eveliina Peltola, and Hao Wu.
\newblock On the uniqueness of global multiple {SLE}s. 
\newblock {\em Ann. Probab. } 49(1), 400-434, 2021. 

\bibitem[Che16]{ChelkakRobustComplexAnalysis}
Chelkak Dmitry.
\newblock Robust discrete complex analysis: A toolbox. 
\newblock {\em Ann. Probab. } 44(1), 628-683, 2016. 


\bibitem[CDCH{\etalchar{+}}14]{CDCHKSConvergenceIsingSLE}
Dmitry Chelkak, Hugo Duminil-Copin, Cl\'{e}ment Hongler, Antti Kemppainen, and
  Stanislav Smirnov.
\newblock Convergence of {I}sing interfaces to {S}chramm's {SLE} curves.
\newblock {\em C. R. Math. Acad. Sci. Paris}, 352(2):157-161, 2014.


\bibitem[CW21]{ChelkakWanMassiveLERW}
Dmitry Chelkak and Yijun Wan.
\newblock {On the convergence of massive loop-erased random walks to massive
  SLE(2) curves}.
\newblock {\em Electron. J. Probab.}, 26:54, 2021.

\bibitem[CS11]{ChelkakSmirnovDiscreteComplexAnalysis}
Chelkak Dmitry and Smirnov Stanislav.
\newblock {Discrete complex analysis on isoradial graphs}.
\newblock {\em Adv. Math.}, 228:1590-1630, 2011.

\bibitem[DC13]{DCParafermionic}
Hugo Duminil-Copin.
\newblock Parafermionic observables and their applications to planar
  statistical physics models.
\newblock {\em Ensaios Matematicos}, 25:1-371, 2013.

  
\bibitem[Dub06]{DubedatEulerIntegralsCommutingSLEs}
Julien Dub\'edat.
\newblock Euler integrals for commuting {SLE}s.
\newblock {\em J. Stat. Phys.}, 123(6):1183-1218, 2006.

\bibitem[Izy15]{IzyurovObservableFree}
Konstantin Izyurov.
\newblock Smirnov's observable for free boundary conditions, interfaces and
  crossing probabilities.
\newblock {\em Comm. Math. Phys.}, 337(1):225-252, 2015.

\bibitem[Kar19]{KarrilaMultipleSLELocalGlobal}
Alex Karrila.
\newblock Multiple {SLE} type scaling limits: from local to global.
\newblock {\em arXiv:1903.10354}, 2019.

\bibitem[Kar20]{KarrilaUSTBranches}
Alex Karrila.
\newblock U{ST} branches, martingales, and multiple {$\rm SLE(2)$}.
\newblock {\em Electron. J. Probab.}, 25:83, 37, 2020.

\bibitem[KS16]{KemppainenSmirnovRandomCurves}
Antti Kemppainen and Stanislav Smirnov.
\newblock Random curves, scaling limits and Loewner evolutions.
\newblock {\em Ann. Probab.}, 45(2):698-779, 2016.

\bibitem[KS18]{KemppainenSmirnovFKIsingHyperSLE}
Antti Kemppainen and Stanislav Smirnov.
\newblock Configurations of {FK} {I}sing interfaces and hypergeometric {SLE}.
\newblock {\em Math. Res. Lett.}, 25(3):875-889, 2018.

\bibitem[KW11]{KenyonWilsonBoundaryPartitionsTreesDimers}
Richard~W. Kenyon and David~B. Wilson.
\newblock Boundary partitions in trees and dimers.
\newblock {\em Trans. Amer. Math. Soc.}, 363(3):1325-1364, 2011.

	
\bibitem[Law05]{Lawler-book}
Gregory~F. Lawler.
\newblock Conformally invariant processes in the plane. Vol.114 of {\em Mathematical Surveys and Monographs}.
\newblock American Mathematical Society, Providence, RI, 2005


\bibitem[LPW21]{LiuPeltolaWuUST}
Mingchang Liu, Eveliina Peltola, and Hao Wu.
\newblock Uniform spanning tree in topological polygon, partition functions for
  {SLE}(8), and correlations in $c=-2$ logarithm {CFT}. 
 \newblock {\em arXiv:2108.04421}, 2021.

\bibitem[LSW99]{LawlerWernerIntersectionExponents}
Gregory~F. Lawler, Oded Schramm, and Wendelin Werner.
\newblock Intersection exponent for planar Brownian motion.
\newblock {\em Ann. Probab.}, 27(4):1601-1642, 1999.


\bibitem[LSW03]{LawlerSchrammWernerConformalRestriction}
Gregory~F. Lawler, Oded Schramm, and Wendelin Werner.
\newblock Conformal restriction: the chordal case.
\newblock {\em J. Amer. Math. Soc.}, 16(4):917-955, 2003.

\bibitem[LSW04]{LawlerSchrammWernerLERWUST}
Gregory~F. Lawler, Oded Schramm, and Wendelin Werner.
\newblock Conformal invariance of planar loop-erased random walks and uniform
  spanning trees.
\newblock {\em Ann. Probab.}, 32(1B):939-995, 2004.

\bibitem[LW23a]{LiuWuLERW}
Mingchang Liu and Hao Wu.
Loop-erased random walk branch of uniform spanning tree in
  topological polygons.
 {\em Bernoulli}, 29(2):1555--1577, 2023.
 
 \bibitem[LW23b]{LupuWuLevellineGFF}
Titus Lupu and Hao Wu.
\newblock A level line of the Gaussian free field with measure-valued boundary
  conditions. 
  \newblock {\em arXiv:2106.15169, to appear in Sci. China Math.}, 2023.


\bibitem[MS16a]{MillerSheffieldIG1}
Jason Miller and Scott Sheffield.
\newblock Imaginary geometry {I}: {I}nteracting {SLE}s.
\newblock {\em Probab. Theory Related Fields}, 164(3-4):553-705, 2016.

\bibitem[MS16b]{MillerSheffieldIG2}
Jason Miller and Scott Sheffield.
\newblock Imaginary geometry {II}: {R}eversibility of SLE$_{\kappa}(\rho_1;\rho_2)$ for $\kappa\in(0,4)$.
\newblock {\em Ann. Probab.}, 44(3):1647-1722, 2016.

\bibitem[MS16c]{MillerSheffieldIG3}
Jason Miller and Scott Sheffield.
\newblock Imaginary geometry {III}: {R}eversibility of SLE$_{\kappa}$
  for $\kappa\in(4,8)$.
\newblock {\em Ann. of Math. (2)}, 184(2):455-486, 2016.

\bibitem[MS17]{MillerSheffieldIG4}
Jason Miller and Scott Sheffield.
\newblock Imaginary geometry {IV}: interior rays, whole-plane reversibility,
  and space-filling trees.
\newblock {\em Probab. Theory Related Fields}, 169(3-4):729-869, 2017.

\bibitem[Pem91]{PemantleSpanningTree}
Robin Pemantle.
\newblock Choosing a spanning tree for the integer lattice uniformly.
\newblock {\em Ann. Probab.}, 19(4):1559-1574, 1991.


\bibitem[Pom92]{Pommerenke}
Ch. Pommerenke.
\newblock Boundary behaviour of conformal maps. Vol. 299 of {\em
  Grundlehren der Mathematischen Wissenschaften},
\newblock Springer-Verlag, Berlin, 1992.

\bibitem[RS05]{RohdeSchrammSLEBasicProperty}
Steffen Rohde and Oded Schramm.
\newblock Basic properties of {SLE}.
\newblock {\em Ann. of Math. (2)}, 161(2):883-924, 2005.

\bibitem[Qia18]{QianConformalRestrictionTrichordal}
Wei Qian.
\newblock Conformal restriction: the trichordal case.
\newblock {\em Probab. Theory Related Fields}, 171(3-4):709-774, 2018.


\bibitem[Sch00]{SchrammScalinglimitsLERWUST}
Oded Schramm.
\newblock Scaling limits of loop-erased random walks and uniform spanning
  trees.
\newblock {\em Israel J. Math.}, 118:221-288, 2000.

\bibitem[KS18]{SmirnovPercolationConformalInvariance}
Antti Kemppainen, Stanislav Smirnov.
\newblock Configurations of {FK} {I}sing interfaces and hypergeometric
              {SLE}.
\newblock {\em Math. Res. Lett.}, 333(3):875-889, 2018.

\bibitem[Smi01]{SmirnovPercolationConformalInvariance}
Stanislav Smirnov.
\newblock Critical percolation in the plane: conformal invariance, {C}ardy's
  formula, scaling limits.
\newblock {\em C. R. Acad. Sci. Paris S\'er. I Math.}, 333(3):239-244, 2001.


\bibitem[SS09]{SchrammSheffieldDiscreteGFF}
Oded Schramm and Scott Sheffield.
\newblock Contour lines of the two-dimensional discrete {G}aussian free field.
\newblock {\em Acta Math.}, 202(1):21-137, 2009.

\bibitem[SW05]{SchrammWilsonSLECoordinatechanges}
Oded Schramm and David~B. Wilson.
\newblock S{LE} coordinate changes.
\newblock {\em New York J. Math.}, 11:659-669, 2005.


\bibitem[Wu20]{WuHyperSLE}
Hao Wu.
\newblock Hypergeometric {SLE}: conformal {M}arkov characterization and
  applications.
\newblock {\em Comm. Math. Phys.}, 374(2):433-484, 2020.

\bibitem[Zha08a]{ZhanDuality}
Dapeng Zhan.
\newblock Duality of chordal {SLE}.
\newblock {\em Invent. Math.}, 174(2):309-353, 2008.

\bibitem[Zha08b]{ZhanReversibility}
Dapeng Zhan.
\newblock Reversibility of chordal {SLE}.
\newblock {\em Ann. Probab.}, 36(4):1472-1494, 2008.

\bibitem[Zha08c]{ZhanLERW}
Dapeng Zhan.
\newblock The scaling limits of planar {LERW} in finitely connected domains.
\newblock {\em Ann. Probab.}, 36(2):467-529, 2008.

\bibitem[Zha10]{ZhanReversibilityMore}
Dapeng Zhan.
\newblock Reversibility of some chordal {${\rm SLE}(\kappa;\rho)$} traces.
\newblock {\em J. Stat. Phys.}, 139(6):1013-1032, 2010.

\end{thebibliography}

{\small
\newcommand{\etalchar}[1]{$^{#1}$}

}

\end{document}